\documentclass[juq]{siamonline250211}

\usepackage{cmap}
\usepackage[T1]{fontenc}
\usepackage[utf8]{inputenc}
\usepackage{lmodern}
\usepackage{microtype}

\usepackage{amsfonts,amssymb,mathtools}
\usepackage{algpseudocode}
\mathtoolsset{centercolon}

\newcommand{\nc}{\normalcolor}

\definecolor{dkblue}{rgb}{0,0,.5}
\definecolor{rust}{rgb}{0.5,0.1,0.1}
\hypersetup{
  colorlinks=true,
  linkcolor=dkblue,
  citecolor=rust,
  urlcolor=rust
}

\newsiamthm{assumption}{Assumption}
\newsiamremark{remark}{Remark}

\theoremstyle{nonumberplain}
\theoremheaderfont{\normalfont\itshape}
\theorembodyfont{\normalfont}
\theoremseparator{.}
\theoremsymbol{}

\newcommand{\R}{\mathbb R}
\newcommand{\Splus}{\mathbb S_+}
\newcommand{\Spp}{\mathbb S_{++}}
\newcommand{\one}{\mathbf 1}

\newcommand{\TheTitle}{Time-Uniform Accuracy of Ensemble Kalman Filters With  Localization}
\newcommand{\TheAuthors}{X. Cheng, D. Sanz-Alonso, and N. Waniorek}

\headers{Time-Uniform Accuracy of Ensemble Kalman Filters}{\TheAuthors}

\title{\TheTitle}
\author{
  Xiaoou Cheng\thanks{University of Chicago, Chicago, Illinois, USA.}
  \and
  Daniel Sanz-Alonso\footnotemark[1]
  \and
  Nathan Waniorek\thanks{New York University, New York, New York, USA.}
}

\ifpdf
\hypersetup{
  pdftitle={\TheTitle},
  pdfauthor={\TheAuthors}
}
\fi

\begin{document}

\maketitle

\begin{abstract}
This paper establishes time-uniform accuracy guarantees for deterministic square-root ensemble Kalman filters in linear--Gaussian state-space models under perfect-model dynamics. For hyperbolic, detectable systems, we prove that the ensemble means and covariances approximate their Kalman filter counterparts uniformly in time, with high probability and without inflation or resampling. The required ensemble size depends on the effective rank of the initial covariance and the dimension of the unstable subspace, rather than on the ambient state dimension. We also analyze a localized square-root ensemble Kalman filter for weakly coupled spatial systems. We show that the stabilizing Kalman covariance inherits spatial decay from the dynamics and derive accuracy bounds that separate sampling error from localization bias. For the localized filter, the required ensemble size depends on the local intrinsic dimension and only logarithmically on the number of spatial blocks, while the localization bias scales with the interaction strength. Our analysis combines stability and forgetting for possibly singular Riccati recursions, nonasymptotic covariance concentration, and perturbation estimates for localized covariance dynamics.
\end{abstract}



\section{Introduction}
\label{sec:introduction}
 Filtering aims to characterize the conditional distribution of an evolving state given noisy and incomplete observations \cite{Jazwinski1970,AschBocquetNodet2016,ReichCotter2015,SanzAlonsoStuartTaeb2023}. This conditional law, known as the filtering distribution, describes both the state estimate and its uncertainty. In linear-Gaussian models, the filtering distribution is Gaussian, and the Kalman filter computes it explicitly by recursively updating its conditional mean and covariance \cite{Kalman1960}. In high-dimensional problems, however, propagating these
quantities directly can be computationally prohibitive. Ensemble Kalman
filters (EnKFs) address this difficulty by representing the filtering
distribution through an ensemble of state realizations and estimating its mean and covariance empirically; see, e.g.
\cite{Evensen1994,EvensenVanLeeuwen1996,HoutekamerMitchell1998,Evensen2009}.

The practical usefulness of an EnKF often depends on whether a moderate
ensemble can remain accurate over long time horizons. This question is not
settled by finite-time convergence as the ensemble size tends to infinity:
sampling errors may be repeatedly propagated by unstable dynamics and may
accumulate over time. The difficulty is especially pronounced when the
ensemble size is much smaller than the state dimension, a setting that often arises in geophysical applications where the state dimension can be of order $10^9$ while the ensemble size is of order $10^2$. 
In spatially extended
systems, localization seeks to overcome this limitation by replacing global
covariance estimation with a collection of local estimation problems. This
reduces sampling error, but introduces a bias by neglecting long-range
dependence.

This paper gives time-uniform accuracy guarantees for square-root EnKFs, with
and without localization, in linear--Gaussian models with perfect-model
dynamics. Without localization, we show that detectability prevents sampling
errors from accumulating along dynamically persistent directions. The
required ensemble size is governed by the effective rank of the initial
covariance and the dimension of the unstable subspace, rather than by the
ambient state dimension. For weakly coupled spatial systems, we show that
localization replaces these global dimensions by local dimensions,
at the cost of a bias controlled by the interaction strength. 

\subsection{Contributions and overview}
Our main contributions are as follows.

\paragraph{Time-uniform accuracy without localization}
We first consider a hyperbolic, detectable linear system and a deterministic
square-root EnKF initialized by independent samples from the initial Gaussian
law. We prove high-probability bounds, uniform over all times, for the errors
in the ensemble means and covariances relative to their Kalman filter
counterparts. The required ensemble size is of order
\[
  r_{\mathrm{eff}}(\widehat P_0)+r_u+\log(1/\delta),
\]
where \(r_{\mathrm{eff}}(\widehat P_0)\) is the effective rank of the initial
covariance, \(r_u\) is the dimension of the unstable subspace, and
\(1-\delta\) is the confidence level. In particular, the ensemble size need
not exceed either the state or the observation dimension. The results require neither covariance
inflation nor resampling.

\paragraph{Time-uniform accuracy with localization}
We next consider spatial systems formed by weakly coupled blocks. The
interactions between blocks have small amplitude and decay polynomially with
distance, while the individual blocks may contain unstable dynamics. We first
show that the stabilizing Kalman covariance and gain inherit the spatial decay
of the dynamics. This provides a structural justification for a localized
square-root EnKF that performs the analysis blockwise while retaining the full
coupled dynamics in the forecast.

For this localized filter, we derive time-uniform bounds for the local
covariance blocks and ensemble means. The covariance error separates into a
sampling term and an \(O(\varepsilon)\) localization bias, where
\(\varepsilon\) measures the strength of the cross-block interactions. The
required ensemble size is of order
\[
  r_{\mathrm{eff},\mathrm{loc}}+\log(J/\delta),
\]
where \(r_{\mathrm{eff},\mathrm{loc}}\) is the largest blockwise effective rank
and \(J\) is the number of spatial blocks. Thus localization can reduce a
sampling requirement that grows proportionally to \(J\) to one that grows
only logarithmically in \(J\).

\paragraph{Proof strategy}
Our analysis exploits the exact correspondence between square-root covariance
updates and the Kalman Riccati recursion. In the nonlocalized setting, we
combine effective-rank covariance concentration with stability and forgetting
estimates for Riccati recursions that may remain singular along stable
directions. In the localized setting, we supplement these arguments with
block-decay estimates and perturbation bounds comparing the coupled Riccati
flow with its block-diagonal counterpart.

\subsection{Related work} 

\paragraph{Finite-time convergence and time-uniform state estimation} Early analyses established convergence of the EnKF in the large-ensemble limit over finite time horizons. In linear--Gaussian models, \(L^p\) convergence to the Kalman filter was proved for perfect-model dynamics in \cite{MandelCobbBeezley2011} and for dynamics with Gaussian noise in \cite{LeGlandMonbetTran2011}; analogous results hold for square-root \cite{KwiatkowskiMandel2015} and continuous-time \cite{LangeStannat2021ContinuousSquareRoot} filters. Nonasymptotic,
effective-dimension bounds over multiple assimilation cycles were established
in \cite{AlGhattasBaoSanzAlonso2024} for an EnKF with an additional
resampling step that breaks the dependence among ensemble members. Subsequent work studied time-uniform stability and state-estimation accuracy. For potentially nonlinear dynamics, \cite{KellyLawStuart2014} proves long-time stability with additive covariance inflation and state-estimation accuracy under small observation noise and complete observations. Related continuous-time results with multiplicative inflation appear in \cite{deWiljesReichStannat2018,TakedaSakajo2024}, again under complete observations. With partial observations, stability and ergodicity were established in \cite{TongMajdaKelly2016,TongMajdaKelly2016Inflation}, and time-uniform state-estimation accuracy was proved for linear--Gaussian models in \cite{MajdaTong2018}, extended to nonlinear dynamics in \cite{SanzAlonsoWaniorek2025}, and obtained in continuous time for the two-dimensional Navier--Stokes equations in \cite{biswas2024unified}. These results require covariance inflation and do not establish convergence of the EnKF to  the filtering  distribution.  
\paragraph{Time-uniform convergence} For the ensemble Kalman--Bucy filter, time-uniform convergence to the Kalman--Bucy filter was first proved under stable  linear  signal dynamics in \cite{DelMoralTugaut2018} 
 and subsequently for possibly unstable linear
signal dynamics under broader detectability and observability conditions in
\cite{bishop2019stability,bishop2020perturbation,BishopDelMoral2023}.
Fewer results are available in discrete time, where time-varying Riccati recursions present additional difficulties; see \cite{del2022note}. Recently, \cite{DelMoralNasriRemillard2026} established time-uniform convergence under detectability and controllability, with an ensemble-size requirement exceeding the state dimension. Our results instead require an ensemble size governed by effective dimensions that may be much smaller than the ambient state dimension. \paragraph{Localization} Localization is essential to operational EnKF implementations \cite{houtekamer2005ensemble,greybush2011balance} and generally takes one of two forms. \emph{Covariance tapering} replaces the empirical covariance by its Schur product with a compactly supported correlation function. Introduced in \cite{houtekamer2001sequential,hamill2001distance}, this approach was adapted to square-root filters in \cite{sakov2008deterministic,BergemannReich2010}. A common choice is the Gaspari--Cohn function \cite{gaspari1999construction}, which can empirically outperform more sophisticated tapering schemes \cite{gilpin2025numerical}. \emph{Domain localization} instead performs the analysis through local, possibly overlapping, updates. It originated in \cite{HoutekamerMitchell1998,OttEtAl2004} and was implemented efficiently in the Local Ensemble Transform Kalman Filter \cite{HuntKostelichSzunyogh2007}. In this paper, we study a domain-localized square-root EnKF. 

Rigorous localization theory remains limited. For the ensemble Kalman--Bucy filter, \cite{deWiljesTong2020} proves finite-time well-posedness, stability, and state-estimation accuracy for tapering-based localization under short-range nonlinear interactions and complete observations. In discrete time, \cite{MajdaTong2018} establishes time-uniform state-estimation accuracy for a tapering-based EnKF under banded dynamics, with ensemble size depending only logarithmically on the state dimension, while assuming stability of the localized forecast covariances. Neither work studies convergence of a localized EnKF to the Kalman filter or domain localization. Nonasymptotic comparisons between localized and mean-field EnKFs were obtained in \cite{AlGhattasChenSanzAlonsoWaniorek2025,AlGhattasSanzAlonso2024}, but only for a single time step. \paragraph{Riccati theory and covariance estimation} Our analysis builds on classical Kalman and Riccati stability theory \cite{AndersonMoore1979}, including contraction of the Riccati map on the positive-definite cone \cite{Bougerol1993} and forgetting of initial conditions \cite{OconePardoux1996}. Its second main ingredient is dimension-independent covariance concentration. Sample-covariance concentration in terms of spectral decay was characterized in \cite{KoltchinskiiLounici2017}; related tapering and thresholding estimators exploit covariance sparsity \cite{BickelLevina2008,cai2010optimal}. Extensions to infinite-dimensional settings appear in \cite{AlGhattasChenSanzAlonsoWaniorek2024, AlGhattasChenSanzAlonsoWaniorek2025}. \paragraph{Beyond linear--Gaussian models} For linear dynamics and observations with a small nonlinear perturbation, \cite{CalvelloEtAl2026} proves convergence of the EnKF to the true filtering distributions up to a bias proportional to the nonlinearities. Under general Lipschitz conditions, \cite{JorgensenBaptistaHoffmannMarzouk2026} proves convergence of the more general transport ensemble filter of \cite{SpantiniBaptistaMarzouk2022}. Neither result is uniform in time. For detectable linear--Gaussian dynamics with non-Gaussian initial conditions, \cite{hoffmann2026large} proves geometric-in-time convergence of the mean-field EnKF to the true filtering distribution. We refer to \cite{CalvelloReichStuart2025} for a survey of mean-field EnKF theory.

\subsection{Outline}
Sections~\ref{sec:problem-setting} and
\ref{sec:localized-model-algorithm} present the nonlocalized and localized
filters in parallel: each section introduces the model and assumptions, states
the algorithm, and gives the corresponding time-uniform accuracy results.
Section~\ref{sec:proofs} contains the proofs, with technical arguments
deferred to the supplementary material.

\subsection{Notation}
 $\mathbb S^{d}_+$ denotes the set of $d\times d$ symmetric, positive-semidefinite matrices, and $\mathbb S^{d}_{++}$ denotes the set of $d\times d$ symmetric, positive-definite matrices. 
 For $A \in \mathbb S^{d}_+$, its principal square root is denoted by $A^{1/2}$.  
 For symmetric matrices $A,B\in \mathbb R^{d\times d}$, we write $A\succ B$ if $A-B$ is positive definite, and $A \succeq B$ if $A-B$ is postive semidefinite. For a vector $x\in \R^d$, we denote by $\|x\|$ the Euclidean vector norm. For a matrix $A\in \mathbb{R}^{d_1\times d_2},$ we abuse notation slightly and denote by $\|A\|$ the Euclidean operator norm.

\section{The ensemble Kalman filter}
\label{sec:problem-setting}

\subsection{Model, filtering task, and assumptions}
\label{subsec:setting-assumptions}
We consider the following discrete-time linear--Gaussian state--space model:
\begin{align}
  u_0
  &\sim \mathcal N(\widehat m_0,\widehat P_0),
  &&\text{(initial state law)},
  \label{eq:initial-state-law}
  \\
  u_{t+1}
  &=Au_t,
  \qquad t=0, 1, \ldots,
  &&\text{(dynamics model)},
  \label{eq:dynamics-model}
  \\
  y_t
  &=Hu_t+\eta_t,
  \qquad
  \eta_t\stackrel{\mathrm{i.i.d.}}{\sim}\mathcal N(0,R),
  \quad t=0,1,\ldots,
  &&\text{(observation model)}.
  \label{eq:observation-model}
\end{align}

We refer to $u_t\in\R^{d_u}$ as the hidden state and to
$y_t\in\R^{d_y}$ as the observation at time $t$. 
The Gaussian initial law is specified by its mean
$\widehat m_0\in\R^{d_u}$ and covariance
$\widehat P_0\in\mathbb S^{d_u}_+$. The dynamics are deterministic. 
The state--transition matrix $A\in\R^{d_u\times d_u}$, observation operator $H\in\R^{d_y\times d_u}$, and observation-noise covariance $R\in\mathbb S^{d_y}_{++}$ are given.
The initial state $u_0$ and observation noises $(\eta_t)_{t=0}^\infty$ are mutually independent.

For $t \ge 0$, let $\mathcal Y_t:=\sigma(y_0,\ldots,y_t)$, and let $\mathcal Y_{-1}$ be
the trivial sigma-algebra. 
Filtering seeks to determine, sequentially in time, the conditional distribution of the state given the observations collected so far. At time $t$, $\mathcal L(u_t\mid\mathcal Y_{t-1})$ is the \emph{forecast} distribution, whereas $\mathcal L(u_t\mid\mathcal Y_t)$ is the filtering, or \emph{analysis}, distribution obtained after incorporating $y_t$.
 Under the linear-Gaussian model, both distributions are Gaussian:
\begin{equation}
  \mathcal L(u_t\mid\mathcal Y_{t-1})
  =\mathcal N(\widehat m_t,\widehat P_t),
  \qquad
  \mathcal L(u_t\mid\mathcal Y_t)
  =\mathcal N(m_t,P_t).
  \label{eq:forecast-filtering-laws}
\end{equation}
Thus $(\widehat m_t,\widehat P_t)$ are the forecast mean and covariance, while $(m_t,P_t)$ are the analysis counterparts. The exact moments are computed by the Kalman filter. Our goal is to determine whether the deterministic square-root EnKF introduced in Section~\ref{subsec:EnKF} accurately approximates these moments uniformly in time with a moderate ensemble size. 

Time-uniform control requires that no
dynamically persistent state direction remain invisible to
the observations. The control-theoretic notion of detectability encodes this requirement: it permits unobserved directions, but only if they
decay asymptotically. The pair $(A,H)$ is \emph{detectable} if, for all $v \in \R^{d_u},$
\[
  HA^tv=0\ \text{for all }t\ge0
  \quad\Longrightarrow\quad
  A^tv\longrightarrow0.
\]

For a hyperbolic real matrix $A$, let $\mathsf U$ and $\mathsf S$ denote its
real unstable and stable spectral subspaces, generated by the generalized
eigenspaces associated with eigenvalues of modulus greater and less than one,
respectively. Then
\[
  \R^{d_u}=\mathsf U\oplus\mathsf S.
\]
We denote by $r_u:=\dim(\mathsf U)$ the dimension of the unstable subspace. 

\begin{assumption}[Dynamics and initialization]
\label{ass:dynamics-initialization}
The matrix $A$ is hyperbolic,  i.e., 
\begin{equation}
  \operatorname{spec}(A)\cap\{z\in\mathbb C:|z|=1\}=\varnothing,
  \label{eq:hyperbolicity}
\end{equation}
and the pair $(A,H)$ is detectable. We assume that the unstable spectral
subspace is nontrivial, so $r_u\geq1$. In fixed coordinates adapted to the
spectral splitting,
\begin{equation}
  A=\begin{pmatrix}A_u&0\\0&A_s\end{pmatrix},
  \qquad
  H=\begin{pmatrix}H_u&H_s\end{pmatrix},
  \label{eq:block-A-H}
\end{equation}
with spectral radii satisfying
\[
  \rho(A_s)<1,
  \qquad
  \rho(A_u^{-1})<1.
\]
The initial covariance is nondegenerate on the unstable subspace: $(\widehat P_0)_{uu}\succ0.$ In particular, $\widehat P_0\neq0$.
\end{assumption}

Hyperbolicity excludes neutral modes and gives a clean stable--unstable
splitting. In this setting, detectability of $(A,H)$ is equivalent to
observability of the unstable restriction $(A_u,H_u)$; see \cite[Chapter~4, Problem~4.1]{AndersonMoore1979}. Equivalently, there are
an integer $L\ge1$ and a constant $c_o>0$ such that
\begin{equation}
\label{eq:observability-unstable}
  \sum_{\ell=0}^{L-1}
  \|R^{-1/2}H_uA_u^{\ell}v\|^2
  \geq c_o\|v\|^2,
  \qquad v\in\mathsf U.
\end{equation}
Thus the observations  can  resolve the unstable directions, while stable
unobserved directions decay on their own. The condition
$(\widehat P_0)_{uu}\succ0$ ensures that every unstable direction is represented
at initialization, but $\widehat P_0$ need not have full rank. The assumption
$r_u\ge1$ focuses the results
on the nontrivial case with at least one unstable direction. All estimates
obtained in coordinates adapted to the stable--unstable decomposition transfer to the original Euclidean norm
by norm equivalence.

\subsection{Deterministic square-root EnKF}\label{subsec:EnKF}

To describe both the exact Kalman filter and its ensemble approximation, for
$P\in\Splus^{d_u}$ define 
\begin{align}
  \mathsf K(P)
  &:=PH^\top\bigl(HPH^\top+R\bigr)^{-1},
  &&\text{(gain map)},
  \label{eq:gain-map}
  \\
  \Psi(P)
  &:=P-\mathsf K(P)HP,
  &&\text{(covariance-analysis map)},
  \label{eq:analysis-map}
  \\
  \Phi(P)
  &:=A\Psi(P)A^\top,
  &&\text{(forecast-to-forecast Riccati map)}.
  \label{eq:riccati-map}
\end{align}
Since $R\in\Spp^{d_y}$, these maps are well defined for every
$P\in\Splus^{d_u}$. 
The exact Kalman filter propagates the forecast and analysis moments as
follows. Given the exact forecast mean and covariance
$(\widehat m_t,\widehat P_t)$, the analysis step is
\begin{align}
m_t
  &= \widehat m_t
     + K_t\bigl(y_t-H\widehat m_t\bigr), \qquad K_t
  = \mathsf K(\widehat P_t), \label{eq:KFmeananalysis} \\
P_t
  &= \Psi(\widehat P_t). \label{eq:KFcovanalysis}
\end{align}
The subsequent forecast step is
\begin{equation}\label{eq:KFforecast}
\widehat m_{t+1}
  = A m_t, \qquad 
\widehat P_{t+1}
  = A P_t A^\top
   = \Phi(\widehat P_t).
\end{equation}

To approximate these exact Kalman recursions, fix an ensemble size
$N\ge 2$ and initialize the forecast ensemble by independent samples
from the initial Gaussian law:
\begin{equation}\label{eq:gaussian-initial-ensemble}
 \widehat u_0^{(1)},\ldots,\widehat u_0^{(N)}
\stackrel{\mathrm{i.i.d.}}{\sim}
\mathcal N(\widehat m_0,\widehat P_0).   
\end{equation}
Thus the exact and ensemble filters use the same prescribed initial
law, and any discrepancy between their initial moments is due solely
to finite-ensemble sampling. For the mean-accuracy result, we
additionally assume that the initial ensemble is independent of $u_0$
and the observation noises $(\eta_t)_{t\ge0}$.

The deterministic square-root EnKF mirrors the analysis--forecast cycle
of the exact Kalman filter, replacing the exact forecast mean and
covariance by empirical moments of an ensemble. The analysis mean is
updated using the corresponding empirical Kalman gain, while the
ensemble anomalies are transformed deterministically so that their
covariance follows the Kalman analysis map.

A particular deterministic square-root EnKF is the Ensemble Transform Kalman filter (ETKF)~\cite{BishopEtAl2001}. At forecast time $t$, let
$\widehat u_t^{(1)},\ldots,\widehat u_t^{(N)}\in\mathbb R^{d_u}$
denote the forecast ensemble. Its empirical mean, anomaly matrix, and
covariance are
\begin{equation}
\widehat m_t^N
  := \frac{1}{N}\sum_{n=1}^N \widehat u_t^{(n)},
\qquad
\widehat X_t
  := \bigl[
      \widehat u_t^{(1)}-\widehat m_t^N,\ldots,
      \widehat u_t^{(N)}-\widehat m_t^N
     \bigr],
\qquad
\widehat P_t^N
  := \frac{1}{N-1}\widehat X_t\widehat X_t^\top .
\label{eq:ensemble-forecast-statistics}
\end{equation}
 Here  
$(\widehat m_t^N,\widehat P_t^N)$ plays the role of the exact
forecast moments $(\widehat m_t,\widehat P_t)$. If
$\mathbf 1\in\mathbb R^N$ denotes the all-ones vector, then
$\widehat X_t\mathbf 1=0$, and
$\operatorname{rank}(\widehat P_t^N)\le N-1$.

Set $Y_t:=H\widehat X_t$. Replacing $(\widehat m_t,\widehat P_t)$ by
$(\widehat m_t^N,\widehat P_t^N)$ in the analysis mean update \eqref{eq:KFmeananalysis} gives
\[
m_t^N
  := \widehat m_t^N
     +K_t^N\bigl(y_t-H\widehat m_t^N\bigr),
\qquad
K_t^N
  := \mathsf K(\widehat P_t^N)
   = \widehat X_tY_t^\top
     \bigl(Y_tY_t^\top+(N-1)R\bigr)^{-1}.
\]

To perform the corresponding covariance update \eqref{eq:KFcovanalysis} without perturbing the
observations, define the symmetric ensemble transform
\begin{equation}
\Omega_t
  := I_N+\frac{1}{N-1}Y_t^\top R^{-1}Y_t,
\qquad
T_t:=\Omega_t^{-1/2},
\label{eq:symmetric-transform}
\end{equation}
and set
\[
X_t:=\widehat X_tT_t,
\qquad
P_t^N
  :=\frac{1}{N-1}X_tX_t^\top
    \overset{(\star)}{=} \Psi(\widehat P_t^N).
\]
Proposition~\ref{prop:exact-square-root-identities} verifies that the ensemble transform satisfies the
identity $(\star)$  and shows that the transform preserves anomaly
centering, \(X_t\mathbf 1=0\). 
Thus the transform updates the empirical covariance according to the
same analysis map as in \eqref{eq:KFcovanalysis}.

Let $e_n\in\mathbb R^N$ be the $n$-th coordinate vector. The analyzed
ensemble members are
\[
u_t^{(n)}:=m_t^N+X_te_n,
\qquad n=1,\ldots,N,
\]
and are propagated through the full dynamics:
\[
\widehat u_{t+1}^{(n)}:=Au_t^{(n)}.
\]
Consequently,
\[
\widehat m_{t+1}^N=Am_t^N,
\qquad
\widehat P_{t+1}^N
   =AP_t^NA^\top
   =\Phi(\widehat P_t^N),
\]
in direct analogy with the exact Kalman forecast in \eqref{eq:KFforecast}. Algorithm~\ref{alg:deterministic-square-root-enkf} summarizes the deterministic square-root EnKF. 

\begin{algorithm}[H]
\caption{Deterministic square-root EnKF}
\label{alg:deterministic-square-root-enkf}
\begin{algorithmic}[1]
\Require Ensemble size $N\geq2$ and initial forecast ensemble
$\widehat u_0^{(1)},\ldots,\widehat u_0^{(N)}$ sampled as in
\eqref{eq:gaussian-initial-ensemble}.
\For{$t=0,1,\ldots$}
  \State Compute $\widehat m_t^N$ and $\widehat X_t$ from the forecast ensemble
  using \eqref{eq:ensemble-forecast-statistics}, and set $Y_t=H\widehat X_t$.
  \State Set
  $K_t^N=\widehat X_tY_t^\top(Y_tY_t^\top+(N-1)R)^{-1}$ and
  $m_t^N=\widehat m_t^N+K_t^N(y_t-H\widehat m_t^N)$.
  \State Set
  $\Omega_t=I_N+(N-1)^{-1}Y_t^\top R^{-1}Y_t$,
  $T_t=\Omega_t^{-1/2}$, and $X_t=\widehat X_tT_t$.
  \State Form $u_t^{(n)}=m_t^N+X_te_n$ for $n=1,\ldots,N$.
  \State Forecast $\widehat u_{t+1}^{(n)}=Au_t^{(n)}$ for
  $n=1,\ldots,N$.
\EndFor
\end{algorithmic}
\end{algorithm}

The matrices $\widehat P_t^N$ and $P_t^N$ are the ensemble
approximations to the exact forecast and analysis covariances, and
their accuracy is a principal object of study below. Computationally,
however, Algorithm~\ref{alg:deterministic-square-root-enkf} evaluates the gain and transform directly from
$\widehat X_t$ and $Y_t$; it does not form or invert an empirical state
covariance.

\begin{remark}[Nonuniqueness]
The symmetric transform is a canonical choice, but it is not unique. If $U_t$
is orthogonal and $U_t\one=\one$, then
$X_t=\widehat X_tT_tU_t$ is also centered and has covariance
$\Psi(\widehat P_t^N)$; see
\cite{TippettEtAl2003,LivingsEtAl2008}.
Consequently, the results below apply to any mean-preserving deterministic
square-root update that produces the analysis covariance
$\Psi(\widehat P_t^N)$; the symmetric transform is used only to specify a
concrete implementation.
\end{remark}

\subsection{Time-uniform accuracy results}
\label{sec:main-results}

The sampling requirements involve two intrinsic dimensions. The effective rank \cite{KoltchinskiiLounici2017}
of the nonzero initial covariance is
\begin{equation}
  r_{\mathrm{eff}}(\widehat P_0)
  :=\frac{\operatorname{Tr}(\widehat P_0)}{\|\widehat P_0\|},
  \label{eq:effective-rank}
\end{equation}
while $r_u$ is the dimension of the unstable spectral subspace. For
\(\delta\in(0,1)\), define the covariance and mean initialization scales
\begin{align}
  \Delta_N(\delta)
  &:={C_0\|\widehat P_0\|}
  \left[
  \sqrt{\frac{r_{\mathrm{eff}}(\widehat P_0)+\log(4/\delta)}{N-1}}
  +\frac{r_{\mathrm{eff}}(\widehat P_0)+\log(4/\delta)}{N-1}
  \right],
  \label{eq:Delta-N-definition}
  \\
  \mu_N(\delta)
  &:={
  \sqrt{\frac{\operatorname{Tr}(\widehat P_0)}{N}}
  +\sqrt{\frac{2\|\widehat P_0\|\log(4/\delta)}{N}}
  },
  \label{eq:mu-N-definition}
\end{align}
where \(C_0\) is a universal constant obtained from the sample-covariance concentration bound in \cite[Corollary~2, Equation~(2.8)]{KoltchinskiiLounici2017}.

Let $C_{\mathrm{init}}$ be a sufficiently large universal constant. For the
main results, assume
\begin{equation}
  N-1\ge C_{\mathrm{init}}
  \left[r_{\mathrm{eff}}(\widehat P_0)+r_u+\log(4/\delta)\right].
  \label{eq:effective-rank-ensemble-size}
\end{equation}
The effective-rank term controls the operator-norm error of the empirical
covariance, while $r_u$ ensures that the ensemble spans the unstable subspace; in particular, condition \eqref{eq:effective-rank-ensemble-size} does not require either $N>d_u$ or $N>d_y$.
The effective rank can be much smaller than the ambient dimension: if
$\lambda_1\ge\lambda_2\ge\cdots$ are the nonzero eigenvalues of
$\widehat P_0$, then
$r_{\mathrm{eff}}(\widehat P_0)=\sum_k\lambda_k/\lambda_1$, which may remain
bounded or grow slowly with $d_u$ under pronounced
eigenvalue decay. If $r_u$ is also moderate, the required ensemble size can therefore be far smaller than $d_u$.

Once the initial covariances of the Kalman filter and the EnKF are fixed, their covariance recursions are deterministic and independent of the realized states and observations. Gaussian ensemble initialization is used only to control the initial empirical covariance with high probability.

The following two theorems are our first main results. They establish
high-probability, time-uniform accuracy of the EnKF forecast and analysis
covariances and means.
\begin{theorem}[Time-uniform covariance accuracy]
\label{thm:time-uniform-covariance-approximation}
Under the linear--Gaussian model
\eqref{eq:initial-state-law}--\eqref{eq:observation-model} and
Assumption~\ref{ass:dynamics-initialization}, initialize the ensemble as in
\eqref{eq:gaussian-initial-ensemble}. Fix \(\delta\in(0,1)\) and suppose
\eqref{eq:effective-rank-ensemble-size} holds. Then there is a constant
\(C<\infty\), depending on \(A,H,R,\widehat P_0\) but not on $d_u,d_y$, \(N\) or
\(\delta\), such that, with probability at least \(1-\delta\) over the initial
ensemble,
\begin{equation}
  \sup_{t\ge0}
  \left(
  \|\widehat P_t^N-\widehat P_t\|
  +\|P_t^N-P_t\|
  \right)
  \le C\Delta_N(\delta).
  \label{eq:time-uniform-covariance-approximation}
\end{equation}
\end{theorem}
 
\begin{remark}[Role of perfect-model dynamics]
The possibility of operator-norm covariance accuracy with \(N\le d_u\)
relies on the perfect-model dynamics in
\eqref{eq:dynamics-model}. If a full-rank process-noise covariance
\(Q\succ0\) were added, the exact forecast covariance after each
propagation step would dominate \(Q\), whereas an empirical covariance
formed from \(N\) ensemble members has rank at most \(N-1\).
Consequently, when \(N\le d_u\), its operator-norm error would be at
least \(\lambda_{\min}(Q)\).
\end{remark}

\begin{remark}[Sharper estimates]
The proof yields geometric decay of both covariance errors, and hence of the
Kalman gain error, at rate \(C\rho^t\Delta_N(\delta)\) for some
\(\rho\in(0,1)\). These refinements are used in the mean analysis and are
discussed further in Section~\ref{subsec:proof-covariance-accuracy}.
\end{remark}

\begin{theorem}[Time-uniform mean accuracy]
\label{thm:time-uniform-mean-approximation}
Under the hypotheses of
Theorem~\ref{thm:time-uniform-covariance-approximation}, assume additionally
that the initial ensemble is independent of \(u_0\) and
\((\eta_t)_{t\geq0}\). Then there is a constant
\(C<\infty\), depending on \(A,H,R,\widehat P_0\) but not on $d_u,d_y$, \(N\) or
\(\delta\), such that, with probability at least \(1-\delta\) under the joint law of the initial
ensemble, \(u_0\), and \((\eta_t)_{t\geq0}\),
\begin{equation}
  \sup_{t\ge0}
  \left(
  \|\widehat m_t^N-\widehat m_t\|
  +\|m_t^N-m_t\|
  \right)
  \le
  C\left[
  \mu_N(\delta)
  +\Delta_N(\delta)\sqrt{\log(4/\delta)}
  \right].
  \label{eq:time-uniform-mean-approximation}
\end{equation}
\end{theorem}

Theorems~\ref{thm:time-uniform-covariance-approximation}
and~\ref{thm:time-uniform-mean-approximation} establish time-uniform accuracy
of the nonlocalized square-root EnKF. Section~\ref{sec:localized-model-algorithm}
develops parallel guarantees for a localized filter, replacing global
intrinsic dimensions by local ones at the cost of an additional localization
bias.

\section{The localized ensemble Kalman filter}
\label{sec:localized-model-algorithm}

We now impose a spatial organization on the state space. The dynamics within
one spatial subsystem may be strong and unstable, while interactions between
distinct subsystems have small amplitude and polynomial decay.   
 In Theorem \ref{thm:B8-spatially-decaying-fixed-point}, we show that the stabilizing Kalman covariance inherits the same spatial decay of interactions that are present in the dynamics. This structure
has a statistical consequence: if the Kalman covariance and gain are
approximately local, the analysis gains can be computed from estimates of
fixed-size local covariance blocks rather than a global covariance whose
intrinsic dimension may grow with the number of blocks. In Algorithm \ref{alg:localized-square-root-enkf}, we present an EnKF that localizes only the analysis, while retaining the
full coupled forecast.
We then demonstrate in Theorems \ref{thm:B6-localized-covariance-accuracy} and \ref{thm:B7-localized-mean-accuracy} that this localized EnKF can exploit the spatial decay structure to estimate the local covariance blocks and means uniformly in time. The required ensemble size for these results scales logarithmically in the number of spatial blocks, which thus may be orders of magnitude smaller than the nominal state and observation dimensions.

\subsection{Model and covariance structure}
\label{subsec:localized-model-assumptions}

\subsubsection{Weakly coupled model and assumptions}
\label{subsubsec:localized-model-assumptions}
 We write the block-decomposition of the state space as 
\[
  \R^{d_u}=\mathsf X_1\oplus\cdots\oplus\mathsf X_J,
  \qquad \mathsf X_j\simeq\R^{d_{u,j}},
  \qquad \sum_{j=1}^Jd_{u,j}=d_u,
\]
and let \(\Pi_j:\R^{d_u}\to\mathsf X_j\) be the coordinate projection.
Denote \(u_{t,j}:=\Pi_j u_t\).
We use the corresponding observation-space decomposition
\[
  \R^{d_y}=\R^{d_{y,1}}\oplus\cdots\oplus\R^{d_{y,J}},
  \qquad \sum_{j=1}^Jd_{y,j}=d_y,
\]
and write \(y_t=(y_{t,1},\ldots,y_{t,J})\) and
\(\eta_t=(\eta_{t,1},\ldots,\eta_{t,J})\).  The blocks are ordered, and
\(|j-k|\) denotes the distance between blocks \(j\) and \(k\). Thus the
spatial geometry in the present results is the one-dimensional ordered-block
distance.
To quantify the localized and spatially decaying structure in the system, we introduce several norms and fix some notation  before stating additional assumptions.
For a block operator \(M=(M_{jk})_{j,k=1}^J\), define
\begin{equation}
  \|M\|_{\mathrm{loc}}
  :=\max\left\{
  \max_j\sum_k\|M_{jk}\|,
  \max_k\sum_j\|M_{jk}\|
  \right\}.
  \label{eq:block-local-norm}
\end{equation}
This norm is submultiplicative and dominates the Euclidean operator norm:
\[
  \|M_1M_2\|_{\mathrm{loc}}
  \le \|M_1\|_{\mathrm{loc}}\|M_2\|_{\mathrm{loc}},
  \qquad
  \|M\|\le \|M\|_{\mathrm{loc}}.
\]
For covariance matrices we use the one-index shorthand for the diagonal blocks
\begin{equation}
  \widehat P_{t,j}:=(\widehat P_t)_{jj},
  \qquad
  P_{t,j}^{N,\mathrm{loc}}:=(P_t^{N,\mathrm{loc}})_{jj},
  \label{eq:diagonal-block-shorthand}
\end{equation}
and analogous notation for other decorated covariance matrices. Two indices,
as in \(P_{t,jk}^{N,\mathrm{loc}}\), continue to denote an arbitrary block.
For a block vector \(x=(x_1,\ldots,x_J)\), set
\(\|x\|_{\max}:=\max_j\|x_j\|\). Then
\begin{equation}
  \|Mx\|_{\max}\le \|M\|_{\mathrm{loc}}\|x\|_{\max}.
  \label{eq:B7-local-norm-action}
\end{equation}
 To capture the polynomial decay between blocks, we fix \(\alpha>1\),  and set 
\[
 Z_\alpha:=\sum_{\nu\in\mathbb Z}(1+|\nu|)^{-\alpha},
 \qquad \omega_\alpha(\nu):=Z_\alpha^{-1}(1+|\nu|)^{-\alpha}.
\]
 We also  define
\begin{equation}
  \|M\|_{\mathcal J_\alpha}
  :=\max_{j,k}(1+|j-k|)^\alpha\|M_{jk}\|.
  \label{eq:block-decay-norm}
\end{equation}
One can interpret \(\omega_\alpha\) as a spatial interaction kernel, in which the influence between blocks separated by distance $|\nu|$ decays as $|\nu|^{-\alpha}$. The normalization of \(\omega_\alpha\) ensures that this quantity scales uniformly in the number of blocks. The $\|\cdot\|_{\mathcal J_\alpha}$ norm is a block-wise maximum norm weighted by the interaction kernel. This norm measures the block-wise worst-case deviation from polynomial spatial decay.  In fact, block matrices satisfying $\|M_{jk}\| \leq C (1 + |j-k|)^{-\alpha}$ form a Banach algebra~\cite{Jaffard1990,benzi2017localization}, which we refer to as the block-decay algebra. 
 We now state the spatial structural assumptions we make on the dynamics and observation models.
\begin{assumption}[Local observations and weak spatial interaction]
\label{ass:localized-structure}
The state-transition matrix has the decomposition
\begin{equation}
  A=D+E,
  \qquad
  D=\operatorname{diag}(A_1,\ldots,A_J),
  \qquad
  E_{jj}=0,
  \qquad
  \|E_{jk}\|\le\varepsilon\,\omega_\alpha(j-k),
  \label{eq:weakly-coupled-dynamics}
\end{equation}
where $A_j:\mathsf X_j\to\mathsf X_j$ and $0\le\varepsilon\le1$.  
The between-block interactions are assumed to be weak with polynomial decay, satisfying 
\begin{equation}
  \|E\|_{\mathrm{loc}}\le\varepsilon,
  \qquad
  \|E\|_{\mathcal J_\alpha}\le Z_\alpha^{-1}\varepsilon.
  \label{eq:decay-implies-weak-interaction}
\end{equation}
The observation operator and observation covariance are exactly block local:
\begin{equation}
  H=\operatorname{diag}(H_1,\ldots,H_J),
  \qquad
  R=\operatorname{diag}(R_1,\ldots,R_J),
  \label{eq:block-local-observations}
\end{equation}
where $H_j:\mathsf X_j\to\R^{d_{y,j}}$ and
$R_j\in\Spp^{d_{y,j}}$.  Thus
\[
  y_{t,j}=H_j u_{t,j}+\eta_{t,j},
  \qquad j=1,\ldots,J.
\]
Since the observation noise is Gaussian, the block-diagonal form of \(R\)
also makes \(\eta_{t,1},\ldots,\eta_{t,J}\) independent at each time.
\end{assumption}

In block form, the dynamics are
\begin{equation}
  u_{t+1,j}=A_j u_{t,j}+\sum_{k\ne j}E_{jk}u_{t,k},
  \qquad j=1,\ldots,J.
  \label{eq:weakly-coupled-block-dynamics}
\end{equation}
Thus the within-block dynamics may be strong or unstable, whereas the
influence of other blocks is weak and decays with their spatial separation. To our knowledge, the weakly coupled dynamics model presented here has not been considered previously in the literature. We argue that strong local dynamics coupled by weakly decaying spatial interactions is a natural model for spatially distributed systems.

For a block covariance \(P_j\in\Splus^{d_{u,j}}\), define the local
gain, analysis, and uncoupled Riccati maps by
\begin{align}
  \mathsf K_j(P_j)
  &:=P_j H_j^\top
  \bigl(H_jP_j H_j^\top+R_j\bigr)^{-1},
  \label{eq:local-gain-map}
  \\
  \Psi_j(P_j)
  &:=P_j-
  P_j H_j^\top
  \bigl(H_jP_j H_j^\top+R_j\bigr)^{-1}
  H_jP_j,
  \label{eq:local-analysis-map}
  \\
  \Phi_j(P_j)
  &:=A_j\Psi_j(P_j)A_j^\top.
  \label{eq:local-riccati-map}
\end{align}

For each hyperbolic block $A_j$, let \(\mathsf U_j\) and \(\mathsf S_j\)
denote its real spectral subspaces associated with eigenvalues of modulus
greater and less than one, respectively, so
$\mathsf X_j=\mathsf U_j\oplus\mathsf S_j$. In fixed coordinates adapted to
\(\mathsf X_j=\mathsf U_j\oplus\mathsf S_j\), write
\[
 A_j=\begin{pmatrix}A_{u,j}&0\\0&A_{s,j}\end{pmatrix},
 \qquad H_j=\begin{pmatrix}H_{u,j}&H_{s,j}\end{pmatrix}.
\]
For each block, we denote the dimension of the unstable spectral subspace of $A_j$ as \(r_{u,j}:=\dim(\mathsf U_j)\). 

Our analysis of the localized EnKF employs a block-diagonal Riccati trajectory as a reference, whose blocks each satisfy the properties of the EnKF introduced in Section~\ref{sec:problem-setting}. We introduce the block-diagonal reference Riccati trajectory by
\begin{equation}\label{eq:block-reference-riccati-flow}
  \widehat P_{t,j}^0:=\Phi_j^t(\widehat P_{0,j}^0),
  \qquad
  \widehat P_t^0:=\operatorname{diag}(\widehat P_{t,1}^0,\ldots,
  \widehat P_{t,J}^0).
\end{equation}
It uses only the uncoupled block dynamics, $D$, in place of the full dynamics, $A$, to compute the Riccati flow. The reference flow is auxiliary to our analysis, and is not computed by the localized EnKF algorithm.

 We now state a uniform blockwise counterpart of
Assumption~\ref{ass:dynamics-initialization} for the weakly coupled model, along with a uniform condition on the initial block-diagonal reference covariance.

\begin{assumption}[Uniform local dynamics and initialization]
\label{ass:uniform-local-riccati}
The following conditions hold uniformly in the block index \(j\) and
the number of blocks \(J\).

\begin{enumerate}
\item The block dimensions and model coefficients are uniformly bounded:
\[
  d_{u,j}\le d_{u,\mathrm{loc}},
  \qquad
  d_{y,j}\le d_{y,\mathrm{loc}},
\]
and
\[
  \|A_j\|\le a_+,
  \qquad
  \|H_j\|\le h_+,
  \qquad
  r_-I\preceq R_j\preceq r_+I.
\]

\item The matrices \(A_j\) are uniformly hyperbolic and the pairs
\((A_j,H_j)\) are uniformly detectable. Each unstable spectral
subspace is nontrivial.

\item  The block-diagonal reference initial covariance
 \( \widehat P_0^0
  =\operatorname{diag}(\widehat P_{0,1}^0,\ldots,
  \widehat P_{0,J}^0) \) satisfies, 
 uniformly in the block index \(j\), 
\begin{equation}
  (\widehat P_{0,j}^0)_{uu}
  \succeq \underline p_0 I_{\mathsf U_j},
  \qquad
  \|\widehat P_{0,j}^0\|\le\overline p_0.
  \label{eq:block-diagonal-initial-cov-class-sketch}
\end{equation}
\end{enumerate}
\end{assumption}

Uniform hyperbolicity means that the coordinates adapted to
\(\mathsf X_j=\mathsf U_j\oplus\mathsf S_j\) have condition numbers
bounded by a constant \(\kappa_{\mathrm{sp}}\), independent of \(j\)
and \(J\), and that there are constants
\(C_{\mathrm{sp}}<\infty\) and
\(\gamma_s,\gamma_u\in(0,1)\), also independent of \(j\) and \(J\),
such that
\begin{equation}
  \|A_{s,j}^t\|\le C_{\mathrm{sp}}\gamma_s^t,
  \qquad
  \|A_{u,j}^{-t}\|\le C_{\mathrm{sp}}\gamma_u^t,
  \qquad t\ge0.
  \label{eq:uniform-local-spectral-rates}
\end{equation}
Because each \(A_j\) is hyperbolic, detectability of \((A_j,H_j)\)
is equivalent to observability of its unstable restriction
\((A_{u,j},H_{u,j})\). The required uniformity means that there are
an integer \(L\ge1\) and a constant \(c_o>0\), independent of \(j\)
and \(J\), such that
\begin{equation}
  \sum_{\ell=0}^{L-1}
  \|R_j^{-1/2}H_{u,j}A_{u,j}^{\ell}v\|^2
  \ge c_o\|v\|^2,
  \qquad v\in\mathsf U_j.
  \label{eq:uniform-local-unstable-observability}
\end{equation}
Thus the observations resolve the locally unstable directions uniformly
over the blocks, while unobserved stable directions decay on their own. For a sufficiently  weak coupling, hyperbolicity of $A$ follows from uniform hyperbolicity of \(A_j\), and detectability of the coupled pair $(A,H)$ from the uniform block assumptions. Our proofs, however, do not invoke these global properties. The coupled estimates follow directly from perturbative comparison with the block-diagonal reference Riccati trajectory~\eqref{eq:block-reference-riccati-flow}, and the uniform bounds keep this comparison from deteriorating as $J$ grows. The lower bound on \((\widehat P_{0,j}^0)_{uu}\) further ensures that  every locally unstable direction is represented at initialization.   
 Overall, Assumptions \ref{ass:localized-structure} and \ref{ass:uniform-local-riccati} impose more refined structure on the system and observations than Assumption \ref{ass:dynamics-initialization}, thus enabling an improved scaling for algorithms that exploit such structure.

We bound both the EnKF sampling error and the effect of the weak interaction against  the 
block-diagonal reference flow. The corresponding block-diagonal reference analysis covariances and gains will also be used in the proof and are given by
\begin{equation}
\begin{aligned}
  P_{t,j}^0&:=\Psi_j(\widehat P_{t,j}^0),
  &K_{t,j}^0&:=\mathsf K_j(\widehat P_{t,j}^0),\\
  P_t^0&:=\Psi(\widehat P_t^0)
  =\operatorname{diag}(P_{t,1}^0,\ldots,P_{t,J}^0),
  &K_t^0&:=\mathsf K(\widehat P_t^0)
  =\operatorname{diag}(K_{t,1}^0,\ldots,K_{t,J}^0).
\end{aligned}
  \label{eq:block-reference-analysis-gain}
\end{equation}

\subsubsection{Spatial structure of the Kalman covariance}
\label{subsec:localized-steady-state-main}

The weak-interaction model has a useful consequence already at the level of
the exact Kalman filter. When $\varepsilon=0$, the stabilizing covariance is
block diagonal.  For $\varepsilon>0$, the next result shows that weak interactions preserve this
structure approximately: the off-diagonal blocks inherit the spatial decay
of the dynamics, while the diagonal blocks remain $O(\varepsilon)$-close to
their uncoupled counterparts. This provides the structural motivation for the
localized EnKF introduced in Section~\ref{subsec:localized-algorithm}.

Let \(\widehat P_{\infty,j}^0\) be the unique stabilizing fixed point of the local Riccati map $\Phi_j$, which is shown to exist in 
 Proposition~\ref{prop:B8-local-fixed-points}, and set
\[
  \widehat P_\infty^0:=\operatorname{diag}(\widehat P_{\infty,1}^0,\ldots,
  \widehat P_{\infty,J}^0),
  \qquad
  K_\infty^0:=\mathsf K(\widehat P_\infty^0),\qquad
  B_\infty^0:=D(I-K_\infty^0H).
\]
Here $B_\infty^0$ is called the forecast closed-loop map because it
 propagates a forecast error through one analysis--forecast
cycle of the uncoupled steady-state Kalman filter.
\begin{theorem}[Spatial decay of the stabilizing Kalman covariance]
\label{thm:B8-spatially-decaying-fixed-point}
Under Assumptions~\ref{ass:localized-structure} and
\ref{ass:uniform-local-riccati}, there exist
$\varepsilon_\infty>0$, $C_\infty<\infty$, and
$\rho_\infty\in(0,1)$, independent of $J$, such that for
$0\le\varepsilon\le\varepsilon_\infty$ the full Riccati map $\Phi$ has a
stabilizing fixed point $\widehat P_\infty$ that is the unique fixed point in
a fixed block-decay neighborhood of $\widehat P_\infty^0$. Define
\[
  K_\infty:=\mathsf K(\widehat P_\infty),
  \qquad
  B_\infty:=A(I-K_\infty H).
\]
Then
\begin{equation}
  \|\widehat P_\infty-\widehat P_\infty^0\|_{\mathcal J_\alpha}
  +\|K_\infty-K_\infty^0\|_{\mathcal J_\alpha}
  +\|B_\infty-B_\infty^0\|_{\mathcal J_\alpha}
  \le C_\infty\varepsilon.
  \label{eq:B8-fixed-point-decay-bound}
\end{equation}
In particular,
\begin{equation}
  \|(\widehat P_\infty)_{jk}\|
  \le C_\infty\varepsilon(1+|j-k|)^{-\alpha},
  \qquad j\ne k,
  \label{eq:B8-fixed-point-entry-decay}
\end{equation}
and
\begin{equation}
  \max_j\| \widehat P_{\infty,j}
  -\widehat P_{\infty,j}^0\|
  \le C_\infty\varepsilon.
  \label{eq:B8-fixed-point-diagonal-bias}
\end{equation}

Moreover, there exist
$\mathfrak d_*,\mathfrak d_{*,\alpha}>0$, independent of $J$, such that if
\begin{equation}
  \|\widehat P_0-\widehat P_0^0\|_{\mathrm{loc}}
  \le\mathfrak d_*,
  \qquad
  \|\widehat P_0-\widehat P_0^0\|_{\mathcal J_\alpha}
  \le\mathfrak d_{*,\alpha},
  \label{eq:B8-fixed-point-initial-neighborhood}
\end{equation}
then 
\begin{equation}
  \|\widehat P_t-\widehat P_\infty\|_{\mathrm{loc}}
  \le C_\infty\rho_\infty^t, \qquad 
  \|\widehat P_t-\widehat P_\infty\|_{\mathcal J_\alpha}
  \le C_\infty\rho_\infty^t.
  \label{eq:B8-full-fixed-point-convergence-decay}
\end{equation}
\end{theorem}

Thus the weakly coupled model produces an exact Kalman covariance whose
long-range blocks are small, rather than assuming such decay directly. 
The next result identifies the leading-order effect of the interaction.
\begin{proposition}[First-order covariance response]
\label{prop:B8-first-order-expansion}
Under Assumptions~\ref{ass:localized-structure} and
\ref{ass:uniform-local-riccati}, suppose the interaction is parameterized as
$A_\varepsilon=D+\varepsilon\mathcal E$, where
$\mathcal E_{jj}=0$ and
$\|\mathcal E_{jk}\|\le\omega_\alpha(j-k)$.
Write $\widehat P_\infty^\varepsilon$ for the corresponding stabilizing fixed
point from Theorem~\ref{thm:B8-spatially-decaying-fixed-point}. Then, as
$\varepsilon\downarrow0$, uniformly in $J$,
\begin{equation}
  \widehat P_\infty^\varepsilon
  =\widehat P_\infty^0
  +\varepsilon \widehat P_\infty^{(1)}
  +O_{\mathcal J_\alpha}(\varepsilon^2),
  \label{eq:B8-first-order-expansion}
\end{equation}
where $\widehat P_\infty^{(1)}$ is the unique solution in the block-decay
algebra to 
\begin{equation}
  \widehat P_\infty^{(1)}
  -B_\infty^0\widehat P_\infty^{(1)}(B_\infty^0)^\top
  =G_\infty^{(1)},
  \label{eq:B8-first-order-lyapunov}
\end{equation}
with
\begin{equation}
  G_\infty^{(1)}
  :=\mathcal E\Psi(\widehat P_\infty^0)D^\top
  +D\Psi(\widehat P_\infty^0)\mathcal E^\top.
  \label{eq:B8-first-order-forcing}
\end{equation}
Equivalently,
\begin{equation}
  \widehat P_\infty^{(1)}
  =\sum_{\ell=0}^\infty
  (B_\infty^0)^{\ell} G_\infty^{(1)}
  ((B_\infty^0)^\top)^{\ell},
  \label{eq:B8-first-order-series}
\end{equation}
and the series converges in $\mathcal J_\alpha$ with a dimension-independent
bound. The stability of $B_\infty^0$ ensures that interaction effects from
successive analysis--forecast cycles decay and the series is summable. The
remainder in \eqref{eq:B8-first-order-expansion} is also uniform in
$J$.
\end{proposition}

The proofs of Theorem~\ref{thm:B8-spatially-decaying-fixed-point} and
Proposition~\ref{prop:B8-first-order-expansion} are given in
Section~\ref{sec:localized-steady-state-decay}. 
These results provide structural motivation for the localized EnKF algorithm, but are not used in the proofs of the time-uniform accuracy results below.

At steady state, the decay of $K_\infty$ 
implied by \eqref{eq:B8-fixed-point-decay-bound}
shows that distant observations have
only a weak influence on each local analysis. The localized EnKF update below
therefore computes its gains using only the diagonal covariance blocks,
avoiding the sampling cost of a global covariance estimate. It nevertheless
retains the full matrix $A$ in every forecast, so the weak cross-block
interactions continue to propagate through the ensemble. The resulting
deterministic localization error is of order $\varepsilon$, as quantified in
Section~\ref{sec:localized-main-results}.

\subsection{Localized square-root EnKF}
\label{subsec:localized-algorithm}
We now describe the localized square-root EnKF that we analyze in Section \ref{sec:localized-main-results}. 
The localized ensemble is initialized by
\begin{equation}
  \widehat u_0^{(1),\mathrm{loc}},\ldots,
  \widehat u_0^{(N),\mathrm{loc}}
  \stackrel{\mathrm{i.i.d.}}{\sim}
  \mathcal N(\widehat m_0,\widehat P_0).
  \label{eq:localized-gaussian-initial-ensemble}
\end{equation}
For the mean-accuracy result Theorem \ref{thm:B7-localized-mean-accuracy}, this ensemble is assumed independent of the
signal and observation noises.  The covariance result Theorem \ref{thm:B6-localized-covariance-accuracy} does not require this
additional independence.
The algorithm is defined for a general initial covariance \(\widehat P_0\);
the accuracy results below specialize to the block-diagonal reference
initialization $\widehat P_0 = \widehat P_0^0$ in Assumption~\ref{ass:uniform-local-riccati}.
The algorithm maintains a full ensemble
\(\widehat u_t^{(1),\mathrm{loc}},\ldots,
\widehat u_t^{(N),\mathrm{loc}}\in\R^{d_u}\). Its local forecast blocks are
\[
  \widehat u_{t,j}^{(n),\mathrm{loc}}
  :=\Pi_j\widehat u_t^{(n),\mathrm{loc}}.
\]
For each block \(j\), define the local mean, local anomaly matrix, and local covariance as
\begin{equation}
\begin{aligned}
  \widehat m_{t,j}^{N,\mathrm{loc}}
  &:=\frac1N\sum_{n=1}^N\widehat u_{t,j}^{(n),\mathrm{loc}},
  &
  \widehat X_{t,j}^{\mathrm{loc}}
  &:=[\widehat u_{t,j}^{(1),\mathrm{loc}}-
  \widehat m_{t,j}^{N,\mathrm{loc}},\ldots,
  \widehat u_{t,j}^{(N),\mathrm{loc}}-
  \widehat m_{t,j}^{N,\mathrm{loc}}],
  \\
  \widehat P_{t,j}^{N,\mathrm{loc}}
  &:=\frac1{N-1}\widehat X_{t,j}^{\mathrm{loc}}
  (\widehat X_{t,j}^{\mathrm{loc}})^\top.
\end{aligned}
\label{eq:localized-forecast-statistics}
\end{equation}
We also write
\(\widehat X_t^{\mathrm{loc}}
:=[(\widehat X_{t,1}^{\mathrm{loc}})^\top,\ldots,
(\widehat X_{t,J}^{\mathrm{loc}})^\top]^\top\)
for the stacked forecast anomaly matrix, and set
\(
  Y_{t,j}^{\mathrm{loc}}:=H_j\widehat X_{t,j}^{\mathrm{loc}}.
\)
As in the EnKF, the local gain can be evaluated directly
from the anomaly matrices:
\begin{equation}
\begin{aligned}
  K_{t,j}^{N,\mathrm{loc}}
  &:=\widehat X_{t,j}^{\mathrm{loc}}
  (Y_{t,j}^{\mathrm{loc}})^\top
  \left[Y_{t,j}^{\mathrm{loc}}(Y_{t,j}^{\mathrm{loc}})^\top
  +(N-1)R_j\right]^{-1}
  =\mathsf K_j(\widehat P_{t,j}^{N,\mathrm{loc}}),\\
  m_{t,j}^{N,\mathrm{loc}}
  &:=\widehat m_{t,j}^{N,\mathrm{loc}}
  +K_{t,j}^{N,\mathrm{loc}}
  \bigl(y_{t,j}-H_j\widehat m_{t,j}^{N,\mathrm{loc}}\bigr).
\end{aligned}
  \label{eq:localized-analysis-mean}
\end{equation}
For each \(j\), use the symmetric local transform
\begin{equation}
  \Omega_{t,j}^{\mathrm{loc}}
  :=I_N+\frac1{N-1}
  (Y_{t,j}^{\mathrm{loc}})^\top R_j^{-1}Y_{t,j}^{\mathrm{loc}},
  \qquad
  T_{t,j}^{\mathrm{loc}}
  :=(\Omega_{t,j}^{\mathrm{loc}})^{-1/2},
  \label{eq:localized-square-root-transform}
\end{equation}
and set
\( 
  X_{t,j}^{\mathrm{loc}}
  :=\widehat X_{t,j}^{\mathrm{loc}}T_{t,j}^{\mathrm{loc}}.
\)
Stacking the local quantities gives the global forecast mean, analysis mean, and
block-diagonal gain
\begin{equation}
\begin{aligned}
  \widehat m_t^{N,\mathrm{loc}}
  &:=(\widehat m_{t,1}^{N,\mathrm{loc}},\ldots,
  \widehat m_{t,J}^{N,\mathrm{loc}}), &
   m_t^{N,\mathrm{loc}}
  &:=(m_{t,1}^{N,\mathrm{loc}},\ldots,m_{t,J}^{N,\mathrm{loc}}),
\\
  K_t^{N,\mathrm{loc}}
  &:=\operatorname{diag}(K_{t,1}^{N,\mathrm{loc}},\ldots,
  K_{t,J}^{N,\mathrm{loc}}).
\end{aligned}
  \label{eq:localized-global-forecast-mean-gain}
\end{equation}
Thus
\[
  m_t^{N,\mathrm{loc}}
  =\widehat m_t^{N,\mathrm{loc}}
  +K_t^{N,\mathrm{loc}}(y_t-H\widehat m_t^{N,\mathrm{loc}}).
\]
Stack the local anomalies according to the fixed block
decomposition: 
\begin{equation}
  X_t^{\mathrm{loc}}
  :=\begin{bmatrix}
  X_{t,1}^{\mathrm{loc}}\\ \vdots\\ X_{t,J}^{\mathrm{loc}}
  \end{bmatrix}.
  \label{eq:localized-stacked-analysis-statistics}
\end{equation}
The \(n\)-th analyzed ensemble member is
\[
  u_t^{(n),\mathrm{loc}}
  :=m_t^{N,\mathrm{loc}}+X_t^{\mathrm{loc}}e_n.
\]
The forecast is then performed with the full matrix \(A\):
\begin{equation}
  \widehat u_{t+1}^{(n),\mathrm{loc}}
  :=A u_t^{(n),\mathrm{loc}},
  \qquad n=1,\ldots,N.
  \label{eq:localized-full-forecast}
\end{equation}
Thus neighboring or interacting blocks communicate through every forecast,
even though the analysis uses only local covariances and local observations. 

\begin{algorithm}[H]
\caption{Localized square-root EnKF}
\label{alg:localized-square-root-enkf}
\begin{algorithmic}[1]
\Require A global initial forecast ensemble
\(\widehat u_0^{(1),\mathrm{loc}},\ldots,
\widehat u_0^{(N),\mathrm{loc}}\), sampled as in
\eqref{eq:localized-gaussian-initial-ensemble}.
\For{$t=0,1,\ldots$}
  \For{$j=1,\ldots,J$}
    \State Compute \(\widehat m_{t,j}^{N,\mathrm{loc}}\),
    \(\widehat X_{t,j}^{\mathrm{loc}}\), and
    \(Y_{t,j}^{\mathrm{loc}}=H_j\widehat X_{t,j}^{\mathrm{loc}}\) from the
    local forecast ensemble.
    \State Evaluate \(K_{t,j}^{N,\mathrm{loc}}\) and update
    \(m_{t,j}^{N,\mathrm{loc}}\) using
    \eqref{eq:localized-analysis-mean}.
    \State Apply the symmetric local transform
    \(X_{t,j}^{\mathrm{loc}}
    =\widehat X_{t,j}^{\mathrm{loc}}(\Omega_{t,j}^{\mathrm{loc}})^{-1/2}\).
  \EndFor
  \State Stack the local means and anomaly matrices into
  \(m_t^{N,\mathrm{loc}}\) and \(X_t^{\mathrm{loc}}\).
  \State Form the global analyzed ensemble and forecast every member by
  \(\widehat u_{t+1}^{(n),\mathrm{loc}}=Au_t^{(n),\mathrm{loc}}\).
\EndFor
\end{algorithmic}
\end{algorithm}

The analysis update uses only local anomaly matrices and local
observation-space inverses; it does not form an empirical state
covariance and is embarrassingly parallel over the local blocks. 
On the other hand, the
polynomially decaying interactions are retained in the full forecast
and are not truncated. 
This distinction is important for time-uniform mean accuracy. Replacing
\(A\) by \(D\) in the forecast would simplify the covariance recursion,
but would introduce a persistent model-mismatch error in the comparison
with the coupled Kalman mean. In an unstable system, this error need not
remain uniformly bounded even when the cross-block interactions are
weak. Retaining the full matrix \(A\) avoids this obstruction.

For the theoretical results, we define the full covariance
\begin{equation}
  P_t^{N,\mathrm{loc}}
  :=\frac1{N-1}X_t^{\mathrm{loc}}(X_t^{\mathrm{loc}})^\top,
  \label{eq:full-localized-analysis-covariance}
\end{equation}
and write \(P_{t,jk}^{N,\mathrm{loc}}\) for its \((j,k)\)-block.  This full
matrix is not formed by Algorithm~\ref{alg:localized-square-root-enkf}.
The diagonal blocks satisfy the exact local square-root identity
\begin{equation}
  P_{t,j}^{N,\mathrm{loc}}
  =\frac1{N-1}X_{t,j}^{\mathrm{loc}}
  (X_{t,j}^{\mathrm{loc}})^\top
  =\Psi_j(\widehat P_{t,j}^{N,\mathrm{loc}}).
  \label{eq:localized-exact-analysis-identity}
\end{equation}

The symmetric local transforms are part of the algorithmic specification.
A common mean-preserving orthogonal right factor leaves the full analyzed
covariance unchanged, but using different block-dependent orthogonal right
factors generally changes the cross-covariances and hence subsequent local
forecast covariances. Thus the nonuniqueness discussed after
Algorithm~\ref{alg:deterministic-square-root-enkf} does not extend to
block-dependent transforms without further qualification.

\subsection{Time-uniform accuracy}
\label{sec:localized-main-results}
 In this section, we state our main results for the localized square-root EnKF, demonstrating the time-uniform accuracy of the local covariance blocks and ensemble means. 
 Recall that \(\widehat P_t\), \(P_t=\Psi(\widehat P_t)\), and
\(K_t=\mathsf K(\widehat P_t)\) defined in Section~\ref{subsec:EnKF} are the full Kalman forecast covariance,
analysis covariance, and gain for the coupled spatial model, and  
\(\widehat P_t^0=\operatorname{diag}(\widehat P_{t,1}^0,\ldots,
\widehat P_{t,J}^0)\) is the block-diagonal reference Riccati trajectory defined in \eqref{eq:block-reference-riccati-flow}.
As an analytical comparison flow, the reference trajectory separates the sampling error of the localized
ensemble from the deterministic discrepancy generated by the weak cross-block interactions. Localization changes the relevant sampling problem from estimating one global
covariance to estimating all diagonal covariance blocks. The complexity is
therefore governed by the largest blockwise effective rank, together with the
logarithmic cost of controlling all blocks simultaneously.

Define
\begin{equation}
  r_{\mathrm{eff},j}:=\frac{\operatorname{Tr}(\widehat P_{0,j})}{\|\widehat P_{0,j}\|},
  \qquad
  r_{\mathrm{eff},\mathrm{loc}}:=\max_j r_{\mathrm{eff},j}.
  \label{eq:B5-local-effective-ranks}
\end{equation}
For \(\delta\in(0,1)\), define the local covariance initialization scale
\begin{equation}
  \Delta_{N,\mathrm{loc}}(\delta)
  :={C_0}\max_j\left\{\|\widehat P_{0,j}\|
  \left[\sqrt{\frac{r_{\mathrm{eff},j}
  +\log\left(4J/\delta\right)}{N-1}}
  +\frac{r_{\mathrm{eff},j}+\log\left(4J/\delta\right)}{N-1}\right]\right\},
  \label{eq:B5-local-covariance-scale}
\end{equation}
where \(C_0\) is the universal constant in the effective-rank Gaussian
sample-covariance inequality.
Here \(r_{\mathrm{eff},\mathrm{loc}}\) measures the largest blockwise intrinsic
dimension, while the factor \(\log J\) is the cost of simultaneous covariance
concentration over all blocks. When $\widehat P_0=\widehat P_0^0$,
Assumption~\ref{ass:uniform-local-riccati} gives
\begin{equation}
  r_{u,j}\le
  \frac{\overline p_0}{\underline p_0}\,r_{\mathrm{eff},j},
  \qquad j=1,\ldots,J,
  \label{eq:local-unstable-dimension-effective-rank}
\end{equation}
so the local effective rank also controls the number of unstable directions in each block.

This gives a substantial reduction when the system consists of many comparable
blocks. For example, in a block system satisfying the hypotheses of both
sections, the global effective rank and unstable dimension in \eqref{eq:effective-rank-ensemble-size} may grow proportionally to $J$
while the local intrinsic dimensions remain bounded. The covariance
ensemble-size requirement is then reduced from order $J$ to order $\log J$,
and the leading square-root sampling scale from order $\sqrt{J/N}$ to
$\sqrt{(\log J)/N}$. In exchange, the localized theorem controls the diagonal
covariance blocks rather than the full covariance and incurs a deterministic
$O(\varepsilon)$ bias.

\subsubsection{Time-uniform covariance accuracy}
\label{sec:localized-covariance-accuracy}
The first result quantifies the time-uniform localized covariance accuracy by the local covariance initialization scale and the deterministic bias introduced by localization.
\begin{theorem}[Time-uniform localized covariance accuracy]
\label{thm:B6-localized-covariance-accuracy}
Under Assumptions~\ref{ass:localized-structure} and
\ref{ass:uniform-local-riccati}, suppose that
$
  \widehat P_0=\widehat P_0^0.
$
Initialize the localized ensemble as in
\eqref{eq:localized-gaussian-initial-ensemble} and fix \(\delta\in(0,1)\).
There exist \(\varepsilon_0>0\), \(C_{\mathrm{init,loc}}<\infty\), and
\(C<\infty\), depending only on the uniform model parameters and independent
of \(J,N,\delta\), such that, if
\(0\le\varepsilon\le\varepsilon_0\) and
\begin{equation}
  N-1\ge
  C_{\mathrm{init,loc}}\left[
    r_{\mathrm{eff},\mathrm{loc}}
    +\log\left(\frac{4J}{\delta}\right)
  \right],
  \label{eq:B6-local-ensemble-size}
\end{equation}
then, with probability at least \(1-\delta\) over the initial ensemble,
\begin{equation}
  \sup_{t\ge0}\left[
  \max_j
  \left\|\widehat P_{t,j}^{N,\mathrm{loc}}-\widehat P_{t,j}\right\|
  +
  \max_j
  \left\|P_{t,j}^{N,\mathrm{loc}}-P_{t,j}\right\|
  \right]
  \le
  C\bigl[\Delta_{N,\mathrm{loc}}(\delta)+\varepsilon\bigr].
  \label{eq:B6-time-uniform-covariance-accuracy}
\end{equation}
\end{theorem}

The right-hand side separates the blockwise covariance-sampling error
$\Delta_{N,\mathrm{loc}}(\delta)$ from the $O(\varepsilon)$ localization bias.
Increasing $N$ reduces the first term but not the second, so localization is
most effective when the local intrinsic dimensions are small and the
cross-block interactions are weak.

\begin{remark}[Sharper estimates]
The proof also gives geometric forgetting of the initial sampling error, the
corresponding gain estimate, and an extension to sufficiently close
non-block-diagonal initial covariances. See
Proposition~\ref{prop:localized-covariance-package} and
Theorems~\ref{thm:B3-localization-bias} and
\ref{thm:B4-empirical-covariance-stability}.
\end{remark}

\begin{remark}[What is and is not approximated]
Theorem~\ref{thm:B6-localized-covariance-accuracy} controls the local forecast
and analysis covariance blocks, not the full rank-deficient empirical
covariance in a global operator norm. When $N\ll d_u$, the latter is necessarily
rank deficient; such global control is neither asserted by the theorem nor
needed to implement the localized analysis.
\end{remark}

\subsubsection{Time-uniform mean accuracy}
\label{sec:localized-mean-accuracy}

The covariance estimate also controls the localized gain; combined with
stability of the resulting mean recursion, this yields the following
uniform-in-time moment bound for the local ensemble means.
\begin{theorem}[Time-uniform localized mean accuracy]
\label{thm:B7-localized-mean-accuracy}
Under the conditions of
Theorem~\ref{thm:B6-localized-covariance-accuracy} and
the linear--Gaussian model
\eqref{eq:initial-state-law}--\eqref{eq:observation-model}, assume in addition
that the initial ensemble is independent of the signal and observation
noises.  There is a constant
\(C<\infty\), independent of \(J,N,\delta,q\), such that there exists an event,
measurable with respect to the initial centered anomaly matrix and of
probability at least \(1-\delta\), on which, for every \(q\ge1\),
\begin{align}
  &\sup_{t\ge0}\Bigg[
  \left(
  \mathbb E\left[
  \|\widehat m_t^{N,\mathrm{loc}}-\widehat m_t\|_{\max}^q
  \,\middle|\,\widehat X_0^{\mathrm{loc}}
  \right]
  \right)^{1/q}
  \nonumber\\
  &\qquad+
  \left(
  \mathbb E\left[
  \|m_t^{N,\mathrm{loc}}-m_t\|_{\max}^q
  \,\middle|\,\widehat X_0^{\mathrm{loc}}
  \right]
  \right)^{1/q}
  \Bigg]
  \le
  C\sqrt{q+\log(2J)}
  \left[
  N^{-1/2}
  +\Delta_{N,\mathrm{loc}}(\delta)
  +\varepsilon
  \right].
  \label{eq:B7-time-uniform-localized-mean}
\end{align}
The conditional expectations hold the centered initial anomalies fixed and
average over the initial sample mean, the signal, and the observation noises.
\end{theorem}

The factor $\sqrt{q+\log(2J)}$ combines Gaussian moment growth with the
logarithmic cost of controlling the maximum over all blocks. Inside the
brackets, $N^{-1/2}$ is the initial sample-mean scale,
$\Delta_{N,\mathrm{loc}}(\delta)$ is the covariance-sampling scale, and
$\varepsilon$ is the localization bias. Unlike the pathwise bound in
Theorem~\ref{thm:time-uniform-mean-approximation}, the localized result is
formulated in uniform-in-time moments; Remark~\ref{rem:B7-no-pathwise-supremum}
explains why the persistent gain discrepancy precludes a corresponding
infinite-time pathwise supremum.

\begin{remark}[Sharper estimates]
We prove a sharper
conditional-moment estimate in \eqref{eq:B7-time-uniform-localized-mean-detailed};
Proposition~\ref{prop:B7-refined-gain-decomposition} supplies the gain bound used
in its proof.
\end{remark}

\begin{remark}[Why the result is formulated in moments]
\label{rem:B7-no-pathwise-supremum}
When \(\varepsilon>0\), the gain discrepancy contains a persistent
\(O(\varepsilon)\) component.  It is driven at every time by Gaussian
innovations.  Even the scalar stable recursion
\[
  z_{t+1}=\rho z_t+\varepsilon\xi_t,
  \qquad |\rho|<1,
  \qquad \xi_t\stackrel{\mathrm{i.i.d.}}{\sim}\mathcal N(0,1),
\]
satisfies \(\sup_{t\ge0}|z_t|=\infty\) almost surely.  Thus an unnormalized
infinite-time pathwise supremum is generally unavailable, whereas the
uniform-in-time moment bound in \eqref{eq:B7-time-uniform-localized-mean} is stable and its persistent
localization contribution is proportional to \(\varepsilon\).
\end{remark}

Theorems~\ref{thm:B6-localized-covariance-accuracy} and
\ref{thm:B7-localized-mean-accuracy} show that weakly coupled localization
trades a controlled $O(\varepsilon)$ bias for sampling requirements governed by
local intrinsic dimension and only logarithmic dependence on the number of
blocks.

\section{Proofs}
\label{sec:proofs}

Throughout this section, \(C<\infty\) and \(\rho\in(0,1)\) denote constants
that depend only on the fixed model parameters and may change from line to
line.  In the localized setting they are also independent of the number of
blocks. Their value is never allowed to depend on \( d_u,d_y, t,N,J,\delta\), or \(q\);
any additional dependence is stated explicitly.
We consider the EnKF and localized EnKF in turn.

\subsection{Ensemble Kalman filter}
\label{subsec:enkf-proofs}

We first collect the identities and estimates used in the EnKF proofs. These auxiliary results are proved in Section~\ref{app:enkf} of the supplementary materials.

 We begin by finding recursive formulas for the mean errors and the Riccati flow of the covariance updates. Set \(\widehat e_t^N:=\widehat m_t^N-\widehat m_t\),
\(e_t^N:=m_t^N-m_t\), and \(\iota_t:=y_t-H\widehat m_t\).  For $P\in\Splus^{d_u}$, define the forecast closed-loop map
\begin{equation}
  B(P):=A(I-\mathsf K(P)H).
  \label{eq:closed-loop-map}
\end{equation}
The map $B(P)$ propagates a forecast error
through one Kalman analysis correction and the subsequent forecast, with the open-loop dynamics $A$.
The derivative identity  in \eqref{eq:enkf-riccati-difference-identity}  below shows that the same map also propagates covariance perturbations.

\begin{proposition}[Exact square-root identities]
\label{prop:exact-square-root-identities}
The transform in \eqref{eq:symmetric-transform} is well defined and
satisfies
\begin{equation}
  T_t\one=\one,
  \qquad
  X_t\one=0.
  \label{eq:centering-identity}
\end{equation}
Moreover,
\begin{equation}
  P_t^N=\Psi(\widehat P_t^N),
  \qquad
  \widehat P_{t+1}^N=\Phi(\widehat P_t^N),
  \label{eq:exact-forecast-identities}
\end{equation}
and hence
\begin{align}
  \widehat P_t^N=\Phi^t(\widehat P_0^N),
  \qquad \widehat P_t=\Phi^t(\widehat P_0),
  \label{eq:exact-riccati-trajectories-main}
\end{align}
while the mean errors satisfy
\begin{align}
  e_t^N&=(I-K_t^NH)\widehat e_t^N+(K_t^N-K_t)\iota_t,
  \nonumber\\
  \widehat e_{t+1}^N
  &=B(\widehat P_t^N)\widehat e_t^N
  +A(K_t^N-K_t)\iota_t,
  \label{eq:main-exact-enkf-recursions}
\end{align}
where \(K_t^N=\mathsf K(\widehat P_t^N)\) and
\(K_t=\mathsf K(\widehat P_t)\).  These conclusions require neither
invertibility of \(\widehat P_t^N\) nor the condition \(N>d_u\).
\end{proposition}

 Now we turn to Riccati identities for the covariance update, written in a form that does not require $P$ to be invertible. These identities factorize $\Phi^t$ into two parts: a reduction of uncertainty by the observations, and an amplification or decay by the dynamics. For \(t\ge0\), define  the finite-horizon observability Gramian $\mathcal O_t$ and the batch posterior covariance of the initial state $\mathcal C_t(P)$ as 
\begin{align}
  \mathcal O_t
  &:={\sum}_{\ell=0}^{t-1}(A^\ell)^\top H^\top R^{-1}HA^\ell,
  &\mathcal O_0&:=0,
  \label{eq:observability-information-matrix}\\
  \mathcal C_t(P)
  &:=P^{1/2}(I+P^{1/2}\mathcal O_tP^{1/2})^{-1}P^{1/2},
  \label{eq:batch-initial-covariance}
\end{align}
and set the associated closed-loop propagator as 
\begin{equation}
  \mathcal B_0(P):=I,
  \qquad
  \mathcal B_t(P):=B(\Phi^{t-1}(P))\cdots B(P),
  \quad t\ge1.
  \label{eq:closed-loop-product}
\end{equation}

\begin{proposition}[Exact Riccati identities]
\label{prop:exact-riccati-identities}
For every \(P,Q\in\Splus^{d_u}\) and
\(t\ge0\),
\begin{align}
  \Phi^t(P)&=A^t\mathcal C_t(P)(A^t)^\top,
  \label{eq:batch-covariance-formula}\\
  \mathcal B_t(P)&=A^t(I+P\mathcal O_t)^{-1}.
  \label{eq:closed-loop-formula}
\end{align}
Moreover, for every symmetric matrix \(Z\), 
\begin{equation}
\begin{aligned}
  \Phi^t(P)-\Phi^t(Q)
  &=\mathcal B_t(P)(P-Q)\mathcal B_t(Q)^\top,\\
  \mathrm D\Phi^t(P)[Z]
  &=\mathcal B_t(P)Z\mathcal B_t(P)^\top.
\end{aligned}
\label{eq:enkf-riccati-difference-identity}
\end{equation}
In particular,
\begin{align}
  \|\Phi^t(P)-\Phi^t(Q)\|
  &\le \|\mathcal B_t(P)\|\,\|P-Q\|\,\|\mathcal B_t(Q)\|,
  \label{eq:riccati-difference-bound}\\
  \mathrm D\Phi(P)[Z]
  &=B(P)ZB(P)^\top.
  \label{eq:one-step-riccati-derivative}
\end{align}
Here the derivative is taken for the smooth extension of \(\Phi\) to a
neighborhood of \(P\) in the space of symmetric matrices.
\end{proposition}

Notice that $\mathcal C_t(P)$ is precisely $\operatorname{Cov}(u_0 \mid  \mathcal Y_{t-1})$ when $\operatorname{Cov}(u_0)=P$. Congruence by $A^t$ propagates it forward and gives $\Phi^t(P) = \operatorname{Cov} (u_t \mid \mathcal Y_{t-1})$. Thus $\mathcal C_t(P)$ reduces uncertainty by the observations while dynamical amplification or decay enter through $A^t$.

Equation \eqref{eq:enkf-riccati-difference-identity} writes the
difference between two Riccati trajectories as the initial covariance
difference propagated on the left and right by the corresponding closed-loop
products.  Hence, if
\(\max\{\|\mathcal B_t(P)\|,\|\mathcal B_t(Q)\|\}\le C\rho^t\), then
\(\|\Phi^t(P)-\Phi^t(Q)\|\le C^2\rho^{2t}\|P-Q\|\), which is exponential
loss of dependence on the initial covariance.  In our theory, $P$ and $Q$ may be any forecast covariances arising along the iteration. We therefore introduce a class containing all such covariances, and prove that the required bounds hold uniformly there.  

Relative to the splitting $\R^{d_u}=\mathsf U\oplus\mathsf S$, write
\[
  P=
  \begin{pmatrix}
    P_{uu}&P_{us}\\
    P_{su}&P_{ss}
  \end{pmatrix}.
\]
For $\underline p,\overline p>0$, define the covariance class
\begin{equation}
  \mathfrak C(\underline p,\overline p)
  :=
  \left\{
  P\in\Splus^{d_u}:
  P_{uu}\succeq \underline p I_{\mathsf U},\quad
  \|P\|\leq \overline p
  \right\}.
  \label{eq:covariance-class}
\end{equation}
These classes are used only to state estimates uniformly over possibly
singular covariances, and they are not additional model assumptions.  The condition \eqref{eq:block-diagonal-initial-cov-class-sketch} on $\widehat P_0^0$ in Assumption~\ref{ass:uniform-local-riccati} is a local counterpart, which we formally define in \eqref{eq:local-covariance-class}. 

\begin{theorem}[Uniform Riccati and closed-loop estimates]
\label{thm:enkf-riccati-package}
Under Assumption~\ref{ass:dynamics-initialization}, for every
\(\underline p,\overline p>0\), there are constants
\(0<\underline p_*\leq\overline p_*<\infty\), $C<\infty$, and
\(\rho\in(0,1)\), depending additionally on $\underline p$ and $\overline p$,
such that
\begin{enumerate}
\item for every \(P\in\mathfrak C(\underline p,\overline p)\) and
\(t\geq0\),
\begin{equation}
  \Phi^t(P)\in\mathfrak C(\underline p_*,\overline p_*);
  \label{eq:enkf-package-common-class}
\end{equation}
\item for \(P,Q\in\mathfrak C(\underline p,\overline p)\) and
\(t\ge s\ge0\),
\begin{align}
  \|B(\Phi^{t-1}(P))\cdots B(\Phi^s(P))\|
  &\le C\rho^{t-s},
  \label{eq:enkf-package-closed-loop}\\
  \|\Phi^t(P)-\Phi^t(Q)\|
  &\le C\rho^{2(t-s)}
  \|\Phi^s(P)-\Phi^s(Q)\|;
  \label{eq:enkf-package-forgetting}
\end{align}
\item for every \(M<\infty\), there is \(C=C(M)<\infty\) such that, for all
\(P,Q\in\Splus^{d_u}\) satisfying
\(\max\{\|P\|,\|Q\|\}\le M\),
\begin{equation}
  \|\Psi(P)-\Psi(Q)\|
  +\|\mathsf K(P)-\mathsf K(Q)\|
  \le C\|P-Q\|.
  \label{eq:enkf-package-regularity}
\end{equation}
\end{enumerate}
The empty product in \eqref{eq:enkf-package-closed-loop} is the
identity.  The last estimate requires no lower bound on \(P\) or \(Q\);
the positive definiteness of \(R\) provides the required uniform invertibility.
\end{theorem}

The proof is given in  Sections~\ref{app:riccati-stability} and \ref{app:auxiliary-estimates}.  The exponential decay of $\mathcal B_t(P)$ in \eqref{eq:enkf-package-closed-loop} is central to the long-time accuracy of mean and covariance estimates.
We
factor \(P\) into uncertainty on the unstable subspace and an
uncorrelated stable residual. The unstable component can be inferred
from the observations under the finite-window observability condition
\eqref{eq:observability-unstable}, while the stable residual is
propagated solely by the contracting stable dynamics. Together,
observability of the unstable component and dynamical contraction of
the stable residual yield the desired contraction of
\(\mathcal B_t(P)\). Restarting this estimate and using the exact
difference identity in
Proposition~\ref{prop:exact-riccati-identities} gives Riccati
forgetting. The regularity estimate follows from \(R\succ0\) and the
uniform covariance bounds.

We next collect the probabilistic estimates used in the main proofs.  The
centered Gaussian sample covariance has the same distribution as the
covariance formed from \(N-1\) independent Gaussian vectors.  Effective-rank
concentration controls its operator-norm error, while a smallest-singular-value
bound for its unstable projection gives the lower covariance bound needed for
the common class.  The Gaussian sample mean is independent of the centered
anomalies.  In addition, Kalman orthogonality makes the innovations independent
centered Gaussian vectors with uniformly bounded covariances
\(H\widehat P_tH^\top+R\).  These observations give the following result.

\begin{proposition}[Gaussian initialization and innovations]
\label{prop:enkf-probabilistic-package}
Suppose the ensemble is initialized as in \eqref{eq:gaussian-initial-ensemble}.  Under the ensemble-size
condition~\eqref{eq:effective-rank-ensemble-size}, for every
\(\delta\in(0,1)\) there is an event of probability at least
\(1-3\delta/4\) on which
\begin{equation}
  \|\widehat m_0^N-\widehat m_0\|\le\mu_N(\delta),
  \qquad
  \|\widehat P_0^N-\widehat P_0\|\le\Delta_N(\delta),
  \label{eq:enkf-package-initialization}
\end{equation}
and \(\widehat P_0^N\) and \(\widehat P_0\) belong to a common deterministic
covariance class. Under the linear--Gaussian model
\eqref{eq:initial-state-law}--\eqref{eq:observation-model} and
Assumption~\ref{ass:dynamics-initialization}, and when the initial ensemble is
independent of $u_0$ and $(\eta_t)_{t\ge0}$, the innovations are independent
centered Gaussian vectors, independent of the initial ensemble.
Moreover, for every \(\beta\in(0,1)\), there is $C_\beta<\infty$, depending
additionally on $\beta$ but not on $\delta$, such that, with probability at
least \(1-\delta\),
\begin{equation}
  \sum_{t=0}^\infty\beta^t\|\iota_t\|
  \le C_\beta\bigl[1+\sqrt{\log(1/\delta)}\bigr].
  \label{eq:enkf-package-innovations}
\end{equation}
\end{proposition}

The detailed probabilistic estimates are proved in
Sections~\ref{app:gaussian-initialization}
and~\ref{app:auxiliary-estimates}.

\subsubsection{Proof of covariance accuracy}
\label{subsec:proof-covariance-accuracy}

\begin{proof}[Proof of Theorem~\ref{thm:time-uniform-covariance-approximation}]
Proposition~\ref{prop:exact-square-root-identities} shows that the empirical
and exact covariances follow the same deterministic Riccati recursion:
\[
  \widehat P_{t+1}^N=\Phi(\widehat P_t^N),
  \qquad
  \widehat P_{t+1}=\Phi(\widehat P_t),
  \qquad
  P_t^N=\Psi(\widehat P_t^N),
  \quad
  P_t=\Psi(\widehat P_t).
\]
In particular,
\(
  \widehat P_t^N=\Phi^t(\widehat P_0^N)
\)
and
\(
  \widehat P_t=\Phi^t(\widehat P_0)
\).

Now work on the initialization event supplied by
Proposition~\ref{prop:enkf-probabilistic-package}.  By
\eqref{eq:enkf-package-initialization}, the two initial covariances
belong to a common deterministic covariance class and satisfy
\(
  \|\widehat P_0^N-\widehat P_0\|\leq\Delta_N(\delta)
\).
The forgetting estimate~\eqref{eq:enkf-package-forgetting}, applied with
\(s=0\), therefore gives
\[
  \|\widehat P_t^N-\widehat P_t\|
  \le C\rho^{2t}\Delta_N(\delta).
\]
Theorem~\ref{thm:enkf-riccati-package} also keeps both trajectories in a
common bounded covariance class.  Hence its regularity estimate
\eqref{eq:enkf-package-regularity} applies uniformly in time and yields
\[
  \|P_t^N-P_t\|+\|K_t^N-K_t\|
  \leq C\|\widehat P_t^N-\widehat P_t\|.
\]
Combining the last two displays and relabeling \(\rho\) gives
\begin{equation}
  \|\widehat P_t^N-\widehat P_t\|
  +\|P_t^N-P_t\|
  +\|K_t^N-K_t\|
  \le C\rho^t\Delta_N(\delta),
  \qquad t\ge0.
  \label{eq:enkf-covariance-gain-decay}
\end{equation}
The event has probability at least \(1-3\delta/4\ge1-\delta\).  Taking the
supremum proves \eqref{eq:time-uniform-covariance-approximation}.
\end{proof}

\subsubsection{Proof of mean accuracy}
\label{subsec:proof-mean-accuracy}

\begin{proof}[Proof of Theorem~\ref{thm:time-uniform-mean-approximation}]
Recall from Proposition~\ref{prop:exact-square-root-identities} that the
forecast and analysis mean errors satisfy
\begin{align*}
  \widehat e_{t+1}^N
  &=B(\widehat P_t^N)\widehat e_t^N
    +A(K_t^N-K_t)\iota_t,\\
  e_t^N
  &=(I-K_t^NH)\widehat e_t^N+(K_t^N-K_t)\iota_t.
\end{align*}
Work first on the initialization event from
Proposition~\ref{prop:enkf-probabilistic-package}.  On this event,
Theorem~\ref{thm:enkf-riccati-package} keeps the empirical covariance
trajectory in a common class, and
\eqref{eq:enkf-package-closed-loop} gives
\[
  \|B(\widehat P_{t-1}^N)\cdots B(\widehat P_s^N)\|
  \leq C\rho^{t-s},
  \qquad t\geq s.
\]
The covariance proof also gives
\(
  \|K_t^N-K_t\|\leq C\rho^t\Delta_N(\delta)
\)
by \eqref{eq:enkf-covariance-gain-decay}.

Iterating the forecast recursion above yields, for \(t\geq1\),
\begin{equation}
  \widehat e_t^N
  = B(\widehat P_{t-1}^N)\cdots B(\widehat P_0^N)\widehat e_0^N
  +\sum_{\ell=0}^{t-1}
  B(\widehat P_{t-1}^N)\cdots B(\widehat P_{\ell+1}^N)
  A(K_\ell^N-K_\ell)\iota_\ell,
\label{eq:mean-err-iterate-from-initial}
\end{equation}
where the product in the summand is the identity when \(\ell=t-1\).
Applying the closed-loop and gain estimates gives
\[
  \|\widehat e_t^N\|
  \le C\rho^t\|\widehat e_0^N\|
  +C\Delta_N(\delta)
  \sum_{\ell=0}^{t-1}
  \rho^{t-1-\ell}\rho^\ell\|\iota_\ell\|.
\]
Taking \(\rho^{t-1-\ell}\rho^\ell\leq\rho^\ell\) controls the forecast errors. The analysis recursion above and the
uniform boundedness of \(K_t^N\) further give
\[
  \|e_t^N\|
  \leq C\|\widehat e_t^N\|
  +C\rho^t\Delta_N(\delta)\|\iota_t\|.
\]
Consequently,
\begin{equation}
  \sup_{t\ge0}\bigl(\|\widehat e_t^N\|+\|e_t^N\|\bigr)
  \le C\|\widehat e_0^N\|
  +C\Delta_N(\delta)
  \sum_{\ell=0}^\infty\rho^\ell\|\iota_\ell\|.
  \label{eq:enkf-mean-reduction}
\end{equation}
The initialization part of
Proposition~\ref{prop:enkf-probabilistic-package} bounds the first term by
\(C\mu_N(\delta)\) on an event of probability at least
\(1-3\delta/4\).  The innovation part of the same proposition, applied with
failure probability \(\delta/4\) and \(\beta=\rho\), bounds the weighted sum
in \eqref{eq:enkf-mean-reduction} by
\(C[1+\sqrt{\log(4/\delta)}]\).  Since \(\delta<1\), the additive constant is
absorbed by \(\sqrt{\log(4/\delta)}\).  A union bound and
\eqref{eq:enkf-mean-reduction} now yield
\[
  \sup_{t\ge0}\bigl(\|\widehat e_t^N\|+\|e_t^N\|\bigr)
  \le C\left[
  \mu_N(\delta)+\Delta_N(\delta)\sqrt{\log(4/\delta)}
  \right]
\]
with probability at least \(1-\delta\), as claimed.
\end{proof}

\subsection{Localized ensemble Kalman filter}
\label{subsec:localized-proofs}

We first collect the additional estimates required for the localized EnKF proofs. Throughout this section, the
constants in these estimates are uniform in the number of blocks. The detailed statements
and proofs are given in
Sections~\ref{app:localized-identities}--
\ref{app:localized-mean-auxiliary}.
 
To establish the time-uniform accuracy of the covariances, the local empirical covariance $\widehat P_{t,j}^{N,\mathrm{loc}}$ and the full Kalman forecast covariance $\widehat P_t$ are both compared with the block-diagonal Riccati trajectory in \eqref{eq:block-reference-riccati-flow}. The discrepancy between $\widehat P_{t,j}^{N,\mathrm{loc}}$ and $\widehat P_{t,j}^0$ involves the empirical sampling error and the localization bias in the forecast step. We show that the stability of the closed-loop and the Riccati map in Theorem~\ref{thm:enkf-riccati-package} hold blockwise for $\widehat P_t^0$, providing contraction over long time. The overall localization bias incurred by neglecting the cross-block interactions  in the forecast is thus controlled by $O(\varepsilon)$. On the other hand, the deviation of $\widehat P_t$ from $\widehat P_t^0$ is a deterministic localization bias, which is also $O(\varepsilon)$. 
This bias arises from both the analysis and the forecast steps,  and the proof is thus more involved. 
 For the time-uniform accuracy of the mean estimates, the theory centers on an analog of \eqref{eq:mean-err-iterate-from-initial} in the localized algorithm. There, the controlled 
discrepancy between $\widehat P_t$ and $\widehat P_t^0$, along with controlled localization bias and sampling error in the local maps, leads to a bounded gain difference. 

The block-decay norm introduced in \eqref{eq:block-decay-norm} defines a dimension-independent matrix algebra, and controls the block-local norm defined in \eqref{eq:block-local-norm}:
\begin{equation}
  \|M_1M_2\|_{\mathcal J_\alpha}
  \le C\|M_1\|_{\mathcal J_\alpha}\|M_2\|_{\mathcal J_\alpha},
  \qquad
  \|M\|_{\mathrm{loc}}\le C\|M\|_{\mathcal J_\alpha},
  \label{eq:localized-package-block-algebra}
\end{equation}
as proved in Lemma~\ref{lem:block-decay-algebra}.

For \(\underline p,\overline p>0\), define the local covariance class
\begin{equation}
  \mathfrak C_j(\underline p,\overline p)
  :=\left\{P_j\in\Splus^{d_{u,j}}:
  (P_j)_{uu}\succeq \underline p I_{\mathsf U_j},\ \|P_j\|\le \overline p\right\}.
  \label{eq:local-covariance-class}
\end{equation}
These classes are the local
analogs of the global covariance classes in \eqref{eq:covariance-class} and allow for estimates over possibly singular local covariances. The classes are not imposed as additional model assumptions. For a local covariance \(P_j\),  write the local forecast closed-loop map 
\begin{equation}
  B_j(P_j):=A_j\bigl(I-\mathsf K_j(P_j)H_j\bigr).
  \label{eq:local-closed-loop-map}
\end{equation}

  The next result has two parts. The first gives stability of the local forecast closed-loop map and of the local Riccati maps. 
  The second shows the block-diagonal reference trajectory stays within $O(\varepsilon)$ of the fully-coupled Riccati trajectory, uniformly in time.  The second part leverages the persistence of these stability bounds under small perturbations of block-diagonality, as detailed in Section~\ref{sec:localized-deterministic-bias}.  
\begin{theorem}[Local Riccati stability and localization bias]
\label{thm:localized-deterministic-package}
Under Assumptions~\ref{ass:localized-structure} and
\ref{ass:uniform-local-riccati}, there are
\(0<\underline p_{\mathrm{loc}}<\overline p_{\mathrm{loc}}<\infty\),
independent of \(j\) and \(J\), such that the reference local Riccati
trajectories lie in
\(\mathfrak C_j(\underline p_{\mathrm{loc}},\overline p_{\mathrm{loc}})\).
For every
\(P_j,Q_j\in
\mathfrak C_j(\underline p_{\mathrm{loc}},\overline p_{\mathrm{loc}})\)
and all integers \(t\ge s\ge0\), uniformly in the block index,
\begin{align}
  \|B_j(\Phi_j^{t-1}(P_j))\cdots B_j(\Phi_j^s(P_j))\|
  &\le C\rho^{t-s},
  \label{eq:localized-package-closed-loop}\\
  \|\Phi_j^t(P_j)-\Phi_j^t(Q_j)\|
  &\le C\rho^{2(t-s)}
  \|\Phi_j^s(P_j)-\Phi_j^s(Q_j)\|.
  \label{eq:localized-package-forgetting}
\end{align}
The product in \eqref{eq:localized-package-closed-loop} is the
identity when \(t=s\).  If \(\varepsilon\) is sufficiently small and
\(\widehat P_0=\widehat P_0^0\), then for every $t\geq 0$, 
\begin{equation}
  \|\widehat P_t-\widehat P_t^0\|_{\mathrm{loc}}
  \le C\varepsilon,
   \qquad 
  \|\widehat P_t-\widehat P_t^0\|_{\mathcal J_\alpha}
  \le C\varepsilon. \label{eq:localized-package-local-bias}
\end{equation}
The corresponding analysis-covariance and gain differences obey the same
bounds in the respective norms.
\end{theorem}
 The bounds \eqref{eq:localized-package-closed-loop} and \eqref{eq:localized-package-forgetting} follow by applying Theorem \ref{thm:enkf-riccati-package} to each of the blocks of the reference trajectory, each of which evolves independently. This is shown in Section~\ref{sec:localized-local-riccati}. The bounds in \eqref{eq:localized-package-local-bias} are shown in 
Section~\ref{sec:localized-deterministic-bias}, which also gives the more general
estimates for nearby, possibly non-block-diagonal initial covariances,
including their geometrically decaying transient terms.

We now turn to the empirical local covariances.  The local square-root
transforms give the exact diagonal analysis identities and the perturbed
forecast recursion
\begin{equation}
  P_{t,j}^{N,\mathrm{loc}}
  =\Psi_j(\widehat P_{t,j}^{N,\mathrm{loc}}),
  \qquad
  \widehat P_{t+1,j}^{N,\mathrm{loc}}
  =\Phi_j(\widehat P_{t,j}^{N,\mathrm{loc}})+\zeta_{t,j}^N,
  \quad
  \max_j\|\zeta_{t,j}^N\|\le C\varepsilon.
  \label{eq:localized-package-perturbed-recursion}
\end{equation}
 These identities are shown in Section~\ref{app:localized-identities}. 
  Note that the full empirical cross-block covariances need not be small since block Cauchy--Schwarz only provides upper bounds by the diagonal variances, but weak interaction
 limits their contribution to each local forecast covariance to
$O(\varepsilon)$.
To control the sampling error in the ensemble initialization, a union bound over the $J$  local Gaussian covariance
estimates gives,
\[
  \max_j\|\widehat P_{0,j}^{N,\mathrm{loc}}-\widehat P_{0,j}\|
  \le\Delta_{N,\mathrm{loc}}(\delta).
\]  This argument is detailed in Section~\ref{sec:localized-gaussian-initialization}. 
Combining these estimates with the preceding finite-step argument gives the
following proposition,  which uniformly bounds the error between the local empirical covariances and the block-diagonal reference covariance trajectory, on a high probability event over the ensemble initialization.

\begin{proposition}[Localized covariance initialization and stability]
\label{prop:localized-covariance-package}
Suppose that \(\widehat P_0=\widehat P_0^0\), initialize the localized
ensemble as in \eqref{eq:localized-gaussian-initial-ensemble}, and
assume the small-interaction and ensemble-size conditions of
Theorem~\ref{thm:B6-localized-covariance-accuracy}.  For every
\(\delta\in(0,1)\), there is an event 
\(\mathcal G_{N,\delta}^{\mathrm{cov,loc}} \in \sigma(\widehat X_0^{\mathrm{loc}})\),  
of probability at least \(1-\delta\), on
which the local empirical covariances remain in a common stability class and
\begin{align}
  \max_j\|\widehat P_{t,j}^{N,\mathrm{loc}}-
  \widehat P_{t,j}^0\|
  &\le C\rho^t\Delta_{N,\mathrm{loc}}(\delta)+C\varepsilon,
  \label{eq:localized-package-forecast-covariance}\\
  \max_j\|P_{t,j}^{N,\mathrm{loc}}-
  P_{t,j}^0\|
  +\max_j\|K_{t,j}^{N,\mathrm{loc}}-
   K_{t,j}^0\|
  &\le C\rho^t\Delta_{N,\mathrm{loc}}(\delta)+C\varepsilon.
  \label{eq:localized-package-analysis-gain}
\end{align}
\end{proposition}
See Section \ref{sec:localized-uniform-empirical-covariance-control} for the proof of this statement.  
For the mean estimates, let
\(\mathcal F_0^{X,N}:=\sigma(\widehat X_0^{\mathrm{loc}})\) and define the local initial-mean and innovation moment scales 
\begin{align}
  \mu_{N,\mathrm{loc}}^{(q)}
  &:={\left(\frac{\max_j\operatorname{Tr}(\widehat P_{0,j})}{N}\right)^{1/2}}
  +{\left(\frac{\max_j\|\widehat P_{0,j}\|
  (q+\log(2J))}{N}\right)^{1/2}},
  \label{eq:localized-package-initial-scale}\\
  \mathcal I_{q,\mathrm{loc}}
  &:=\sqrt{\sup_t\max_j\operatorname{Tr}((S_t)_{jj})}
  +\sqrt{\sup_t\max_j\|(S_t)_{jj}\|(q+\log(2J))},
  \label{eq:localized-package-innovation-scale}
\end{align}
where \(S_t:=H\widehat P_tH^\top+R\).  For a random block vector \(Z\), set  
\begin{equation}
  \|Z\|_{q\mid X}
  :=\left(\mathbb E[\|Z\|_{\max}^q\mid
  \mathcal F_0^{X,N}]\right)^{1/q}.
\label{eq:def-random-block-vec-norm}
\end{equation}

To obtain the mean accuracy, an error decomposition analogous to \eqref{eq:mean-err-iterate-from-initial} requires bounded gain error $K_t^{N,\mathrm{loc}} - K_t$ and closed-loop propagator contraction. The gain error splits into a sampling error and a localization bias. The sampling error comes from the discrepancy between two uncoupled local reference trajectories, initialized at the exact and empirical local covariances. The localization bias measures the deviation from the reference trajectories. The controlled gain error also ensures that the closed-loop maps are small perturbations of their uncoupled counterparts, and therefore the closed-loop propagator remains exponentially stable. Besides, Gaussian block-maximum estimates control the initial sample mean and the innovations.

\begin{proposition}[Localized mean estimates]
\label{prop:localized-mean-package}
Under the conditions of Proposition~\ref{prop:localized-covariance-package}
and the linear--Gaussian model
\eqref{eq:initial-state-law}--\eqref{eq:observation-model}, assume in addition
that the initial ensemble is independent of 
$u_0$ and $(\eta_t)_{t\geq 0}$.
Then, $\widehat m_0^{N,\mathrm{loc}} - \widehat m_0$ and $(\iota_t)_{t\geq 0}$ are independent of $\mathcal F_0^{X,N}$. The independence makes their unconditional and conditional expectations identical, and therefore, for every \(q\ge1\), on
\(\mathcal G_{N,\delta}^{\mathrm{cov,loc}}\),
\begin{equation}
  \|\widehat m_0^{N,\mathrm{loc}}-\widehat m_0\|_{q\mid X}
  \le C\mu_{N,\mathrm{loc}}^{(q)},
  \qquad 
  \sup_{t\ge0}\|\iota_t\|_{q\mid X}
  \le C\mathcal I_{q,\mathrm{loc}}.
  \label{eq:localized-package-random-inputs}
\end{equation}

Moreover,  on
\(\mathcal G_{N,\delta}^{\mathrm{cov,loc}}\), the gain error admits a decomposition
\begin{equation}
  K_t^{N,\mathrm{loc}}-K_t
  =\Gamma_t^{\mathrm{sam}}+\Gamma_t^{\mathrm{loc}},
  \qquad
  \|\Gamma_t^{\mathrm{sam}}\|_{\mathrm{loc}}
  \le C\rho^t\Delta_{N,\mathrm{loc}}(\delta),
  \quad
  \|\Gamma_t^{\mathrm{loc}}\|_{\mathrm{loc}}
  \le C\varepsilon,
  \label{eq:localized-package-gain-decomposition}
\end{equation}
and, for \(B_t^{N,\mathrm{loc}}:=A(I-K_t^{N,\mathrm{loc}}H)\),
\begin{equation}
  \|B_{t-1}^{N,\mathrm{loc}}\cdots B_s^{N,\mathrm{loc}}\|_{\mathrm{loc}}
  \le C\rho^{t-s},
  \qquad t\ge s\ge0.
  \label{eq:localized-package-mean-stability}
\end{equation}
The product is the identity when \(t=s\).
\end{proposition}
The proofs of these estimates are given in Section~\ref{app:localized-mean-auxiliary}. 
We now apply these ingredients to prove our main accuracy results for the localized EnKF. The
covariance proof separates sampling error from localization bias, while the
mean proof uses the same decomposition inside a stable random recursion.

\subsubsection{Proof of localized covariance accuracy}

\begin{proof}[Proof of Theorem~\ref{thm:B6-localized-covariance-accuracy}]
Recall \(P_{t,j}^0:=\Psi_j(\widehat P_{t,j}^0)\). The forecast and analysis
covariance errors decompose through the block-diagonal reference trajectory as
\begin{align*}
  \widehat P_{t,j}^{N,\mathrm{loc}}-\widehat P_{t,j}
  &=(\widehat P_{t,j}^{N,\mathrm{loc}}-\widehat P_{t,j}^0)
    +(\widehat P_{t,j}^0-\widehat P_{t,j}),\\
  P_{t,j}^{N,\mathrm{loc}}-P_{t,j}
  &=(P_{t,j}^{N,\mathrm{loc}}-P_{t,j}^0)
    +(P_{t,j}^0-P_{t,j}).
\end{align*}
We control the first term in each line by empirical stability and the second
by deterministic localization bias.

Work on the event \(\mathcal G_{N,\delta}^{\mathrm{cov,loc}}\) supplied by
Proposition~\ref{prop:localized-covariance-package}.  Its forecast estimate
\eqref{eq:localized-package-forecast-covariance} gives
\[
  \max_j\|\widehat P_{t,j}^{N,\mathrm{loc}}-\widehat P_{t,j}^0\|
  \leq C\rho^t\Delta_{N,\mathrm{loc}}(\delta)+C\varepsilon.
\]
Because \(\widehat P_0=\widehat P_0^0\), the deterministic bias estimate
\eqref{eq:localized-package-local-bias} in
Theorem~\ref{thm:localized-deterministic-package} 
yields
\[
  \max_j\|\widehat P_{t,j}^0-\widehat P_{t,j}\|
  \leq C\varepsilon.
\]
The first decomposition and the triangle inequality therefore give
\[
  \max_j\|\widehat P_{t,j}^{N,\mathrm{loc}}-\widehat P_{t,j}\|
  \leq C\rho^t\Delta_{N,\mathrm{loc}}(\delta)+C\varepsilon.
\]

For the analysis covariance,
\eqref{eq:localized-package-analysis-gain} and the corresponding
conclusion of Theorem~\ref{thm:localized-deterministic-package} give
\begin{align*}
  \max_j\|P_{t,j}^{N,\mathrm{loc}}-P_{t,j}^0\|
  &\le C\rho^t\Delta_{N,\mathrm{loc}}(\delta)+C\varepsilon,\\
  \max_j\|P_{t,j}^0-P_{t,j}\|
  &\le C\varepsilon.
\end{align*}
Thus the second decomposition implies
\[
  \max_j\|P_{t,j}^{N,\mathrm{loc}}-P_{t,j}\|
  \leq C\rho^t\Delta_{N,\mathrm{loc}}(\delta)+C\varepsilon.
\]
The event has probability at least \(1-\delta\).  Adding the forecast and
analysis bounds and taking the supremum over \(t\geq0\) proves \eqref{eq:B6-time-uniform-covariance-accuracy}.
\end{proof}

\subsubsection{Proof of localized mean accuracy}

\begin{proof}[Proof of Theorem~\ref{thm:B7-localized-mean-accuracy}]
Set
\[
  \widehat e_t^{N,\mathrm{loc}}
  :=\widehat m_t^{N,\mathrm{loc}}-\widehat m_t,
  \qquad
  e_t^{N,\mathrm{loc}}:=m_t^{N,\mathrm{loc}}-m_t.
\]
Analogous to \eqref{eq:main-exact-enkf-recursions}, subtracting the exact Kalman mean updates from the localized ensemble updates
gives
\begin{align}
  e_t^{N,\mathrm{loc}}
  &=(I-K_t^{N,\mathrm{loc}}H)\widehat e_t^{N,\mathrm{loc}}
  +(K_t^{N,\mathrm{loc}}-K_t)\iota_t,
  \label{eq:B7-localized-analysis-mean-error}\\
  \widehat e_{t+1}^{N,\mathrm{loc}}
  &=B_t^{N,\mathrm{loc}}\widehat e_t^{N,\mathrm{loc}}
  +A(K_t^{N,\mathrm{loc}}-K_t)\iota_t,
  \label{eq:B7-localized-forecast-mean-error}
\end{align}
where \(B_t^{N,\mathrm{loc}}=A(I-K_t^{N,\mathrm{loc}}H)\).

Proposition~\ref{prop:localized-exact-covariance-identities} shows that the
empirical gains and closed-loop factors are $\mathcal F_0^{X,N}$-measurable.
Condition on \(\mathcal F_0^{X,N}\) and work on
\(\mathcal G_{N,\delta}^{\mathrm{cov,loc}}\).  Iterating the forecast recursion
in \eqref{eq:B7-localized-forecast-mean-error} yields, for \(t\geq1\),
\begin{equation}
  \widehat e_t^{N,\mathrm{loc}}
  ={}B_{t-1}^{N,\mathrm{loc}}\cdots B_0^{N,\mathrm{loc}}
  \widehat e_0^{N,\mathrm{loc}}
  +\sum_{s=0}^{t-1}
  B_{t-1}^{N,\mathrm{loc}}\cdots B_{s+1}^{N,\mathrm{loc}}
  A(K_s^{N,\mathrm{loc}}-K_s)\iota_s,
\label{eq:mean-err-iterate-from-initial-localized} 
\end{equation}
where the product in the summand is the identity when \(s=t-1\).  This parallels \eqref{eq:mean-err-iterate-from-initial}.

We now invoke Proposition~\ref{prop:localized-mean-package} following the
structure of this recursion.  Recall that \eqref{eq:localized-package-mean-stability}
controls the propagators, while
\eqref{eq:localized-package-gain-decomposition} splits each 
gain error
into a geometrically decaying sampling contribution and a persistent
\(O(\varepsilon)\) localization contribution.  Finally,
\eqref{eq:localized-package-random-inputs} controls the conditional
moments of the initial error and innovations.  The block-local norm
action~\eqref{eq:B7-local-norm-action} and conditional Minkowski therefore give
\begin{equation*}
  \|\widehat e_t^{N,\mathrm{loc}}\|_{q\mid X}
  \le C\rho^t\mu_{N,\mathrm{loc}}^{(q)}
  +C\mathcal I_{q,\mathrm{loc}}
  \sum_{s=0}^{t-1}\rho^{t-1-s}
  \left[\rho^s\Delta_{N,\mathrm{loc}}(\delta)+\varepsilon\right].
\end{equation*}
For \(t\ge1\), the sampling convolution equals
\(t\rho^{t-1}\Delta_{N,\mathrm{loc}}(\delta)\) and is therefore bounded by
\(C\Delta_{N,\mathrm{loc}}(\delta)\); the localization convolution is bounded
by \(C\varepsilon\).  Hence
\begin{equation}
  \sup_{t\ge0}\|\widehat e_t^{N,\mathrm{loc}}\|_{q\mid X}
  \le C\mu_{N,\mathrm{loc}}^{(q)}
  +C\bigl[\Delta_{N,\mathrm{loc}}(\delta)+\varepsilon\bigr]
  \mathcal I_{q,\mathrm{loc}}.
  \label{eq:B7-forecast-mean-moment-bound}
\end{equation}
For the analysis error, the identity \eqref{eq:B7-localized-analysis-mean-error},
the uniform gain bounds, the gain decomposition
\eqref{eq:localized-package-gain-decomposition}, and the innovation estimate
\eqref{eq:localized-package-random-inputs} give
\[
  \|e_t^{N,\mathrm{loc}}\|_{q\mid X}
  \leq C\|\widehat e_t^{N,\mathrm{loc}}\|_{q\mid X}
  +C\left[\rho^t\Delta_{N,\mathrm{loc}}(\delta)+\varepsilon\right]
  \mathcal I_{q,\mathrm{loc}}.
\]
Combining this estimate with
\eqref{eq:B7-forecast-mean-moment-bound} yields
\begin{equation}
  \sup_{t\ge0}\Bigl(
  \|\widehat m_t^{N,\mathrm{loc}}-\widehat m_t\|_{q\mid X}
  +\|m_t^{N,\mathrm{loc}}-m_t\|_{q\mid X}
  \Bigr)
  \le C\mu_{N,\mathrm{loc}}^{(q)}
  +C\bigl[\Delta_{N,\mathrm{loc}}(\delta)+\varepsilon\bigr]
  \mathcal I_{q,\mathrm{loc}}.
  \label{eq:B7-time-uniform-localized-mean-detailed}
\end{equation}
Finally, the uniform local dimension and covariance bounds imply
\[
  \mu_{N,\mathrm{loc}}^{(q)}
  \le C N^{-1/2}\sqrt{q+\log(2J)},
  \qquad
  \mathcal I_{q,\mathrm{loc}}
  \le C\sqrt{q+\log(2J)}.
\]
Substitution proves \eqref{eq:B7-time-uniform-localized-mean}.
\end{proof}

\section*{Acknowledgments}
DSA was partly funded by NSF CAREER award
DMS-2237628. The authors used OpenAI Codex to assist with proof auditing and grammatical editing.
Claude was also used for grammatical editing. The authors assume full responsibility for all content.

\bibliographystyle{siamplain}
\bibliography{references}

\appendix

\section{Ensemble Kalman filter}
\label{app:enkf}

This supplement collects the technical results used in the EnKF proofs in
Section~\ref{subsec:enkf-proofs}.  
Section~\ref{app:proof-aux-propositions} establishes exact EnKF and Riccati identities,
Section~\ref{app:riccati-stability}
develops Riccati stability and forgetting,
Section~\ref{app:gaussian-initialization} treats Gaussian initialization, and
Section~\ref{app:auxiliary-estimates} records covariance, gain, and innovation
estimates.  The Gaussian reduction in
Section~\ref{app:gaussian-initialization} is also used in
Section~\ref{app:localized-enkf}.

Throughout the supplement, $C<\infty$ and $\rho\in(0,1)$ denote constants
that may change from line to line and depend only on the fixed model parameters
and the indicated covariance-class endpoints. In localized statements, these
constants are independent of $J$ and $N$ unless stated otherwise.

\subsection{Exact identities}\label{app:proof-aux-propositions}

Here we prove Propositions~\ref{prop:exact-square-root-identities} and \ref{prop:exact-riccati-identities}.

\begin{proof}[Proof of Proposition~\ref{prop:exact-square-root-identities} ]
    The matrix added to \(I_N\) in \eqref{eq:symmetric-transform} is
positive semidefinite, so its principal inverse square root exists.  Since
\(\widehat X_t\one=0\), the matrix under that square root fixes \(\one\),
and therefore \(T_t\one=\one\) and \(X_t\one=0\).  Woodbury's identity gives
\[
  T_tT_t^\top
  =I_N-\frac1{N-1}\widehat X_t^\top H^\top
  (H\widehat P_t^NH^\top+R)^{-1}H\widehat X_t.
\]
It follows directly that
\[
   P_t^N =  \frac{X_tX_t^\top}{N-1}
  =\widehat P_t^N-\widehat P_t^NH^\top
  (H\widehat P_t^NH^\top+R)^{-1}H\widehat P_t^N
  =\Psi(\widehat P_t^N).
\]
The forecast relation \(\widehat X_{t+1}=AX_t\) proves the identity for $\widehat P_{t+1}^N$.  Finally, subtracting the exact and ensemble analysis-mean
updates, and then forecasting, proves
\eqref{eq:main-exact-enkf-recursions}.
\end{proof}

\begin{proof}[Proof of Proposition~\ref{prop:exact-riccati-identities} ]
    Since $P, Q, \mathcal O_t$ are positive semidefinite,  Sylvester's determinant identity shows that \(I+P\mathcal O_t\) and
\(I+\mathcal O_tP\) are invertible.  The push-through identity then gives
\begin{equation}
  \mathcal C_t(P)=(I+P\mathcal O_t)^{-1}P
  =P(I+\mathcal O_tP)^{-1}\succeq0.
\label{eq:Ct-equiv-expression}
\end{equation}
Set \(G_t=HA^t\).  From
\(\mathcal O_{t+1}=\mathcal O_t+G_t^\top R^{-1}G_t\), Woodbury's identity
gives
\[
  \mathcal C_{t+1}(P)
  =\mathcal C_t(P)-\mathcal C_t(P)G_t^\top
  (R+G_t\mathcal C_t(P)G_t^\top)^{-1}G_t\mathcal C_t(P).
\]
Congruence by \(A^{t+1}\) proves the first formula by induction.  The same
update implies
\[
  (I-\mathsf K(\Phi^t(P))H)A^t(I+P\mathcal O_t)^{-1}
  =A^t(I+P\mathcal O_{t+1})^{-1},
\]
which proves the closed-loop formula.  Finally, the elementary identity
\[
  (I+P\mathcal O_t)^{-1}P-Q(I+\mathcal O_tQ)^{-1}
  =(I+P\mathcal O_t)^{-1}(P-Q)(I+\mathcal O_tQ)^{-1}
\]
gives the difference formula after congruence by \(A^t\).  Taking operator
norms proves the norm bound.  Direct differentiation also gives
\[
  \mathrm D\Psi(P)[Z]
  =(I-\mathsf K(P)H)Z(I-\mathsf K(P)H)^\top.
\]
Congruence by \(A\) and the chain rule along the Riccati trajectory prove
the derivative identities.
\end{proof}

\subsection{Riccati stability and forgetting}
\label{app:riccati-stability}

Recall the covariance class $\mathfrak C(\underline p,\overline p)$ from
\eqref{eq:covariance-class}. The results below provide the detailed
estimates underlying the uniform bounds, stability, and forgetting statements
in Section~\ref{subsec:enkf-proofs}.

Riccati contraction on the positive-definite cone is classical; see, for
example, \cite{Bougerol1993}. Perfect-model filters with degenerate initial
covariances are studied in \cite{BocquetEtAl2017}. Neither result directly
provides the uniform Euclidean estimates over the positive-semidefinite class needed here, so we retain the argument, while reusing its common factorization and information matrices throughout.

 Our analysis centers on the exponential decay of $\|\mathcal B_t(P)\|$ in Theorem~\ref{thm:enkf-riccati-package}. We rewrite $\mathcal B_t(P)$ in coordinates that separate the uncertainty encoded by $P$ into an unstable component and an uncorrelated stable residual.
Fix $P\in\mathfrak C(\underline p,\overline p)$. Define \begin{equation}
  F:=P_{su}P_{uu}^{-1},
  \qquad
  S_c:=P_{ss}-P_{su}P_{uu}^{-1}P_{us}.
  \label{eq:F-Schur}
\end{equation}
Then we have the decomposition
\begin{equation}
  P
  =
  L_F
 \mathsf{D}_P 
  L_F^\top,
  \qquad
  L_F:=
  \begin{pmatrix}I&0\\F&I\end{pmatrix},\quad  \mathsf{D}_P:=\begin{pmatrix}P_{uu}&0\\0&S_c\end{pmatrix}.
  \label{eq:covariance-LDL}
\end{equation}
Under these coordinates, define the transformed system
\begin{equation}
  A_F:=L_F^{-1}AL_F
  =
  \begin{pmatrix}
    A_u&0\\
    A_sF-FA_u&A_s
  \end{pmatrix},
  \qquad
  H_F:=HL_F=\begin{pmatrix}H_u+H_sF&H_s\end{pmatrix}.
  \label{eq:transformed-system}
\end{equation}
 Let $\mathcal O_{F,t}=L_F^\top\mathcal O_tL_F$ be the observability
information matrix for $(A_F,H_F)$. By Proposition~\ref{prop:exact-riccati-identities},
\begin{equation}
  \mathcal B_t(P)
  = A^t(I+P\mathcal O_t)^{-1}=
  L_F A_F^t(I+\mathsf D_P\mathcal O_{F,t})^{-1}L_F^{-1}.
  \label{eq:transformed-closed-loop}
\end{equation}
Write $\mathcal M_t := I + \mathsf D_P \mathcal O_{F,t}$. We will establish that $\|A_F^t \mathcal M_t^{-1}\|$ decays exponentially in time, through observability on the unstable component and dynamical decay on the stable residual.

\subsubsection{Preliminary estimates for the spectral splitting}

\begin{lemma}[Stable and backward-stable powers]
\label{lem:power-bounds}
There are constants $C_s,C_u<\infty$ and
$\gamma_s,\gamma_u\in(0,1)$ such that
\begin{equation}
  \|A_s^t\|\leq C_s\gamma_s^t,
  \qquad
  \|A_u^{-t}\|\leq C_u\gamma_u^t,
  \qquad t\geq0.
  \label{eq:power-bounds}
\end{equation}
\end{lemma}

\begin{proof}
The bounds in \eqref{eq:power-bounds} follow from
$\rho(A_s)<1$ and $\rho(A_u^{-1})<1$: choose
$\gamma_s\in(\rho(A_s),1)$ and
$\gamma_u\in(\rho(A_u^{-1}),1)$ and use the standard spectral-radius
bound for matrix powers.
\end{proof}

\begin{lemma}[Uniform covariance factorization]
\label{lem:covariance-factorization}
Let $P\in\mathfrak C(\underline p,\overline p)$ as in 
\eqref{eq:covariance-class}. 
Then $S_c$ defined in \eqref{eq:F-Schur} satisfies $S_c \succeq0$. 
Moreover, there is $C_{\underline p,\overline p}<\infty$ such that, uniformly over
$P\in\mathfrak C(\underline p,\overline p)$,
\begin{equation}
  \|F\|+\|L_F\|+\|L_F^{-1}\|+\|S_c\|
  \leq C_{\underline p,\overline p}.
  \label{eq:factor-uniform-bounds}
\end{equation}
\end{lemma}
\begin{proof}
The Schur complement $S_c$ is positive semidefinite because
$P\succeq0$ and $P_{uu}\succ0$. Also,
\[
  \|P_{uu}^{-1}\|\leq \underline p^{-1},
  \qquad
  \|P_{us}\|\leq\|P\|\leq \overline p,
\]
so $\|F\|\leq \overline p/\underline p$. The remaining bounds follow from the explicit
triangular form of $L_F$ and from
$0\preceq S_c\preceq P_{ss}\preceq \overline p I$.
\end{proof}

Write
\begin{equation}
  \mathcal O_{F,t}
  =
  \begin{pmatrix}
    O_{uu,t}&O_{us,t}\\
    O_{su,t}&O_{ss,t}
  \end{pmatrix}.
  \label{eq:transformed-Gramian-blocks}
\end{equation}
It follows that
\begin{equation}
  \mathcal M_t=I+\mathsf D_P\mathcal O_{F,t}
  =
  \begin{pmatrix}
    I+P_{uu}O_{uu,t}&P_{uu}O_{us,t}\\
    S_cO_{su,t}& \Lambda_{s,t}
  \end{pmatrix}, \qquad \Lambda_{s,t}:= I + S_c O_{ss,t}.
\label{eq:def-Lambda-st}
\end{equation}
\begin{lemma}[Matrix inverse]
\label{lem:inverse-block-M_t}
The inverse of $\mathcal M_t$ is given by
\begin{equation}
  \mathcal M_t^{-1}
  =
  \begin{pmatrix}
    X_{11,t}&X_{12,t}\\
    X_{21,t}&X_{22,t}
  \end{pmatrix},
  \label{eq:calM-inverse-blocks}
\end{equation}
where, with
\begin{equation}
  C_{s,t}:=\Lambda_{s,t}^{-1}S_c, \quad
  Q_t
  :=
  O_{uu,t}-O_{us,t}C_{s,t}O_{su,t}, \quad  \Theta_t:=P_{uu}^{-1}+Q_t, 
  \label{eq:effective-Q}
\end{equation}
one has
\begin{align}
  X_{11,t}
  &=\Theta_t^{-1}P_{uu}^{-1},
  \label{eq:X11}
  \\
  X_{12,t}
  &=-\Theta_t^{-1}O_{us,t}\Lambda_{s,t}^{-1},
  \label{eq:X12}
  \\
  X_{21,t}
  &=-C_{s,t}O_{su,t}X_{11,t},
  \label{eq:X21}
  \\
  X_{22,t}
  &=\Lambda_{s,t}^{-1}
  +C_{s,t}O_{su,t}\Theta_t^{-1}O_{us,t}\Lambda_{s,t}^{-1}.
  \label{eq:X22}
\end{align}
\end{lemma}
\begin{proof}
The matrix
\(\Lambda_{s,t}=I+S_cO_{ss,t}\) is invertible. Indeed,
\(S_c\) and \(O_{ss,t}\) are positive semidefinite, so Sylvester's
determinant identity gives
\[
  \det(I+S_cO_{ss,t})
  =
  \det(I+S_c^{1/2}O_{ss,t}S_c^{1/2})>0.
\]
Similarly,
\[
  \det(\mathcal M_t)
  =
  \det\bigl(
    I+\mathsf D_P^{1/2}
    \mathcal O_{F,t}\mathsf D_P^{1/2}
  \bigr)>0,
\]
so \(\mathcal M_t\) is invertible. Since both
\(\mathcal M_t\) and \(\Lambda_{s,t}\) are invertible, the Schur
complement of \(\Lambda_{s,t}\),
\[
  I+P_{uu}Q_t
  =
  P_{uu}\Theta_t,
\]
is also invertible. Since \(P_{uu}\) is invertible, so is
\(\Theta_t\). Block Gaussian elimination now gives
\eqref{eq:X11}--\eqref{eq:X22}.
\end{proof}

We now provide estimates for the components of $A_F^t \mathcal M_t^{-1}$.
\begin{lemma}[Stable information and normalized cross-information]
\label{lem:information-block-bounds}
There are constants $C_o<\infty$ and $\gamma\in(0,1)$, depending only on
$A,H,R,\underline p,\overline p$, such that, uniformly over
$P\in\mathfrak C(\underline p,\overline p)$ and $t\geq0$,
\begin{equation}
  \|O_{ss,t}\|\leq C_o,
  \label{eq:Oss-bound}
\end{equation}
and
\begin{equation}
  \|(A_u^{-t})^\top O_{us,t}\|
  \leq C_o\gamma^t.
  \label{eq:normalized-Ous-bound}
\end{equation}
\end{lemma}

\begin{proof}
Recall $\mathcal O_{F,t} = L_F^\top \mathcal O_t L_F = \sum_{\ell=0}^{t-1}(A_F^{\ell})^\top H_F^\top R^{-1} H_F A_F^{\ell}$ and the block form of $\mathcal O_{F,t}$ in \eqref{eq:transformed-Gramian-blocks}. For a transformed initial state $(u,s)$, one has
\begin{equation}
  H_FA_F^{\ell}
  \begin{pmatrix}u\\0\end{pmatrix}
  =H_uA_u^{\ell} u+H_sA_s^{\ell} Fu,
  \qquad 
  H_FA_F^{\ell}
  \begin{pmatrix}0\\s\end{pmatrix}
  =H_sA_s^{\ell} s.
\end{equation}
Consequently,
\begin{align}
  O_{ss,t}
  &=
  \sum_{\ell=0}^{t-1}
  (A_s^{\ell})^\top H_s^\top R^{-1}H_sA_s^{\ell},
  \label{eq:Oss-explicit}
  \\
  O_{us,t}
  &=
  \sum_{\ell=0}^{t-1}
  \bigl(H_uA_u^{\ell}+H_sA_s^{\ell} F\bigr)^\top
  R^{-1}H_sA_s^{\ell} = \left(\sum_{\ell=0}^{t-1}
   (A_u^{\ell})^\top H_u^\top R^{-1}H_sA_s^{\ell}\right) + F^\top O_{ss,t}.
  \label{eq:Ous-explicit}
\end{align}
The geometric bound on $A_s^{\ell}$ in Lemma~\ref{lem:power-bounds} proves \eqref{eq:Oss-bound}.

For the first contribution in \eqref{eq:Ous-explicit},
\begin{equation*}
  (A_u^{-t})^\top
  \sum_{\ell=0}^{t-1}
  (A_u^{\ell})^\top H_u^\top R^{-1}H_sA_s^{\ell}
  =
  \sum_{\ell=0}^{t-1}
  (A_u^{-(t-\ell)})^\top
  H_u^\top R^{-1}H_sA_s^{\ell}.
\end{equation*}
Its norm is bounded by
\[
  C\sum_{\ell=0}^{t-1}
  \gamma_u^{t-\ell}\gamma_s^{\ell}.
\]
For any $\gamma\in(\max\{\gamma_u,\gamma_s\},1)$, this convolution is at
most $C_\gamma\gamma^t$. The second contribution satisfies
\begin{align*}
  &\left\|
  (A_u^{-t})^\top F^\top
  O_{ss,t}
  \right\|
  \leq C\gamma_u^t,
\end{align*}
using \eqref{eq:factor-uniform-bounds} and \eqref{eq:Oss-bound}.
After increasing $C_o$, this proves \eqref{eq:normalized-Ous-bound}.
\nc
\end{proof}

\begin{lemma}[Effective unstable information]
\label{lem:effective-unstable-information}
The matrices $C_{s,t}$ and $Q_t$ defined in \eqref{eq:effective-Q} are symmetric positive semidefinite, and with $\Lambda_{s,t}$ as in \eqref{eq:def-Lambda-st},
\begin{equation}
  \|\Lambda_{s,t}^{-1}\|+\|C_{s,t}\|\leq C
  \label{eq:stable-inverse-bounds}
\end{equation}
uniformly over $P\in\mathfrak C(\underline p,\overline p)$ and $t$. Moreover, there are $t_0\geq1$ and $c_Q>0$
such that
\begin{equation}
  Q_t\succeq c_Q(A_u^t)^\top A_u^t,
  \qquad t\geq t_0,
  \label{eq:Q-lower-bound}
\end{equation}
uniformly over $P\in\mathfrak C(\underline p,\overline p)$.
\end{lemma}

\begin{proof}
The push-through identity gives
\begin{equation}
  C_{s,t}
  =
  S_c^{1/2}
  \bigl(I+S_c^{1/2}O_{ss,t}S_c^{1/2}\bigr)^{-1}
  S_c^{1/2}.
  \label{eq:Cs-symmetric-form}
\end{equation}
Thus $C_{s,t}$ is symmetric positive semidefinite and
$\|C_{s,t}\|\leq\|S_c\|\leq \overline p$ since $S_c \preceq P_{ss} \preceq \overline p I$. Also,
\[
  \Lambda_{s,t}^{-1}
  =
  I-S_c^{1/2}
  \bigl(I+S_c^{1/2}O_{ss,t}S_c^{1/2}\bigr)^{-1}
  S_c^{1/2}O_{ss,t},
\]
so \eqref{eq:factor-uniform-bounds} and  \eqref{eq:Oss-bound}  give a uniform bound on $\Lambda_{s,t}^{-1}$.

We next prove \eqref{eq:Q-lower-bound}, by viewing $Q_t$ as a Schur complement.  For $\xi\in(0,1]$, define the regularized matrix 
$S_{c,\xi}=S_c+\xi I$ and set
\[
  C_{s,t}^{\xi}
  :=
  (S_{c,\xi}^{-1}+O_{ss,t})^{-1},
  \qquad
  Q_t^\xi
  :=
  O_{uu,t}-O_{us,t}C_{s,t}^{\xi}O_{su,t}.
\]
We first prove a statement for $Q_t^\xi$ and then recover the result for $Q_t$ by letting $\xi \downarrow 0$.
The Schur-complement variational formula gives, for $u\in\mathsf U$,
\begin{align}
  u^\top Q_t^\xi u
  =
  \inf_{w\in\mathsf S}
  \left\{
  w^\top S_{c,\xi}^{-1}w
  +
  \sum_{\ell=0}^{t-1}
  \left\|
  R^{-1/2}
  \left[
  H_uA_u^{\ell} u+H_sA_s^{\ell}(Fu+w)
  \right]
  \right\|^2
  \right\}.
  \label{eq:Q-variational}
\end{align}
Since $0\preceq S_c\preceq \overline p I$, one has
\begin{equation}
  S_{c,\xi}^{-1}\succeq (\overline p+1)^{-1}I.
  \label{eq:regularized-covariance-penalty}
\end{equation}

We will show that the terms with $A_u^\ell$ dominate for large enough $t$. Let $t\geq L$ and write $v=A_u^{t-L}u$. Restricting the observation sum in
\eqref{eq:Q-variational} to the last $L$ terms and using
$\|a+b\|^2\geq\frac12\|a\|^2-\|b\|^2$ gives
\begin{equation*}
  \sum_{\ell=0}^{L-1}
  \left\|
  R^{-1/2}
  \left[
  H_uA_u^{\ell}v
  +H_sA_s^{t-L+\ell}(Fu+w)
  \right]
  \right\|^2
  \geq
  \frac{c_o}{2}\|v\|^2
  -C\gamma_s^{2(t-L)}\|Fu+w\|^2.
\end{equation*}
By \eqref{eq:factor-uniform-bounds}, there is $C_F<\infty$ such that
\[
  \|Fu+w\|^2
  \leq 2C_F^2\|u\|^2+2\|w\|^2.
\]
Set $h=t-L$ and $\kappa=(\overline p+1)^{-1}$. The preceding estimates therefore give
\begin{align*}
  w^\top S_{c,\xi}^{-1}w
  +
  \sum_{\ell=0}^{t-1}  \left\|
  R^{-1/2}
  \left[
  H_uA_u^{\ell} u+H_sA_s^{\ell}(Fu+w)
  \right]
  \right\|^2 
  &\geq
  \frac{c_o}{2}\|A_u^hu\|^2
  +\bigl(\kappa-2C\gamma_s^{2h}\bigr)\|w\|^2 \\
  &\qquad
  -2CC_F^2\gamma_s^{2h}\|u\|^2.
\end{align*}
Since $u=A_u^{-h}A_u^hu$, Lemma~\ref{lem:power-bounds} implies
\[
  \|u\|^2
  \leq C_u^2\gamma_u^{2h}\|A_u^hu\|^2.
\]
Choose $h_0$ so large that
\[
  2C\gamma_s^{2h}\leq\frac\kappa2,
  \qquad
  2CC_F^2C_u^2(\gamma_s\gamma_u)^{2h}
  \leq\frac{c_o}{4},
  \qquad h\geq h_0.
\]
Such an $h_0$ exists because $\gamma_s,\gamma_u\in(0,1)$. Setting
$t_0=L+h_0$ yields, uniformly in $w$ and $\xi$,
\begin{equation}
  w^\top S_{c,\xi}^{-1}w
  +
  \sum_{\ell=0}^{t-1}  \left\|
  R^{-1/2}
  \left[
  H_uA_u^{\ell} u+H_sA_s^{\ell}(Fu+w)
  \right]
  \right\|^2 
  \geq
  \frac{c_o}{4}\|A_u^{t-L}u\|^2,
  \qquad t\geq t_0.
  \label{eq:variational-lower-final}
\end{equation}
Since
\[
  \|A_u^tu\|\leq\|A_u^{L}\|\,\|A_u^{t-L}u\|,
\]
\eqref{eq:variational-lower-final} implies
\[
  Q_t^\xi\succeq c_Q(A_u^t)^\top A_u^t
\]
with $c_Q>0$ independent of $\xi$, $P$, and $t\geq t_0$.
Finally,
\[
  C_{s,t}^{\xi}
  =
  (I+S_{c,\xi}O_{ss,t})^{-1}S_{c,\xi}
  \longrightarrow C_{s,t}
\]
as $\xi\downarrow0$. Since each $Q_t^\xi$ is the Schur complement
of the positive-definite stable block in the positive-semidefinite quadratic
form appearing in \eqref{eq:Q-variational}, one has $Q_t^\xi\succeq0$.
Hence $Q_t^\xi\to Q_t$ implies $Q_t\succeq0$, and the lower bound
passes to the limit.
\end{proof}

\subsubsection{Detailed closed-loop estimate}
\label{subsec:detailed-closed-loop-estimate}

We now prove the time-zero case of \eqref{eq:enkf-package-closed-loop} in Theorem~\ref{thm:enkf-riccati-package}, uniformly over
positive-semidefinite, possibly singular covariances in
$\mathfrak C(\underline p,\overline p)$.

\begin{proof}[Proof of bound \eqref{eq:enkf-package-closed-loop} for $s=0$]
Fix $P\in\mathfrak C(\underline p,\overline p)$. 
We first control the upper row of
$A_F^t\mathcal M_t^{-1}$.  This is where observability offsets the growth of
the unstable dynamics. 

The
powers of the transformed dynamics satisfy the exact identity
\begin{equation}
  A_F^t
  =L_F^{-1}A^tL_F
  =
  \begin{pmatrix}
    A_u^t&0\\
    A_s^tF-FA_u^t&A_s^t
  \end{pmatrix}.
  \label{eq:AF-power}
\end{equation}
Combining \eqref{eq:calM-inverse-blocks} and \eqref{eq:AF-power}, the top row of
$A_F^t\mathcal M_t^{-1}$ is
\[
  \begin{pmatrix}A_u^tX_{11,t}&A_u^tX_{12,t}\end{pmatrix}.
\] 
For $t\geq t_0$, the definition of $\Theta_t$ and  Lemma~\ref{lem:effective-unstable-information} give
\[
  \Theta_t\succeq Q_t\succeq c_Q(A_u^t)^\top A_u^t.
\]
Therefore
\begin{equation}
  \Theta_t^{-1}
  \preceq
  c_Q^{-1}A_u^{-t}(A_u^{-t})^\top.
  \label{eq:T-inverse-bound}
\end{equation}
Taking operator norms in \eqref{eq:T-inverse-bound} gives
\begin{equation}
  \|\Theta_t^{-1/2}\|
  \leq C\|A_u^{-t}\|
  \leq C\gamma_u^t.
  \label{eq:T-half-decay}
\end{equation}
Moreover, congruence of \eqref{eq:T-inverse-bound} by $A_u^t$ yields
\begin{equation}
  \|A_u^t\Theta_t^{-1/2}\|^2
  =
  \|A_u^t\Theta_t^{-1}(A_u^t)^\top\|
  \leq c_Q^{-1}.
  \label{eq:Au-T-half-bound}
\end{equation}
Using $X_{11,t}=\Theta_t^{-1}P_{uu}^{-1}$, we therefore obtain
\begin{align}
  \|A_u^tX_{11,t}\|
  &\leq
  \|A_u^t\Theta_t^{-1/2}\|
  \|\Theta_t^{-1/2}\|
  \|P_{uu}^{-1}\|
  \leq C\gamma_u^t.
  \label{eq:top-left-decay}
\end{align}

Similarly, congruence of \eqref{eq:T-inverse-bound} by $O_{us,t}$ together with the bound \eqref{eq:normalized-Ous-bound} give 
\begin{equation}
  \|\Theta_t^{-1/2}O_{us,t}\|^2=\|O_{su,t}\Theta_t^{-1}O_{us,t}\|\leq c_Q^{-1}\|(A_u^{-t})^\top O_{us,t}\|^2\leq C\gamma^{2t}.
  \label{eq:T-half-Ous-bound}
\end{equation}
Equations \eqref{eq:X12}, \eqref{eq:Au-T-half-bound},
\eqref{eq:T-half-Ous-bound}, and \eqref{eq:stable-inverse-bounds} imply
\begin{align}
  \|A_u^tX_{12,t}\|
  &\leq
  \|A_u^t\Theta_t^{-1/2}\|
  \|\Theta_t^{-1/2}O_{us,t}\|
  \|\Lambda_{s,t}^{-1}\|
  \leq C\gamma^t.
  \label{eq:top-right-decay}
\end{align}

It remains to control the bottom row of $A_F^t \mathcal M_t^{-1}$.  It can be written as
\begin{align}
  &(A_s^tF-FA_u^t)
  \begin{pmatrix}X_{11,t}&X_{12,t}\end{pmatrix}
  +A_s^t
  \begin{pmatrix}X_{21,t}&X_{22,t}\end{pmatrix}
  \\
  &\quad=
  A_s^t
  \begin{pmatrix}
    FX_{11,t}+X_{21,t}&FX_{12,t}+X_{22,t}
  \end{pmatrix}
  -F
  \begin{pmatrix}
    A_u^tX_{11,t}&A_u^tX_{12,t}
  \end{pmatrix}.
\label{eq:bottom-row-A_FM_tinv}
\end{align}
The unstable components appearing there
are already covered by the preceding estimates, while the remaining terms are
multiplied by the stable dynamics.  We first record uniform bounds on the
lower blocks of $\mathcal M_t^{-1}$. Recall from \eqref{eq:X21} in Lemma~\ref{lem:inverse-block-M_t} that $X_{21,t} = - C_{s,t} O_{su,t} \Theta_t^{-1} P_{uu}^{-1}$. From
\eqref{eq:T-half-decay} and \eqref{eq:T-half-Ous-bound},
\begin{equation}
  \|\Theta_t^{-1}O_{us,t}\|
  \leq
  \|\Theta_t^{-1/2}\|
  \|\Theta_t^{-1/2}O_{us,t}\|
  \leq C(\gamma_u\gamma)^t.
  \label{eq:T-inverse-Ous}
\end{equation}
Since $\Theta_t$ and $P_{uu}$ are symmetric,
\[
  \|O_{su,t}\Theta_t^{-1}P_{uu}^{-1}\|
  =
  \|P_{uu}^{-1}\Theta_t^{-1}O_{us,t}\| \leq \|P_{uu}^{-1}\|\|\Theta_t^{-1}O_{us,t}\|,
\]
which is uniformly bounded, by \eqref{eq:T-inverse-Ous} together with $P \in \mathfrak{C}(\underline p, \overline p)$. Combining this estimate with the bound on $C_{s,t}$ in \eqref{eq:stable-inverse-bounds} shows that $X_{21,t}$
is uniformly bounded.  Equation \eqref{eq:T-half-Ous-bound} also shows that
$O_{su,t}\Theta_t^{-1}O_{us,t}$ is uniformly bounded. Hence
\eqref{eq:X22}, together with \eqref{eq:stable-inverse-bounds}, shows that
$X_{22,t}$ is uniformly bounded.
Equations~\eqref{eq:X11}--\eqref{eq:X12},
\eqref{eq:T-half-decay}, and
\eqref{eq:T-half-Ous-bound} also show that \(X_{11,t}\) and
\(X_{12,t}\) are uniformly bounded.

In \eqref{eq:bottom-row-A_FM_tinv}, the first term decays exponentially because $A_s^t$ does and all matrices in
brackets are uniformly bounded. The second decays by the top-row estimates
and the uniform bound on $F$ in \eqref{eq:factor-uniform-bounds}. Hence, combining the top and bottom row estimates, there are $C<\infty$ and
$\rho\in(0,1)$ such that
\begin{equation}
  \|A_F^t(I+\mathsf D_P\mathcal O_{F,t})^{-1}\|
  \leq C\rho^t,
  \qquad t\geq t_0.
  \label{eq:transformed-decay}
\end{equation}

For the finitely many times $0\leq t<t_0$, the same estimate follows after
increasing $C$. To justify uniformity over $\mathfrak{C}(\underline p, \overline p)$, note that the parameter set determined
by
\[
  \underline p I\preceq P_{uu}\preceq \overline p I,
  \qquad
  0\preceq S_c\preceq \overline p I,
  \qquad
  \|F\|\leq \overline p/\underline p
\]
is compact. For fixed $t$, the maps from $(P_{uu},S_c,F)$ to $A_F^t$ and
$\mathcal O_{F,t}$ are continuous, and the preceding determinant argument
shows that $I+\mathsf D_P\mathcal O_{F,t}$ is invertible throughout this compact set.
The inverse therefore depends continuously on the parameters and has a
uniformly bounded norm. Finally, \eqref{eq:factor-uniform-bounds} and
\eqref{eq:transformed-closed-loop} transfer
\eqref{eq:transformed-decay} to the original coordinates and prove
the claimed time-zero estimate.
\end{proof}

\begin{remark}[Purely unstable dynamics]
If $\mathsf S=\{0\}$, the proof above applies with all stable-block terms
omitted.
\end{remark}

\subsubsection{Uniform bounds along Riccati trajectories}

The closed-loop estimate proved above starts at time zero. For
the later mean analysis, we need the same estimate for products starting at
arbitrary times. This requires showing that every Riccati iterate remains in
a common class of the form \eqref{eq:covariance-class}.

\begin{lemma}[Uniform upper and lower bounds on the unstable covariance]
\label{lem:unstable-covariance-bounds}
For every \(\underline p,\overline p>0\), there exist
\(0<\underline p_*\leq \overline p_{u,*}<\infty\) such that
\begin{equation}
  \underline p_*I_{\mathsf U}
  \preceq[\Phi^t(P)]_{uu}
  \preceq \overline p_{u,*}I_{\mathsf U},
  \qquad
  P\in\mathfrak C(\underline p,\overline p),\quad t\geq0.
  \label{eq:unstable-covariance-bounds}
\end{equation}
\end{lemma}

\begin{proof}  
Fix $P\in\mathfrak C(\underline p,\overline p)$ and use the factorization,
transformed system, and information blocks in
\eqref{eq:F-Schur}--\eqref{eq:transformed-Gramian-blocks}.  
Equation \eqref{eq:Ct-equiv-expression} gives
\[
  \mathcal C_t(P)
  =
  L_F\mathcal M_t^{-1}\mathsf D_PL_F^\top.
\]
Since the upper-left block of
\(\mathcal M_t^{-1}\mathsf D_P\) is
\(X_{11,t}P_{uu}=\Theta_t^{-1}\), it follows from \eqref{eq:batch-covariance-formula} that
\begin{equation}
  \bigl[\Phi^t(P)\bigr]_{uu}
  =
  A_u^t\Theta_t^{-1}(A_u^t)^\top.
  \label{eq:unstable-covariance-representation}
\end{equation}
Recall from \eqref{eq:effective-Q} that $\Theta_t = P_{uu}^{-1} + Q_t$. For $t\geq t_0$, Lemma~\ref{lem:effective-unstable-information} gives the
required lower bound on $\Theta_t$.  For $t<t_0$, 
$\Theta_t\succeq P_{uu}^{-1}\succeq\overline p^{-1}I$, and since $A_u^t (A_u^t)^\top \preceq \|A_u^t\|^2 I$ with only finitely many times to consider, it follows that, uniformly in $P$ and $t$,
\begin{equation}
  \Theta_t\succeq c_1(A_u^t)^\top A_u^t,
  \qquad t\geq0.
  \label{eq:T-lower}
\end{equation}

For the reverse bound, \(0\preceq Q_t\preceq O_{uu,t}\) because
\(C_{s,t}\succeq0\).\nc The explicit form of the
unstable information block and the bounds
\eqref{eq:power-bounds}--\eqref{eq:factor-uniform-bounds} give, for
$v\in\mathsf U$,
\begin{equation}
  v^\top O_{uu,t}v
  \leq
  C\left[
  \sum_{\ell=1}^{t}\gamma_u^{2\ell}
  +\gamma_u^{2t}\sum_{\ell=0}^{t-1}\gamma_s^{2\ell}
  \right]
  \|A_u^tv\|^2
  \leq C\|A_u^tv\|^2.
\end{equation}
Here we used
$\|A_u^\ell v\|\leq C\gamma_u^{t-\ell}\|A_u^tv\|$ and
$\|v\|\leq C\gamma_u^t\|A_u^tv\|$, both direct consequences of
\eqref{eq:power-bounds}. The latter inequality also gives
 $P_{uu}^{-1} \preceq \underline p^{-1} I \nc \preceq C(A_u^t)^\top A_u^t$, and hence
\begin{equation}
  \Theta_t\preceq C_1(A_u^t)^\top A_u^t,
  \qquad t\geq0.
  \label{eq:T-upper}
\end{equation}

Since $A_u$ is invertible, inversion of
\eqref{eq:T-lower}--\eqref{eq:T-upper} yields
\begin{equation}
  C_1^{-1}A_u^{-t}(A_u^{-t})^\top
  \preceq \Theta_t^{-1}
  \preceq
  c_1^{-1}A_u^{-t}(A_u^{-t})^\top.
\end{equation}
Congruence by $A_u^t$ and \eqref{eq:unstable-covariance-representation} give
\begin{equation}
  C_1^{-1}I
  \preceq
  \bigl[\Phi^t(P)\bigr]_{uu}
  \preceq
  c_1^{-1}I.
\end{equation}
This proves the claim with $\underline p_*=C_1^{-1}$ and $\overline p_{u,*}=c_1^{-1}$.
\end{proof}

\begin{lemma}[Uniform boundedness of the full covariance trajectory]
\label{lem:full-covariance-bound}
For every \(\underline p,\overline p>0\), there exists
\(\overline p_*<\infty\) such that
\begin{equation}
  \|\Phi^t(P)\|\leq\overline p_*,
  \qquad
  P\in\mathfrak C(\underline p,\overline p),\quad t\geq0.
  \label{eq:full-covariance-bound}
\end{equation}
Consequently,
\begin{equation}
  \Phi^t(P)\in\mathfrak C(\underline p_*,\overline p_*),
  \qquad
  P\in\mathfrak C(\underline p,\overline p),\quad t\geq0,
  \label{eq:trajectory-in-common-class}
\end{equation}
where \(\underline p_*\) is given by
Lemma~\ref{lem:unstable-covariance-bounds}.
\end{lemma}

\begin{proof} 
By \eqref{eq:batch-covariance-formula},
$\Phi^t(P)=A^t\mathcal C_t(P)(A^t)^\top$ with
$0\preceq\mathcal C_t(P)\preceq P$. Hence, 
\begin{equation}
  \bigl\|[\Phi^t(P)]_{ss}\bigr\|=\left\|A_s^t[\mathcal C_t(P)]_{ss}(A_s^t)^\top\right\|\leq\|A_s^t\|^2\|P\|\leq C_s^2\overline p.
  \label{eq:stable-covariance-bound}
\end{equation}
The unstable block is bounded by
Lemma~\ref{lem:unstable-covariance-bounds}. Positive semidefiniteness bounds
the cross block by the geometric mean of the two diagonal-block norms.
Consequently,
\[
  \|\Phi^t(P)\|
  \leq\left(\overline p_{u,*}^{1/2}+C_s\overline p^{1/2}\right)^2,
\]
which proves \eqref{eq:full-covariance-bound} and
\eqref{eq:trajectory-in-common-class}.
\end{proof}

\begin{proof}[Proof of Theorem~\ref{thm:enkf-riccati-package}]
Lemmas~\ref{lem:unstable-covariance-bounds} and
\ref{lem:full-covariance-bound} give \eqref{eq:enkf-package-common-class}.  The time-zero estimate proved
in Section~\ref{subsec:detailed-closed-loop-estimate}, applied to the common
class in \eqref{eq:trajectory-in-common-class}, can be restarted at
any time because the Riccati recursion is autonomous.  This gives \eqref{eq:enkf-package-closed-loop}.

Applying the exact difference identity in
Proposition~\ref{prop:exact-riccati-identities} to the two trajectories at time
\(s\) gives
\[
  \Phi^t(P)-\Phi^t(Q)
  =\mathcal B_{t-s}(\Phi^s(P))
   [\Phi^s(P)-\Phi^s(Q)]
   \mathcal B_{t-s}(\Phi^s(Q))^\top.
\]
The shifted closed-loop estimate proves \eqref{eq:enkf-package-forgetting}.  Finally, since $R \succ 0$,
the matrices \(HPH^\top+R\) have uniformly bounded inverses. On $\|P\| \leq M$, the remaining $P$-factor in $\mathsf K(P)$ is also uniformly bounded. The resolvent
identity therefore shows that \(\mathsf K\) is Lipschitz on this set; since
\(\Psi(P)=P-\mathsf K(P)HP\), the same is true of \(\Psi\).  This proves \eqref{eq:enkf-package-regularity}.
\end{proof}

\subsection{Gaussian initialization}
\label{app:gaussian-initialization}

This section proves the Gaussian initialization estimates used for the EnKF.
The Gaussian reduction below is also used for the localized EnKF. Recall that
$r_u=\dim(\mathsf U)$.

\subsubsection{Gaussian reduction}

\begin{lemma}[Mean--covariance independence and Wishart reduction]
\label{lem:gaussian-wishart-reduction}
The random variables $\widehat m_0^N$ and $\widehat P_0^N$ are independent,
\begin{equation}
  \widehat m_0^N-\widehat m_0
  \sim\mathcal N\left(0,\frac{\widehat P_0}{N}\right),
  \label{eq:initial-mean-distribution}
\end{equation}
and
\begin{equation}
  \widehat P_0^N
  \stackrel{\mathrm d}{=}
  \widehat P_0^{1/2}
  \left(\frac1{N-1}GG^\top\right)
  \widehat P_0^{1/2},
  \label{eq:sample-covariance-wishart}
\end{equation}
where $G\in\R^{d_u\times (N-1)}$ has i.i.d.\ standard Gaussian entries.
In particular,
\begin{equation}
  (\widehat P_0^N)_{uu}
  \stackrel{\mathrm d}{=}
  (\widehat P_0)_{uu}^{1/2}
  \left(\frac1{N-1}G_uG_u^\top\right)
  (\widehat P_0)_{uu}^{1/2},
  \label{eq:unstable-wishart-reduction}
\end{equation}
where $G_u\in\R^{r_u\times (N-1)}$ has i.i.d.\ standard Gaussian entries.
\end{lemma}

\begin{proof}
These are the standard normal-sampling identities; see
\cite[Chapter~3]{Muirhead1982}. For completeness, they remain valid when
$\widehat P_0$ is singular: write the centered sample matrix as
$Z=\widehat P_0^{1/2}G_N$ where $G_N=[g_1, \dots, g_N]\in \R^{d_u \times N}$ contains i.i.d.\ standard Gaussian vectors. Rotate  its columns by an orthogonal matrix whose
first column is $N^{-1/2}\one$. Rotational invariance makes the first rotated
column, $\sqrt N(\widehat m_0^N-\widehat m_0)$, independent of the remaining
$N-1$ columns. These are $\widehat P_0^{1/2}$ times standard Gaussian columns, whose sum of outer products is
$\widehat X_0\widehat X_0^\top$. This proves
\eqref{eq:initial-mean-distribution}--\eqref{eq:sample-covariance-wishart};
projection onto $\mathsf U$ gives \eqref{eq:unstable-wishart-reduction}.
\end{proof}

\subsubsection{Concentration estimates}
The next result establishes the bounds \eqref{eq:enkf-package-initialization} in Proposition~\ref{prop:enkf-probabilistic-package}.
\begin{proposition}[Gaussian initialization event]
\label{prop:gaussian-initialization}
For every \(\delta\in(0,1)\), with probability at least \(1-3\delta/4\),
\begin{align}
  \|\widehat m_0^N-\widehat m_0\|
  &\leq\mu_N(\delta),
  \label{eq:initial-mean-concentration}\\
  \|\widehat P_0^N-\widehat P_0\|
  &\leq\Delta_N(\delta),
  \label{eq:initial-covariance-concentration}\\
  (\widehat P_0^N)_{uu}
  &\succeq
  \left(
  1-\sqrt{\frac{r_u}{N-1}}-\sqrt{\frac{2\log(4/\delta)}{N-1}}
  \right)_+^2
  (\widehat P_0)_{uu}.
  \label{eq:initial-unstable-lower-bound}
\end{align}
\end{proposition}
\begin{proof}
Recall that \eqref{eq:initial-mean-distribution} and Gaussian concentration for
the Euclidean norm give, with probability at least $1-\delta/4$,
\[
  \|\widehat m_0^N-\widehat m_0\|
  \leq
  \mathbb E\|\widehat m_0^N-\widehat m_0\|
  +\sqrt{\frac{2\|\widehat P_0\|\log(4/\delta)}{N}}.
\]
The expectation is at most
$\sqrt{\operatorname{Tr}(\widehat P_0)/N}$, proving
\eqref{eq:initial-mean-concentration}.

By Lemma~\ref{lem:gaussian-wishart-reduction}, the centered empirical
covariance has the same distribution as a sample covariance based on $N-1$
i.i.d. centered Gaussian vectors with covariance $\widehat P_0$.
The effective-rank covariance concentration inequality of
\cite{KoltchinskiiLounici2017} therefore gives
\eqref{eq:initial-covariance-concentration} with probability at
least $1-\delta/4$.

By \eqref{eq:unstable-wishart-reduction}, for every $x\in\mathsf U$,  with the smallest singular value of a matrix  denoted by $s_{\min}(\cdot)$,
\begin{equation*}
  x^\top(\widehat P_0^N)_{uu}x=\frac1{N-1}\left\|G_u^\top(\widehat P_0)_{uu}^{1/2}x\right\|^2\geq\frac{s_{\min}(G_u)^2}{N-1}x^\top(\widehat P_0)_{uu}x.
\end{equation*}
The Gaussian smallest-singular-value inequality gives, with probability at
least $1-\delta/4$,
\[
  s_{\min}(G_u)
  \geq
  \left(\sqrt{N-1}-\sqrt{r_u}-\sqrt{2\log(4/\delta)}\right)_+;
\]
see, for example, \cite[Theorem~II.13]{DavidsonSzarek2001}. This proves
\eqref{eq:initial-unstable-lower-bound}. A union bound completes
the proof.
\end{proof}

The next result establishes the claim in Proposition~\ref{prop:enkf-probabilistic-package} that \(\widehat P_0^N\) and \(\widehat P_0\) belong to a common deterministic
covariance class. 

\begin{corollary}[A fixed Riccati-stability class]
\label{cor:gaussian-fixed-stability-class}
Set
\begin{equation}
  \lambda_u:=\lambda_{\min}((\widehat P_0)_{uu})>0.
  \label{eq:lambda-u-definition}
\end{equation}
If \eqref{eq:effective-rank-ensemble-size} holds with
$C_{\mathrm{init}}$ sufficiently large, then
\begin{equation}
  \left(
  1-\sqrt{\frac{r_u}{N-1}}-\sqrt{\frac{2\log(4/\delta)}{N-1}}
  \right)_+^2\geq\frac14,
  \qquad
  \Delta_N(\delta)\leq\|\widehat P_0\|,
  \label{eq:fixed-class-sample-conditions}
\end{equation}
and, with probability at least $1-3\delta/4$,
\begin{equation}
  \widehat P_0,\widehat P_0^N
  \in
  \mathfrak C\left(\frac{\lambda_u}{4},
  2\|\widehat P_0\|\right).
  \label{eq:common-random-stability-class}
\end{equation}
In particular, $\widehat P_0^N$ need not be positive definite on the stable
subspace: operator-norm accuracy together with positive variance on every
unstable direction is sufficient.
\end{corollary}
\begin{proof}
The ensemble-size condition and the inequality
\begin{equation}
  \sqrt{r_u}+\sqrt{2\log(4/\delta)}
  \leq\sqrt{2\bigl(r_u+2\log(4/\delta)\bigr)}
\end{equation}
show that the first inequality in
\eqref{eq:fixed-class-sample-conditions} holds after increasing
$C_{\mathrm{init}}$ if necessary. The definition of $\Delta_N(\delta)$
similarly gives $\Delta_N(\delta)\leq\|\widehat P_0\|$, proving
\eqref{eq:fixed-class-sample-conditions}. On the event of
Proposition~\ref{prop:gaussian-initialization}, it follows that
\begin{equation}
  (\widehat P_0^N)_{uu}
  \succeq\frac14(\widehat P_0)_{uu}
  \succeq\frac{\lambda_u}{4}I_{\mathsf U}.
\end{equation}
Moreover,
\begin{equation}
  \|\widehat P_0^N\|
  \leq
  \|\widehat P_0\|+
  \|\widehat P_0^N-\widehat P_0\|
  \leq2\|\widehat P_0\|.
\end{equation}
The same bounds, with room to spare, hold for $\widehat P_0$ itself.
\end{proof}

\subsection{Covariance, gain, and innovation estimates}
\label{app:auxiliary-estimates}

\subsubsection{Analysis and gain regularity}
The next result refines the bound \eqref{eq:enkf-package-regularity} in Theorem~\ref{thm:enkf-riccati-package}, making the dependence on $\overline p$ explicit. 

\begin{lemma}[Analysis and gain regularity]
\label{lem:analysis-covariance-difference}
\label{lem:gain-lipschitz}
For every \(P,Q\in\Splus^{d_u}\),
\begin{equation}
  \Psi(P)-\Psi(Q)
  =(I-\mathsf K(P)H)(P-Q)(I-\mathsf K(Q)H)^\top.
  \label{eq:analysis-covariance-difference}
\end{equation}
If \(\|P\|,\|Q\|\le\overline p\), then
\begin{align}
  \|\Psi(P)-\Psi(Q)\|
  &\le C_\Psi(\overline p)\|P-Q\|,
  \label{eq:analysis-map-lipschitz-bounded-set}\\
  \|\mathsf K(P)-\mathsf K(Q)\|
  &\le C_K(\overline p)\|P-Q\|,
  \label{eq:gain-lipschitz}
\end{align}
where one may take
\begin{equation}
  C_\Psi(\overline p)
  =\left(1+\frac{\overline p\|H\|^2}{\lambda_{\min}(R)}\right)^2,
  \qquad
  C_K(\overline p)
  =\frac{\|H\|}{\lambda_{\min}(R)}
  +\frac{\overline p\|H\|^3}{\lambda_{\min}(R)^2}.
  \label{eq:gain-lipschitz-constant}
\end{equation}
\end{lemma}

\begin{proof}
Set \(G:=H^\top R^{-1}H\).  The push-through and Woodbury identities give
\[ I - \mathsf K(P)H = (I + PG)^{-1},  \quad 
  \Psi(P)=(I+PG)^{-1}P,
  \quad
  \Psi(Q)=Q(I+GQ)^{-1}.
\]
Using
\(P-Q=P(I+GQ)-(I+PG)Q\) yields 
\eqref{eq:analysis-covariance-difference}.  Moreover,
\[
  \|I-\mathsf K(P)H\|,
  \|I-\mathsf K(Q)H\|
  \le 1+\frac{\overline p\|H\|^2}{\lambda_{\min}(R)},
\]
which proves \eqref{eq:analysis-map-lipschitz-bounded-set}.

For the gain estimate, set
\(S_P:=HPH^\top+R\) and \(S_Q:=HQH^\top+R\).  Then
\[
  \|S_P^{-1}\|,\|S_Q^{-1}\|\le \lambda_{\min}(R)^{-1},
  \qquad
  S_P^{-1}-S_Q^{-1}=S_P^{-1}H(Q-P)H^\top S_Q^{-1}.
\]
Substitution into
\[
  \mathsf K(P)-\mathsf K(Q)
  =(P-Q)H^\top S_P^{-1}
  +QH^\top(S_P^{-1}-S_Q^{-1})
\]
proves \eqref{eq:gain-lipschitz}--\eqref{eq:gain-lipschitz-constant}.
\end{proof}

\subsubsection{Innovation estimates}
 Recall from Section~\ref{subsec:enkf-proofs} that $\iota_t:=y_t-H\widehat m_t$, 
 and from Section~\ref{subsec:localized-proofs} that  $S_t:=H\widehat P_tH^\top+R$.
Set 
\begin{equation}
  M_\iota:=\sup_{t\ge0}\|S_t\|,
  \qquad
  \tau_\iota:=\sup_{t\ge0}\operatorname{Tr}(S_t).
  \label{eq:innovation-covariance-bounds}
\end{equation}

Recall that \(\mathcal Y_{t-1}:=\sigma(y_0,\ldots,y_{t-1})\), with
\(\mathcal Y_{-1}\) the trivial sigma-algebra.
Under the linear--Gaussian model
\eqref{eq:initial-state-law}--\eqref{eq:observation-model}, the exact Kalman
forecast mean and covariance satisfy
\begin{equation}
  \widehat m_t=\mathbb E[u_t\mid\mathcal Y_{t-1}],
  \qquad
  \widehat P_t=\operatorname{Cov}(u_t\mid\mathcal Y_{t-1}).
  \label{eq:kalman-conditional-moments}
\end{equation}

The next result establishes the statements about the innovations in Proposition~\ref{prop:enkf-probabilistic-package}.
\begin{lemma}[Gaussian innovations and a weighted all-time bound]
\label{lem:weighted-innovation-bound}
Under the linear--Gaussian model
\eqref{eq:initial-state-law}--\eqref{eq:observation-model}, the innovations
\((\iota_t)_{t\geq0}\) are independent and
\begin{equation}
  \iota_t\sim\mathcal N(0,S_t).
  \label{eq:innovation-distribution}
\end{equation}
Under Assumption~\ref{ass:dynamics-initialization}, the constants in \eqref{eq:innovation-covariance-bounds} are finite.
If an initial ensemble is independent of the initial state $u_0$ and observation noises $(\eta_t)_{t \geq 0}$,
then the initial ensemble \nc is also independent of the innovation sequence.
Whenever the constants in \eqref{eq:innovation-covariance-bounds} are finite, for every
\(\beta\in(0,1)\) and \(\delta\in(0,1)\), with probability at least \(1-\delta\),
\begin{equation}
  \sum_{t=0}^\infty\beta^t\|\iota_t\|
  \leq
  \frac{\sqrt{\tau_\iota}}{1-\beta}
  +\sqrt{\frac{2M_\iota\log(1/\delta)}{1-\beta^2}}.
  \label{eq:weighted-innovation-bound}
\end{equation}
\end{lemma}

\begin{proof}
The vector $(u_0,y_0,y_1,\ldots)$ is jointly Gaussian, and
\eqref{eq:kalman-conditional-moments} gives
\[
  \mathbb E[\iota_t\mid\mathcal Y_{t-1}]=0,
  \qquad
  \operatorname{Cov}(\iota_t\mid\mathcal Y_{t-1})=S_t.
\]
Here $S_t$ is deterministic. If $s<t$, conditioning gives
$\mathbb E[\iota_t\iota_s^\top]=0$; joint Gaussianity therefore gives
independence and \eqref{eq:innovation-distribution}. Independence from an
initial ensemble follows from the ensemble's assumed independence from $u_0$ and the
observation noises.

Under Assumption~\ref{ass:dynamics-initialization},
Lemma~\ref{lem:full-covariance-bound}, applied to $\widehat P_0$, gives
$\sup_t\|\widehat P_t\|<\infty$. The asserted finiteness of $M_\iota$ and
$\tau_\iota$ follows immediately from 
$S_t:=H\widehat P_tH^\top+R$.

Write $\iota_t=S_t^{1/2}z_t$, where $(z_t)_{t\geq0}$ are i.i.d.
standard Gaussian vectors in $\R^{d_y}$. For $T\geq0$, define
\[
  F_T(z_0,\ldots,z_T)
  :=\sum_{t=0}^T\beta^t\|S_t^{1/2}z_t\|.
\]
By Cauchy--Schwarz, $F_T$ is Lipschitz on the product Euclidean space with
constant at most
\[
  \left(\sum_{t=0}^T\beta^{2t}\|S_t\|\right)^{1/2}
  \leq\sqrt{\frac{M_\iota}{1-\beta^2}}.
\]
Moreover,
\[
  \mathbb E F_T\leq\sum_{t=0}^T\beta^t\bigl(\mathbb E\|\iota_t\|^2\bigr)^{1/2}\leq\frac{\sqrt{\tau_\iota}}{1-\beta}.
\]
Gaussian concentration therefore gives, with probability at least $1-\delta$,
\[
  F_T
  \leq
  \frac{\sqrt{\tau_\iota}}{1-\beta}
  +\sqrt{\frac{2M_\iota\log(1/\delta)}{1-\beta^2}}.
\]
The bound is uniform in $T$. Letting $T\to\infty$ and using monotone
convergence proves \eqref{eq:weighted-innovation-bound}.
\end{proof}

    \section{Localized ensemble Kalman filter}
\label{app:localized-enkf}

This section contains the localization-specific results used in
Section~\ref{subsec:localized-proofs}.
Sections~\ref{app:localized-identities}--
\ref{sec:localized-deterministic-bias} develop the deterministic block-local
theory, Sections~\ref{sec:localized-empirical-covariance}
and~\ref{app:localized-mean-auxiliary} control local sampling errors and the
mean, 
and Section~\ref{sec:localized-steady-state-decay} treats steady-state
covariance and asymptotic spatial decay.

\subsection{Block-decay algebra and exact localized covariance identities}
\label{app:localized-identities}
Throughout Section \ref{subsec:localized-proofs}, we use the fact that the block-decay norm defines a dimension-independent matrix algebra. While this fact is classical \cite{Jaffard1990,benzi2017localization}, we include a proof for completeness.
\begin{lemma}[Block-decay algebra]
\label{lem:block-decay-algebra}
There is a constant $C_\alpha<\infty$, depending only on $\alpha$, such that
for all compatible block matrices,
\begin{equation}
  \|M_1M_2\|_{\mathcal J_\alpha}
  \le C_\alpha\|M_1\|_{\mathcal J_\alpha}
  \|M_2\|_{\mathcal J_\alpha}.
  \label{eq:block-decay-product}
\end{equation}
Moreover,
\begin{equation}
  \|M\|_{\mathrm{loc}}
  \le Z_\alpha\|M\|_{\mathcal J_\alpha},
  \qquad
  \|M^\top\|_{\mathcal J_\alpha}=\|M\|_{\mathcal J_\alpha}.
  \label{eq:block-decay-controls-local}
\end{equation}
All constants are independent of $J$ and of the local block dimensions.
\begin{proof}
The transpose identity is immediate, and
\[
  \sum_{k=1}^J\|M_{jk}\|
  \le \|M\|_{\mathcal J_\alpha}
  \sum_{k=1}^J(1+|j-k|)^{-\alpha}
  \le Z_\alpha\|M\|_{\mathcal J_\alpha},
\]
(and its column analog) proves the local-norm estimate. For the product,
the standard polynomial-weight convolution bound
\[
  \sup_{j,k}\sum_{i=1}^J
  \frac{(1+|j-k|)^\alpha}
  {(1+|j-i|)^\alpha(1+|i-k|)^\alpha}
  \le C_\alpha
\]
follows by splitting according to which of $|j-i|$ and $|i-k|$ is at least
$|j-k|/2$ and using $\alpha>1$; see also \cite{Jaffard1990}. Therefore
\[
  (1+|j-k|)^\alpha\|(M_1M_2)_{jk}\|
  \le C_\alpha\|M_1\|_{\mathcal J_\alpha}
  \|M_2\|_{\mathcal J_\alpha},
\]
which proves \eqref{eq:block-decay-product}.
\end{proof}

\end{lemma}

Define the full empirical forecast covariance (which is not formed by the localized EnKF algorithm)  by
\[
  \widehat P_{t+1}^{N,\mathrm{loc}}
  :=\frac1{N-1}\widehat X_{t+1}^{\mathrm{loc}}
  (\widehat X_{t+1}^{\mathrm{loc}})^\top.
\] 
Recall that $\mathcal F_0^{X,N}:=\sigma(\widehat X_0^{\mathrm{loc}})$.  
In the next result, 
we characterize the exact covariance recursion specified by the localized square-root EnKF. We show that the anomaly matrices, empirical covariances, and local gains are measurable with respect to $\mathcal F_0^{X,N}$, and that replacing $A=D+E$ by $D$ in the Riccati map incurs an error of at most order $\varepsilon$ over a single step.
\begin{proposition}[Exact local covariance recursion]
\label{prop:localized-exact-covariance-identities}
For every \(t\ge0\), the localized square-root EnKF satisfies:
\begin{enumerate}
\item The global analyzed anomaly matrix is centered,
\[
  X_t^{\mathrm{loc}}\one=0.
\]
\item Its diagonal covariance blocks are the exact local Kalman analyses,
\begin{equation}
  P_{t,j}^{N,\mathrm{loc}}
  =\Psi_j(\widehat P_{t,j}^{N,\mathrm{loc}}),
  \qquad j=1,\ldots,J.
  \label{eq:localized-exact-analysis-block}
\end{equation}
\item The full and local forecast covariances satisfy
\begin{equation}
  \widehat P_{t+1}^{N,\mathrm{loc}}
  =A P_t^{N,\mathrm{loc}}A^\top,
  \qquad
  \widehat P_{t+1,j}^{N,\mathrm{loc}}
  =\sum_{i,k=1}^J A_{ji}P_{t,ik}^{N,\mathrm{loc}}A_{jk}^\top.
  \label{eq:localized-exact-forecast-block}
\end{equation}
\item The empirical cross-covariance blocks obey
\begin{equation}
  \|P_{t,jk}^{N,\mathrm{loc}}\|
  \le
  \|P_{t,j}^{N,\mathrm{loc}}\|^{1/2}
  \|P_{t,k}^{N,\mathrm{loc}}\|^{1/2}.
  \label{eq:localized-cross-covariance-bound}
\end{equation}
\item For all 
$t\ge0$,  the anomaly matrices, empirical covariances, local
gains, and matrices 
$A(I-K_t^{N,\mathrm{loc}}H)$  are
$\mathcal F_0^{X,N}$-measurable.
\end{enumerate}

If, in addition, for some $\overline p<\infty$,
\begin{equation}
  \max_{1\le j\le J}\|P_{t,j}^{N,\mathrm{loc}}\|\le \overline p,
  \label{eq:localized-diagonal-covariance-upper-bound}
\end{equation}
then with the residual denoted by $\zeta_{t,j}^N$, 
\begin{equation}
  \widehat P_{t+1,j}^{N,\mathrm{loc}}
  =\Phi_j(\widehat P_{t,j}^{N,\mathrm{loc}})
  +\zeta_{t,j}^N,
  \qquad
  \max_j\|\zeta_{t,j}^N\|
  \le \overline p(2a\varepsilon+\varepsilon^2),
  \label{eq:localized-perturbed-local-riccati}
\end{equation}
where \(a:=\max_j\|A_j\|\).
\end{proposition}

\begin{proof}
Because \(\widehat X_{t,j}^{\mathrm{loc}}\one=0\), one has
\(\Omega_{t,j}^{\mathrm{loc}}\one=\one\), and hence
\(T_{t,j}^{\mathrm{loc}}\one=\one\). Therefore every local analyzed anomaly
matrix is centered, and so is their block-row stacking.
Since every anomaly update is a deterministic function of the preceding
anomalies and the fixed model matrices, the measurability assertion follows
by induction.

Applying the calculation in Proposition~\ref{prop:exact-square-root-identities} to
each block gives
\[
  \frac1{N-1}X_{t,j}^{\mathrm{loc}}
  (X_{t,j}^{\mathrm{loc}})^\top
  =\Psi_j(\widehat P_{t,j}^{N,\mathrm{loc}}),
\]
which proves \eqref{eq:localized-exact-analysis-block}. Since the
forecast~\eqref{eq:localized-full-forecast} is global and deterministic,
\(\widehat X_{t+1}^{\mathrm{loc}}=A X_t^{\mathrm{loc}}\), and therefore \eqref{eq:localized-exact-forecast-block} follows by taking diagonal
blocks.

For the cross-covariance estimate, write
\[
  P_{t,jk}^{N,\mathrm{loc}}
  =\frac1{N-1}X_{t,j}^{\mathrm{loc}}
  (X_{t,k}^{\mathrm{loc}})^\top.
\]
Submultiplicativity of the operator norm yields
\[
  \|P_{t,jk}^{N,\mathrm{loc}}\|
  \le
  \left\|\frac1{\sqrt{N-1}}X_{t,j}^{\mathrm{loc}}\right\|
  \left\|\frac1{\sqrt{N-1}}X_{t,k}^{\mathrm{loc}}\right\|,
\]
and the two factors are respectively
\(\|P_{t,j}^{N,\mathrm{loc}}\|^{1/2}\) and
\(\|P_{t,k}^{N,\mathrm{loc}}\|^{1/2}\). This proves
\eqref{eq:localized-cross-covariance-bound}.

Finally, write \(A=D+E\). The \((j,j)\)-block of the forecast covariance in \eqref{eq:localized-exact-forecast-block} is \nc 
\begin{align*}
  \widehat P_{t+1,j}^{N,\mathrm{loc}}
  &=A_j P_{t,j}^{N,\mathrm{loc}}A_j^\top
  +(DP_t^{N,\mathrm{loc}}E^\top)_{jj}
  +(EP_t^{N,\mathrm{loc}}D^\top)_{jj}
  +(EP_t^{N,\mathrm{loc}}E^\top)_{jj}.
\end{align*}
The first term equals
\(\Phi_j(\widehat P_{t,j}^{N,\mathrm{loc}})\) by
\eqref{eq:localized-exact-analysis-block}. Under
\eqref{eq:localized-diagonal-covariance-upper-bound},
\eqref{eq:localized-cross-covariance-bound} gives
\(\|P_{t,ik}^{N,\mathrm{loc}}\|\le \overline p\) for all \(i,k\). Hence
\begin{align*}
  \|(DP_t^{N,\mathrm{loc}}E^\top)_{jj}\|
  &\le a\overline p\sum_i\|E_{ji}\|,
  &
  \|(EP_t^{N,\mathrm{loc}}D^\top)_{jj}\|
  &\le a\overline p\sum_i\|E_{ji}\|,
  \\
  \|(EP_t^{N,\mathrm{loc}}E^\top)_{jj}\|
  &\le \overline p\left(\sum_i\|E_{ji}\|\right)^2.
\end{align*}
Since each block-row sum of \(E\) is bounded by
\(\|E\|_{\mathrm{loc}}\le\varepsilon\), summing these estimates proves
\eqref{eq:localized-perturbed-local-riccati}.
\end{proof}

\begin{remark}[Role of the local transforms]
The matrices \(T_{t,j}^{\mathrm{loc}}\) may differ across blocks. Stacking their
columns still defines a valid centered global ensemble, but it also creates
cross-block empirical covariances. These cross-covariances are neither removed
nor assumed small. Note that \eqref{eq:localized-cross-covariance-bound} shows
that bounded local variances control them, while
\eqref{eq:localized-perturbed-local-riccati} shows that their effect on
a local forecast covariance is multiplied by the weak interaction.
\end{remark}

\subsection{Uniform block-local Riccati theory}
\label{sec:localized-local-riccati}
 In this section, we develop the block-wise Riccati stability estimates
 for
the block-diagonal reference trajectory in Theorem~\ref{thm:localized-deterministic-package}. The proof ideas are straightforward: we apply the global Riccati estimates from Section~\ref{app:riccati-stability} and leverage the uniform bounds on the model parameters in each block given in Assumption~\ref{ass:uniform-local-riccati}. This assumption also ensures $\widehat P_{0,j}^0 \in \mathfrak C_j(\underline p_0, \overline p_0)$. 
  In what follows, 
the subscripts $0$, $\mathrm{loc}$, and $\mathrm{tr}$
in the covariance class parameters 
distinguish the
initialization class, the common local stability class, and the uniform
trajectory class, respectively. 
\begin{theorem}[Uniform block-local Riccati estimates]
\label{thm:uniform-block-local-riccati}
Under Assumption~\ref{ass:uniform-local-riccati}, there are constants
\[
  0<\underline p_{\mathrm{loc}}<\overline p_{\mathrm{loc}}<\infty,
  \qquad
  0<\underline p_{\mathrm{tr}}<\overline p_{\mathrm{tr}}<\infty,
  \qquad
  C_B,C_\Phi<\infty,
  \qquad
  \rho_B,\rho_\Phi\in(0,1),
\]
independent of $j$ and $J$, with the following properties.

First, every local reference trajectory lies strictly inside the common local  
covariance class:
\begin{equation}
  \widehat P_{t,j}^0
  \in
  \mathfrak C_j\left(2\underline p_{\mathrm{loc}},
  \frac12\overline p_{\mathrm{loc}}\right),
  \qquad t\ge0,
  \quad j=1,\ldots,J.
  \label{eq:reference-inner-local-class}
\end{equation}

Moreover, every trajectory initialized in the common covariance class remains
in a uniform trajectory class:
\begin{equation}
  \Phi_j^t(P_j)
  \in\mathfrak C_j(\underline p_{\mathrm{tr}},\overline p_{\mathrm{tr}}),
  \qquad
  P_j\in\mathfrak C_j(\underline p_{\mathrm{loc}},\overline p_{\mathrm{loc}}),
  \quad t\ge0.
  \label{eq:uniform-local-common-trajectory-class}
\end{equation}

Second, for every
$P_j\in\mathfrak C_j(\underline p_{\mathrm{loc}},\overline p_{\mathrm{loc}})$
and all integers $t\ge s\ge0$,
\begin{equation}
  \left\|
  B_j(\Phi_j^{t-1}(P_j))\cdots
  B_j(\Phi_j^s(P_j))
  \right\|
  \le C_B\rho_B^{t-s},
  \label{eq:uniform-local-shifted-closed-loop}
\end{equation}
where the product is the identity when $t=s$.

Third, for every $P_j,Q_j\in
\mathfrak C_j(\underline p_{\mathrm{loc}},\overline p_{\mathrm{loc}})$ and every
$t\ge0$,
\begin{equation}
  \|\Phi_j^t(P_j)-\Phi_j^t(Q_j)\|
  \le C_\Phi\rho_\Phi^{2t}\|P_j-Q_j\|.
  \label{eq:uniform-local-riccati-forgetting}
\end{equation}
More generally, for integers $t\ge s\ge0$,
\begin{equation}
  \|\Phi_j^t(P_j)-\Phi_j^t(Q_j)\|
  \le C_\Phi\rho_\Phi^{2(t-s)}
  \|\Phi_j^s(P_j)-\Phi_j^s(Q_j)\|.
  \label{eq:uniform-local-shifted-forgetting}
\end{equation}
The closed-loop and forgetting estimates in \eqref{eq:uniform-local-shifted-closed-loop}--\eqref{eq:uniform-local-shifted-forgetting} also hold, with possibly different
constants, for any fixed
$0<\underline p^\sharp<\overline p^\sharp<\infty$ whenever the initial
covariances belong to
$\mathfrak C_j(\underline p^\sharp,\overline p^\sharp)$. These constants remain
uniform in $j$ and $J$.
\end{theorem}

\begin{proof}
This is the uniform blockwise application of
Section~\ref{app:riccati-stability}.
Indeed, for the stability of the local Riccati map $\Phi_j^t$,  \eqref{eq:uniform-local-spectral-rates} supplies its forward-stable
and backward-unstable power bounds, while
\eqref{eq:uniform-local-unstable-observability} supplies its finite-window
observability lower bound. Applying Lemmas~\ref{lem:unstable-covariance-bounds}
and \ref{lem:full-covariance-bound} to
$\mathfrak C_j(\underline p_0,\overline p_0)$ gives, uniformly in $j$,
\[
  (\widehat P_{t,j}^0)_{uu}\succeq \underline p_*I_{\mathsf U_j},
  \qquad
  \|\widehat P_{t,j}^0\|\le \overline p_*,
  \qquad t\ge0.
\]
Set
\[
  \underline p_{\mathrm{loc}}:=\frac12\underline p_*,
  \qquad
  \overline p_{\mathrm{loc}}:=2\overline p_*.
\]
This proves \eqref{eq:reference-inner-local-class}.

Applying the same two lemmas to
$\mathfrak C_j(\underline p_{\mathrm{loc}},\overline p_{\mathrm{loc}})$ gives
\eqref{eq:uniform-local-common-trajectory-class}. Applying the closed-loop
restart and forgetting argument from the proof of
Theorem~\ref{thm:enkf-riccati-package} blockwise to the common class
$\mathfrak C_j(\underline p_{\mathrm{loc}},\overline p_{\mathrm{loc}})$ then give
\eqref{eq:uniform-local-shifted-closed-loop}--\eqref{eq:uniform-local-shifted-forgetting}. Applying the same results to any
fixed class $\mathfrak C_j(\underline p^\sharp,\overline p^\sharp)$ proves the
final assertion. All constants depend only on the endpoints of the relevant
covariance class, the uniform local dimension and coordinate-conditioning
bounds, $a_+,h_+,r_-^{-1}$, and the spectral and observability constants in
Assumption~\ref{ass:uniform-local-riccati}; hence none depends on $j$ or $J$.
\end{proof}
Recall the coupled  forecast-to-forecast Riccati map \(\Phi\) and the coupled forecast closed-loop map  \(B\). 
 Define the corresponding maps
obtained by replacing \(A\) with its block-diagonal part \(D\): 
\begin{equation}
  \Phi_D(P):=D\Psi(P)D^\top,
  \qquad
  B_D(P):=D(I-\mathsf K(P)H).
  \label{eq:B3-full-reference-maps}
\end{equation}
 When $P$ is block-diagonal, they stack local maps diagonally.  Stability estimates for these mappings follow from the previous result. 
\begin{corollary}[Block-diagonal stability]
\label{cor:block-diagonal-reference-stability}
For a block-diagonal covariance $P^0$, with blocks
$P_1,\ldots,P_J$, the maps $\Phi_D$ and $B_D$
act blockwise as
\[
  \Phi_D(P^0)
  =\operatorname{diag}(\Phi_1(P_1),\ldots,\Phi_J(P_J)),
  \qquad
  B_D(P^0)
  =\operatorname{diag}(B_1(P_1),\ldots,B_J(P_J)).
\]
Let
\[
  \mathfrak C_{\mathrm{loc}}(\underline p,\overline p)
  :=
  \left\{
    \operatorname{diag}(P_1,\ldots,P_J):
    P_j\in\mathfrak C_j(\underline p,\overline p)\ \text{for every }j
  \right\}.
\]
Then, for  block-diagonal $P^0,Q^0\in
\mathfrak C_{\mathrm{loc}}(\underline p_{\mathrm{loc}},\overline p_{\mathrm{loc}})$,
\begin{equation}
  \|\Phi_D^t(P^0)-\Phi_D^t(Q^0)\|_{\mathrm{loc}}
  \le C_\Phi\rho_\Phi^{2t}\|P^0-Q^0\|_{\mathrm{loc}}.
  \label{eq:block-diagonal-reference-forgetting}
\end{equation}
Moreover, the corresponding block-diagonal closed-loop products satisfy
\begin{equation}
  \left\|
  B_D(\Phi_D^{t-1}(P^0))\cdots
  B_D(\Phi_D^s(P^0))
  \right\|_{\mathrm{loc}}
  \le C_B\rho_B^{t-s}.
  \label{eq:block-diagonal-reference-closed-loop}
\end{equation}
\end{corollary}

\begin{proof}
For a block-diagonal matrix, the symmetric local norm in
\eqref{eq:block-local-norm} equals the maximum operator norm of its
diagonal blocks. Both claims therefore follow by taking the maximum over $j$
in Theorem~\ref{thm:uniform-block-local-riccati}.
\end{proof} 
In particular, for the reference trajectory, set \(B_t^0:=D(I-K_t^0H) = B_D(\widehat P_t^0)\). Corollary~\ref{cor:block-diagonal-reference-stability} gives
\begin{equation}
  \|B_{t-1}^0\cdots B_s^0\|_{\mathrm{loc}}
  \le C_B\rho_B^{t-s},
  \qquad t\ge s.
  \label{eq:B3-reference-product-stability}
\end{equation}
We also note that $\Phi_D^t(\widehat P_0^0) = \widehat P_t^0$. 

\subsubsection{A common neighborhood and regularity estimates}
\label{subsec:common-local-neighborhood}
In this section, we show that the stability estimates of the previous section, along with regularity estimates for the Kalman gain and analysis mappings, hold uniformly over local neighborhoods. 
These results are needed to control perturbations of the reference trajectories. 
Define 
\begin{equation}
  \vartheta_{\mathrm{loc}}
  :=\frac12\min\left\{\underline p_{\mathrm{loc}},
  \frac12\overline p_{\mathrm{loc}}\right\}
  \label{eq:local-stability-radius}
\end{equation}
and
\begin{equation}
  \mathcal U_{t,j}
  :=\left\{P_j\succeq0:
  \|P_j-\widehat P_{t,j}^0\|\le\vartheta_{\mathrm{loc}}\right\}.
  \label{eq:local-stability-neighborhood}
\end{equation}

\begin{lemma}[Uniform local stability neighborhood]
\label{lem:uniform-local-stability-neighborhood}
For every $t\ge0$ and $j=1,\ldots,J$,
\begin{equation}
  \mathcal U_{t,j}
  \subset
  \mathfrak C_j(\underline p_{\mathrm{loc}},\overline p_{\mathrm{loc}}).
  \label{eq:local-neighborhood-contained}
\end{equation}
Consequently, all closed-loop and Riccati estimates in
Theorem~\ref{thm:uniform-block-local-riccati} hold with the same constants for
covariances in these neighborhoods.
\end{lemma}

\begin{proof}
Let $P_j\in\mathcal U_{t,j}$. By
\eqref{eq:reference-inner-local-class},
\[
  (\widehat P_{t,j}^0)_{uu}\succeq2\underline p_{\mathrm{loc}}I,
  \qquad
  \|\widehat P_{t,j}^0\|\le\frac12\overline p_{\mathrm{loc}}.
\]
Since $\vartheta_{\mathrm{loc}}\le \underline p_{\mathrm{loc}}/2$,
\[
  (P_j)_{uu}
  \succeq
  (\widehat P_{t,j}^0)_{uu}
  -\vartheta_{\mathrm{loc}}I
  \succeq \underline p_{\mathrm{loc}}I.
\]
Similarly, $\vartheta_{\mathrm{loc}}\le \overline p_{\mathrm{loc}}/4$ gives
\[
  \|P_j\|
  \le\|\widehat P_{t,j}^0\|+\vartheta_{\mathrm{loc}}
  \le\frac34\overline p_{\mathrm{loc}}<\overline p_{\mathrm{loc}}.
\]
This proves \eqref{eq:local-neighborhood-contained}.
\end{proof}

\begin{lemma}[Uniform local gain and analysis regularity]
\label{lem:uniform-local-gain-analysis-regularity}
There are constants $L_K,L_\Psi<\infty$, independent of $j$ and $J$, such
that, for every
$P_j,Q_j\in\mathfrak C_j(\underline p_{\mathrm{loc}},\overline p_{\mathrm{loc}})$,
\begin{align}
  \|\mathsf K_j(P_j)-\mathsf K_j(Q_j)\|
  &\le L_K\|P_j-Q_j\|,
  \label{eq:uniform-local-gain-lipschitz}
  \\
  \|\Psi_j(P_j)-\Psi_j(Q_j)\|
  &\le L_\Psi\|P_j-Q_j\|.
  \label{eq:uniform-local-analysis-lipschitz}
\end{align}
For any fixed $0<\underline p^\sharp<\overline p^\sharp<\infty$, 
the same
conclusions, with constants depending on $\overline p^\sharp$, hold on
$\mathfrak C_j(\underline p^\sharp,\overline p^\sharp)$, uniformly in $j$ and
$J$.
\end{lemma}

\begin{proof}
Apply the dimension-free argument of
Lemma~\ref{lem:analysis-covariance-difference} with
\((H,R)=(H_j,R_j)\).  Its explicit bounds depend only on an upper bound for
\(\|P_j\|\) and \(\|Q_j\|\), an upper bound for \(\|H_j\|\), and a positive
lower bound for \(R_j\).  Assumption~\ref{ass:uniform-local-riccati} supplies
these uniformly in \(j\), proving \eqref{eq:uniform-local-gain-lipschitz}--\eqref{eq:uniform-local-analysis-lipschitz}.  
Replacing
$\overline p_{\mathrm{loc}}$ by any fixed $\overline p^\sharp$ proves the final
assertion.
\end{proof}

The next estimate is conditional on the perturbed and reference trajectories
remaining in one common covariance class.  This hypothesis will be verified
separately in each application.

\begin{proposition}[Forced local Riccati estimate]
\label{prop:forced-block-local-riccati}
Let $(Q_{t,j})_{t\ge0}$ and $(P_{t,j})_{t\ge0}$ be positive-semidefinite
matrices satisfying
\[
  Q_{t+1,j}=\Phi_j(Q_{t,j})+\zeta_{t,j},
  \qquad
  P_{t+1,j}=\Phi_j(P_{t,j}).
\]
Suppose that there are fixed $\underline p^\sharp,\overline p^\sharp>0$ such that
\begin{equation}
  Q_{t,j},\ P_{t,j},\ \Phi_j(Q_{t,j})
  \in\mathfrak C_j(\underline p^\sharp,\overline p^\sharp)
  \qquad
  \text{for all }t\ge0\text{ and }j.
  \label{eq:forced-local-common-class}
\end{equation}
Then there are $C<\infty$ and $\rho\in(0,1)$, independent of $j$, $J$, and
$t$, such that
\begin{equation}
  \max_j\|Q_{t,j}-P_{t,j}\|
  \le C\rho^t\max_j\|Q_{0,j}-P_{0,j}\|
  +C\sum_{s=0}^{t-1}\rho^{t-1-s}\max_j\|\zeta_{s,j}\|.
  \label{eq:forced-local-riccati-estimate}
\end{equation}
In particular, if
$\sup_{s,j}\|\zeta_{s,j}\|\le\zeta_*$, then
\[
  \sup_{t\ge0}\max_j\|Q_{t,j}-P_{t,j}\|
  \le C\max_j\|Q_{0,j}-P_{0,j}\|
  +\frac{C}{1-\rho}\zeta_*.
\]
\end{proposition}

\begin{proof}
For $t\ge1$, the nonlinear telescoping identity gives
\begin{equation*}
  Q_{t,j}-P_{t,j}
  =\Phi_j^t(Q_{0,j})-\Phi_j^t(P_{0,j})
  +\sum_{s=0}^{t-1}
  \left[
    \Phi_j^{t-1-s}(Q_{s+1,j})
    -\Phi_j^{t-1-s}(\Phi_j(Q_{s,j}))
  \right].
\end{equation*}
All starting points in this identity belong to the common class in \eqref{eq:forced-local-common-class}. The shifted forgetting estimate
in Theorem~\ref{thm:uniform-block-local-riccati} therefore yields
\[
  \|Q_{t,j}-P_{t,j}\|
  \le C\rho^t\|Q_{0,j}-P_{0,j}\|
  +C\sum_{s=0}^{t-1}\rho^{t-1-s}
  \|Q_{s+1,j}-\Phi_j(Q_{s,j})\|.
\]
The last difference equals $\zeta_{s,j}$. Taking the maximum over $j$ proves \eqref{eq:forced-local-riccati-estimate}; the uniform estimate follows
by summing the geometric series.
\end{proof}

\subsection{Deterministic localization bias}
\label{sec:localized-deterministic-bias}
 We now prove the localization-bias estimates \eqref{eq:localized-package-local-bias}
in Theorem~\ref{thm:localized-deterministic-package}. 
These estimates bound the deterministic discrepancy between the
block-diagonal reference trajectory and the fully coupled trajectory.
Theorem~\ref{thm:B3-localization-bias} establishes a stronger version
that allows for non-block-diagonal initial covariances and captures the
geometric decay of the resulting initial transients. 

Let
\(\operatorname{bdiag}(P):=\operatorname{diag}(P_1,\ldots,P_J)\). For the deterministic initial bias, set  
\begin{equation}
  \mathfrak d_0:=\|\widehat P_0-\widehat P_0^0\|_{\mathrm{loc}},
  \qquad
  \mathfrak d_{0,\alpha}:=
  \|\widehat P_0-\widehat P_0^0\|_{\mathcal J_\alpha}.
  \label{eq:B3-initial-bias}
\end{equation}

\begin{theorem}[Spatially decaying deterministic localization bias]
\label{thm:B3-localization-bias}
Under Assumptions~\ref{ass:localized-structure} and
\ref{ass:uniform-local-riccati}, there exist constants
$\varepsilon_*,\mathfrak d_*,\mathfrak d_{*,\alpha}>0$,
$C_{\mathrm{bias}}<\infty$, and $\rho_{\mathrm{bias}}\in(0,1)$, independent
of $J$, such that the following holds.  If
\[
  0\le\varepsilon\le\varepsilon_*,
  \qquad
  \mathfrak d_0\le\mathfrak d_*,
  \qquad
  \mathfrak d_{0,\alpha}\le\mathfrak d_{*,\alpha},
\]
then, for every $t\ge0$,
\begin{align}
  \|\widehat P_t-\widehat P_t^0\|_{\mathrm{loc}}
  &\le C_{\mathrm{bias}}\rho_{\mathrm{bias}}^t\mathfrak d_0
  +C_{\mathrm{bias}}\varepsilon,
  \label{eq:B3-forecast-bias}\\
  \|\widehat P_t-\widehat P_t^0\|_{\mathcal J_\alpha}
  &\le C_{\mathrm{bias}}\rho_{\mathrm{bias}}^t
  \mathfrak d_{0,\alpha}+C_{\mathrm{bias}}\varepsilon.
  \label{eq:B3-decay-bias}
\end{align}
Consequently,
\begin{align}
  \|\widehat P_t-\operatorname{bdiag}(\widehat P_t)\|_{\mathcal J_\alpha}
  &\le C_{\mathrm{bias}}\rho_{\mathrm{bias}}^t
  \label{eq:B3-diagonal-actual-offdiagonal-bias}
  \mathfrak d_{0,\alpha}+C_{\mathrm{bias}}\varepsilon,\\
  \max_j\| \widehat P_{t,j}-\widehat P_{t,j}^0\|
  &\le C_{\mathrm{bias}}\rho_{\mathrm{bias}}^t\mathfrak d_0
  +C_{\mathrm{bias}}\varepsilon.
  \label{eq:B3-diagonal-offdiagonal-bias}
\end{align}
The same estimates hold, in the corresponding norms, for the analysis
covariance and gain differences
$\Psi(\widehat P_t)-\Psi(\widehat P_t^0)$ and
$\mathsf K(\widehat P_t)-\mathsf K(\widehat P_t^0)$.
In particular, for
a block-diagonal initialization $\widehat P_0 = \widehat P_0^0$, 
\begin{equation}
  \|(\widehat P_t)_{jk}\|
  \le C_{\mathrm{bias}}\varepsilon(1+|j-k|)^{-\alpha},
  \qquad j\ne k,
  \quad t\ge0.
  \label{eq:B3-transient-cross-decay}
\end{equation}
\end{theorem}

Essentially, Theorem~\ref{thm:B3-localization-bias} establishes the time-uniform stability of the Riccati trajectory upon small perturbations in the initial covariance and from the weak interactions of the forecast dynamics. The proof proceeds in two steps. First, in a small neighborhood of the block-diagonal covariances, also referred to as a tube, the regularity of the involved maps are preserved. Second, the uncoupled Riccati map $\Phi_D$ defined in \eqref{eq:B3-full-reference-maps} remains exponentially stable in a tube, while the localization bias between $\Phi$ and $\Phi_D$ over finite steps is $O(\varepsilon)$ within the tube. Lemmas~\ref{lem:B3-tube-regularity} and \ref{lem:B3-finite-step-perturbation} make these claims precise,
 and we prove Theorem~\ref{thm:B3-localization-bias} after establishing them.  

In the following, we normalize the block-decay norm so that both norms are submultiplicative, allowing us to introduce a single notation and treat the two cases simultaneously.
Let $C_\alpha$ be
the product constant in \eqref{eq:block-decay-product} and define the
normalized block-decay norm
\begin{equation}
  \|M\|_{\alpha,*}:=C_\alpha\|M\|_{\mathcal J_\alpha}.
  \label{eq:normalized-block-decay-norm}
\end{equation}
Then $\|M_1M_2\|_{\alpha,*}\le
\|M_1\|_{\alpha,*}\|M_2\|_{\alpha,*}$. 
In this section
$\|\cdot\|_\sharp$ denotes either $\|\cdot\|_{\mathrm{loc}}$ or
$\|\cdot\|_{\alpha,*}$. For block-diagonal matrices, these two norms are equivalent.  We denote by $\|\cdot\|_{\mathrm{op},\sharp}$ the operator norm induced by $\|\cdot\|_{\sharp}$, defined for a linear map $T:\mathcal S_+^{d_u}\to \mathcal S_+^{d_u}$ as
\begin{align*}
    \|T\|_{\mathrm{op},\sharp}:=\sup_{\|Z\|_{\sharp}=1}\|T[Z]\|_{\sharp}.
\end{align*}

\begin{lemma}[Regularity in local and decay tubes]
\label{lem:B3-tube-regularity}
For either choice of $\|\cdot\|_\sharp$, there exist constants
$\bar\vartheta_\sharp>0$ and
$C,L_K,L_B,L_0,L_1,c_E<\infty$, independent of $J$, such that the
following holds.  Define  the tube 
\[
  \mathcal V_{t,\sharp}:=
  \{P\in\Splus^{d_u}:\|P-\widehat P_t^0\|_\sharp
  \le\bar\vartheta_\sharp\}.
\]
For $P,Q\in\mathcal V_{t,\sharp}$,
\begin{align}
  \|P\|_\sharp&\le C,
  &
  \|(HPH^\top+R)^{-1}\|_\sharp&\le C,
  \label{eq:B3-inverse-bound}
  \\
  \|\mathsf K(P)\|_\sharp&\le C,
  &
  \|\mathsf K(P)-\mathsf K(Q)\|_\sharp
  &\le L_K\|P-Q\|_\sharp.
  \label{eq:B3-gain-regularity}
\end{align}
For the block-diagonal closed-loop map $B_D$  defined in \eqref{eq:B3-full-reference-maps}, 
\begin{equation}
  \|B_D(P)\|_\sharp\le C,
  \qquad
  \|B_D(P)-B_D(Q)\|_\sharp
  \le L_B\|P-Q\|_\sharp.
  \label{eq:B3-BD-regularity}
\end{equation}
Moreover,
\begin{equation}
  \mathrm D\Phi_D(P)[Z]=B_D(P)ZB_D(P)^\top,
  \label{eq:B3-derivative-identity}
\end{equation}
and
\begin{align}
  \|\mathrm D\Phi_D(P)\|_{\mathrm{op},\sharp}&\le L_0,
  \label{eq:B3-derivative-bound}
  \\
  \|\mathrm D\Phi_D(P)-\mathrm D\Phi_D(Q)\|_{\mathrm{op},\sharp}
  &\le L_1\|P-Q\|_\sharp.
  \label{eq:B3-derivative-regularity}
\end{align}
Finally, for $\varepsilon\le1$,
\begin{equation}
  \|\Phi(P)-\Phi_D(P)\|_\sharp\le c_E\varepsilon.
  \label{eq:B3-one-step-defect}
\end{equation}
\end{lemma}

\begin{proof}
The reference matrices $\widehat P_t^0$ are block diagonal 
with blocks in the common covariance class~\eqref{eq:reference-inner-local-class}, hence  uniformly bounded in both norms. The matrices in the tube are therefore bounded.
Likewise,
$S_t^0:=H\widehat P_t^0H^\top+R$ and $(S_t^0)^{-1}$ are block diagonal and
uniformly bounded in both norms.  Since $H$ is block diagonal,
\[
  HPH^\top+R
  =S_t^0\bigl[I+(S_t^0)^{-1}H(P-\widehat P_t^0)H^\top\bigr].
\]
Choose the tube radius so that the second factor differs from the identity by
at most $1/2$ in $\|\cdot\|_\sharp$.  The Neumann series gives the inverse
bound in \eqref{eq:B3-inverse-bound}. 
The gain matrix $\mathsf K(P)$ is therefore bounded, by its definition.

For $S_P:=HPH^\top+R$ and $S_Q:=HQH^\top+R$, the resolvent identity
\[
  S_P^{-1}-S_Q^{-1}=S_P^{-1}H(Q-P)H^\top S_Q^{-1}
\]
and submultiplicativity imply a Lipschitz inverse bound.  Substituting this
into
\[
  \mathsf K(P)-\mathsf K(Q)
  =(P-Q)H^\top S_P^{-1}
  +QH^\top(S_P^{-1}-S_Q^{-1})
\]
proves \eqref{eq:B3-gain-regularity};
\eqref{eq:B3-BD-regularity} follows because $D$ and $H$ are block diagonal
and uniformly bounded 
and 
\[
B_D(P) - B_D(Q) = D(\mathsf K(Q)-\mathsf K(P))H.
\]

The derivative identity is the standard Riccati derivative formula; see \eqref{eq:one-step-riccati-derivative}. 
It gives \eqref{eq:B3-derivative-bound},
while
\begin{equation*}
  \mathrm D\Phi_D(P)[Z]-\mathrm D\Phi_D(Q)[Z]
  =(B_D(P)-B_D(Q))ZB_D(P)^\top
  +B_D(Q)Z(B_D(P)-B_D(Q))^\top
\end{equation*}
proves \eqref{eq:B3-derivative-regularity}.

Finally, $A=D+E$ gives
\[
  \Phi(P)-\Phi_D(P)
  =E\Psi(P)D^\top+D\Psi(P)E^\top+E\Psi(P)E^\top.
\]
The analysis covariance is uniformly bounded in the tube.  
By \eqref{eq:decay-implies-weak-interaction}, 
$\|E\|_\sharp\le C\varepsilon$ in
both cases.  Submultiplicativity proves \eqref{eq:B3-one-step-defect}.
\end{proof}

\subsubsection{A finite-step perturbation principle}

\begin{lemma}[Stable flow under a uniformly small perturbation]
\label{lem:B3-finite-step-perturbation}
Fix either norm $\|\cdot\|_\sharp$ above.  Under
Lemma~\ref{lem:B3-tube-regularity} and \eqref{eq:B3-reference-product-stability}, there exist
$T_*\ge1$, $\vartheta_*>0$, $\theta_*\in(0,1)$, and $c_*<\infty$, independent of
$J$, such that:

\begin{enumerate}
\item if $\|P-\widehat P_t^0\|_\sharp\le\vartheta_*$, then
\begin{equation}
  \|\Phi_D^{T_*}(P)-\widehat P_{t+T_*}^0\|_\sharp
  \le \theta_*\|P-\widehat P_t^0\|_\sharp;
  \label{eq:B3-finite-step-contraction}
\end{equation}
\item for $\varepsilon$ sufficiently small and the same $P$,
\begin{equation}
  \|\Phi^{T_*}(P)-\Phi_D^{T_*}(P)\|_\sharp
  \le c_*\varepsilon;
  \label{eq:B3-finite-step-consistency}
\end{equation}
\item all intermediate reference and perturbed iterates remain in the tubes of
Lemma~\ref{lem:B3-tube-regularity}.
\end{enumerate}
\end{lemma}

\begin{proof}
Choose $T_*$ so that $C_B^2\rho_B^{2T_*}\le1/4$.  The derivative of
$\Phi_D^{T_*}$ at $\widehat P_t^0$ is
\[
  \mathrm D\Phi_D^{T_*}(\widehat P_t^0)[Z]
  =\mathcal B_{t,T_*}^0Z(\mathcal B_{t,T_*}^0)^\top,
  \qquad
  \mathcal B_{t,T_*}^0:=B_{t+T_*-1}^0\cdots B_t^0; 
\]
 see, for example, \eqref{eq:enkf-riccati-difference-identity}.  Therefore, \eqref{eq:B3-reference-product-stability} of the block-diagonal $\mathcal B_{t,T_*}^0$ implies that the operator norm is at most $1/4$ in either norm $\| \cdot \|_\sharp$ after increasing $T_*$ if necessary.  
Lemma~\ref{lem:B3-tube-regularity} gives uniform bounds and a
uniform Lipschitz constant for the derivative of every one-step map.  By the
chain rule, the derivative of the fixed-length composition is uniformly
Lipschitz in a sufficiently small tube, with a constant independent of $t$ and
$J$.  Shrink the tube so that
\[
  \sup_{\|P-\widehat P_t^0\|_\sharp\le\vartheta_*}
  \|\mathrm D\Phi_D^{T_*}(P)\|_{\mathrm{op},\sharp}\le \theta_*<1.
\]
The fundamental theorem of calculus proves
\eqref{eq:B3-finite-step-contraction}.  Short-time Lipschitz estimates rule
out exit from the regularity tubes during the $T_*$ intermediate steps, by a bootstrap argument. 

For the consistency estimate, start the coupled and reference recursions from
the same $P$, and write
\(P_\ell^A:=\Phi^\ell(P)\) and
\(P_\ell^D:=\Phi_D^\ell(P)\).  By a bootstrap argument,  before exit from the tubes,  \eqref{eq:B3-one-step-defect} and \eqref{eq:B3-derivative-bound} give  
\[
  \|P_{\ell+1}^A-P_{\ell+1}^D\|_\sharp
  \le c_E\varepsilon+L_0\|P_\ell^A-P_\ell^D\|_\sharp.
\]
Since $T_*$ is fixed, iteration gives
\[
  \|P_{T_*}^A-P_{T_*}^D\|_\sharp
  \le c_E\varepsilon\sum_{\ell=0}^{T_*-1}L_0^{\ell}=:c_*\varepsilon.
\]
After decreasing $\vartheta_*$ and the admissible interaction amplitude, the
same estimate ensures that the coupled intermediate iterates remain in the
regularity tubes.  This proves all assertions.
\end{proof}

\begin{proof}[Proof of Theorem~\ref{thm:B3-localization-bias}]
Apply Lemma~\ref{lem:B3-finite-step-perturbation}  
in either norm $\| \cdot \|_\sharp$.
At grid times $\ell T_*$, letting $e_\ell := \|\widehat P_{\ell T_*} - \widehat P_{\ell T_*}^0\|_\sharp$, the error obeys
\[
  e_{\ell+1}\le \theta_*e_{\ell}+c_*\varepsilon.
\]
Choose the initial-error and interaction thresholds so that
$e_0+c_*\varepsilon/(1-\theta_*)$ is strictly below the contraction radius.  Then
induction closes the tube bootstrap and gives
\[
  e_{\ell}\le \theta_*^{\ell}e_0+\frac{c_*}{1-\theta_*}\varepsilon.
\]
Uniform short-time estimates between consecutive grid times give
\eqref{eq:B3-forecast-bias} and \eqref{eq:B3-decay-bias}, with
$\rho_{\mathrm{bias}}=\theta_*^{1/T_*}$ after adjusting constants.

The diagonal and off-diagonal deviations 
\eqref{eq:B3-diagonal-offdiagonal-bias} and \eqref{eq:B3-diagonal-actual-offdiagonal-bias} 
are controlled by \eqref{eq:B3-decay-bias}.
The gain estimate follows from
\eqref{eq:B3-gain-regularity}.  The analysis estimate follows from the exact
identity
\[
  \Psi(P)-\Psi(Q)
  =(I-\mathsf K(P)H)(P-Q)(I-\mathsf K(Q)H)^\top
\]
and the uniform gain 
bound~\eqref{eq:B3-gain-regularity}
in the tube.  Finally,
\eqref{eq:B3-transient-cross-decay} is the entrywise meaning of
\eqref{eq:B3-decay-bias} when the reference covariance is block diagonal.
\end{proof}
 The next lemma demonstrates that a small enough perturbation of an exponentially stable product remains exponentially stable. 
\begin{lemma}[Robustness of exponentially stable products]
\label{lem:B3-robust-products}
Let a sequence of matrices $(C_t^0)_{t\ge0}$ satisfy, in a submultiplicative norm,
\[
  \|C_{t-1}^0\cdots C_s^0\|\le C_{\mathrm{ref}}\rho_{\mathrm{ref}}^{t-s},
  \qquad t\ge s,
\]
which implies $\sup_t\|C_t^0\| \le C_{\mathrm{fac}}:= C_{\mathrm{ref}}\rho_{\mathrm{ref}}$. 
There exists $\eta_*>0$, depending
only on $C_{\mathrm{ref}}$ and $\rho_{\mathrm{ref}}$,
such that
$\sup_t\|C_t-C_t^0\|\le\eta_*$ implies
\begin{equation}
  \|C_{t-1}\cdots C_s\|\le C\rho^{t-s},
  \qquad t\ge s,
  \label{eq:B3-robust-product-bound}
\end{equation}
for some $C<\infty$ and $\rho\in(0,1)$.
\end{lemma}

\begin{proof}
Choose $T\ge1$ so that
$C_{\mathrm{ref}}\rho_{\mathrm{ref}}^T\le1/4$.  For $\eta_*\le1$, all
perturbed factors have norm at most $C_{\mathrm{fac}}^+:=C_{\mathrm{fac}}+1$.  The telescoping identity
for a product of $T$ factors gives
\[
  \|C_{s+T-1}\cdots C_s-C_{s+T-1}^0\cdots C_s^0\|
  \le T(C_{\mathrm{fac}}^+)^{T-1}\eta_*.
\]
Choose $\eta_*$ so that the last quantity is at most $1/4$.  Every product of
$T$ perturbed factors then has norm at most $1/2$.  Splitting a general
product into complete blocks of length $T$ and a remainder proves
\eqref{eq:B3-robust-product-bound}.
\end{proof}
Using this lemma, we can now show stability of the fully-coupled 
closed-loop maps. 
Note that we have not assumed hyperbolicity or detectability of the fully coupled pair $(A,H).$ 
\begin{corollary}[Full coupled closed-loop stability]
\label{cor:B3-full-closed-loop-stability}
Under the smallness conditions of
Theorem~\ref{thm:B3-localization-bias}, after decreasing the thresholds if
necessary, let
\begin{equation}
  B_t:=A(I-\mathsf K(\widehat P_t)H).
\label{eq:B_t-def}
\end{equation}
There exist $C_{\mathrm{full}}<\infty$ and
$\rho_{\mathrm{full}}\in(0,1)$, independent of $J$, such that
\begin{equation}
  \|B_{t-1}\cdots B_s\|_{\sharp}
  \le C_{\mathrm{full}}\rho_{\mathrm{full}}^{t-s},
  \qquad t\ge s\ge0.
  \label{eq:B3-full-closed-loop-bound}
\end{equation}
\end{corollary}

\begin{proof}
Using $A=D+E$,
\[
  B_t-B_t^0
  =E(I-\mathsf K(\widehat P_t)H)
  +D\bigl(\mathsf K(\widehat P_t^0)-\mathsf K(\widehat P_t)\bigr)H.
\] 
Theorem~\ref{thm:B3-localization-bias},
Lemma~\ref{lem:B3-tube-regularity}, and \eqref{eq:decay-implies-weak-interaction} 
show that this difference is uniformly
$O(\mathfrak d_0+\varepsilon)$ in the local norm and
$O(\mathfrak d_{0,\alpha}+\varepsilon)$ in the decay norm.
Apply Lemma~\ref{lem:B3-robust-products} to the corresponding reference products, as the reference products are block-diagonal and exponentially stable in both norms by \eqref{eq:B3-reference-product-stability}.
\end{proof}

\subsection{Localized initialization and empirical covariance control}
\label{sec:localized-empirical-covariance} 

We now control the sampling error.
Section~\ref{sec:localized-gaussian-initialization} shows that, with
high probability, the initial local empirical covariances are
simultaneously within
\(\Delta_{N,\mathrm{loc}}(\delta)\) of their exact counterparts over
all \(J\) blocks. Section~\ref{sec:localized-uniform-empirical-covariance-control}
then shows that, on this event, the local empirical covariances remain
in a common stability class and stay within a geometrically decaying
sampling error plus an \(O(\varepsilon)\) interaction error of the
block-diagonal reference trajectory, uniformly in time. These
estimates are collected in
Proposition~\ref{prop:localized-covariance-package}.
\subsubsection{Simultaneous Gaussian initialization}
\label{sec:localized-gaussian-initialization}

\begin{proposition}[Simultaneous local Gaussian initialization]
\label{prop:B5-local-gaussian-initialization}
Assume \(\widehat P_0=\widehat P_0^0\), and initialize
\[
  \widehat u_0^{(1),\mathrm{loc}},\ldots,
  \widehat u_0^{(N),\mathrm{loc}}
  \stackrel{\mathrm{i.i.d.}}{\sim}
  \mathcal N(\widehat m_0,\widehat P_0).
\]
Then, for every \(\delta\in(0,1)\), with probability at least
\(1-\delta/4\),
\begin{equation}
  \max_{1\le j\le J}
  \|\widehat P_{0,j}^{N,\mathrm{loc}}-\widehat P_{0,j}\|
  \le \Delta_{N,\mathrm{loc}}(\delta).
  \label{eq:B5-local-covariance-concentration}
\end{equation}
We denote this event by $\mathcal G_{N,\delta}^{\mathrm{cov,loc}}$. Moreover, the vector of local empirical means is independent of the
collection of local empirical covariance blocks.
\end{proposition}

\begin{proof}
For each fixed block, apply
Lemma~\ref{lem:gaussian-wishart-reduction} to the projected ensemble, with
covariance \(\widehat P_{0,j}\).  The effective-rank concentration inequality
of \cite{KoltchinskiiLounici2017}, with tail parameter
\(\log(4J/\delta)\), gives the corresponding covariance bound for each fixed
block with failure probability at most \(\delta/(4J)\). A union bound gives
\eqref{eq:B5-local-covariance-concentration}; no independence across
blocks is used.

Finally, Lemma~\ref{lem:gaussian-wishart-reduction} shows that the global
sample mean is independent of the centered anomaly matrix.  The local means
are projections of the former, whereas the local covariance blocks are
functions of the latter, proving the independence assertion.
\end{proof}

\subsubsection{Uniform empirical covariance control}\label{sec:localized-uniform-empirical-covariance-control}

Define also the maximal local covariance error
\begin{equation}
  \mathfrak e_t^N:=\max_{1\le j\le J}
  \|\widehat P_{t,j}^{N,\mathrm{loc}}-\widehat P_{t,j}^0\|.
  \label{eq:B4-local-empirical-error}
\end{equation}

We now prove that when the initial empirical covariance is close enough to the reference initial covariance and the spatial interaction is weak enough, all the empirical local covariances remain inside of the local stability neighborhood $\mathcal U_{t,j}$ in \eqref{eq:local-stability-neighborhood}, and the maximal local covariance error is controlled. We invoke a bootstrap argument and proceed as in the proof of Proposition~\ref{prop:forced-block-local-riccati}.

\begin{theorem}[Uniform empirical covariance stability and accuracy]
\label{thm:B4-empirical-covariance-stability}
Under Assumptions~\ref{ass:localized-structure} and
\ref{ass:uniform-local-riccati}, there exist constants
\[
  \varepsilon_{\mathrm{emp}}>0,
  \qquad
  \vartheta_{\mathrm{emp}}>0,
  \qquad
  C_{\mathrm{emp}}<\infty,
  \qquad
  \rho_{\mathrm{emp}}\in(0,1),
\]
independent of the number of blocks~$J$ and the ensemble size~$N$, such that
the following holds.  If
\begin{equation}
  0\le\varepsilon\le\varepsilon_{\mathrm{emp}},
  \qquad
  \mathfrak e_0^N\le\vartheta_{\mathrm{emp}},
  \label{eq:B4-smallness-assumptions}
\end{equation}
then, for every $t\ge0$ and every block $j$,
\begin{equation}
  \widehat P_{t,j}^{N,\mathrm{loc}}\in\mathcal U_{t,j}
  \subset\mathfrak C_j(\underline p_{\mathrm{loc}},
  \overline p_{\mathrm{loc}}).
  \label{eq:B4-empirical-covariance-invariance}
\end{equation}
In particular, the local empirical covariances remain in the common stability
class for all time.  Moreover,
\begin{equation}
  \mathfrak e_t^N
  \le C_{\mathrm{emp}}\rho_{\mathrm{emp}}^t \mathfrak e_0^N
  +C_{\mathrm{emp}}\varepsilon,
  \qquad t\ge0.
  \label{eq:B4-empirical-covariance-bound}
\end{equation}

Then the local analysis covariances and gains satisfy
\begin{align}
  \max_j\|P_{t,j}^{N,\mathrm{loc}}-P_{t,j}^0\|
  &\le C_{\mathrm{emp}}\rho_{\mathrm{emp}}^t \mathfrak e_0^N
  +C_{\mathrm{emp}}\varepsilon,
  \label{eq:B4-analysis-covariance-bound}\\
  \max_j\|K_{t,j}^{N,\mathrm{loc}}-K_{t,j}^0\|
  &\le C_{\mathrm{emp}}\rho_{\mathrm{emp}}^t \mathfrak e_0^N
  +C_{\mathrm{emp}}\varepsilon.
  \label{eq:B4-local-gain-bound}
\end{align}
\end{theorem}

\begin{proof}
Set
\[
  \lambda:=\rho_\Phi^2\in(0,1),
  \qquad
  C_{\Phi,0}:=\max\{1,C_\Phi\},
  \qquad
  c_\zeta:=(2a_++1)\overline p_{\mathrm{loc}}.
\]
Here $C_\Phi$ and $\rho_\Phi$ are the uniform local forgetting constants in
Theorem~\ref{thm:uniform-block-local-riccati}.  We use
$0\le\varepsilon\le1$, as already imposed in
Assumption~\ref{ass:localized-structure}.

We first choose the smallness constants.  Let
\[
  r_*:=\frac12\vartheta_{\mathrm{loc}}.
\]
Choose $\vartheta_{\mathrm{emp}}>0$ and $\varepsilon_{\mathrm{emp}}>0$ so that
\begin{equation}
  \vartheta_{\mathrm{emp}}\le\frac{\vartheta_{\mathrm{loc}}}{8C_{\Phi,0}},
  \qquad
  c_\zeta\varepsilon_{\mathrm{emp}}\le r_*,
  \qquad
  \frac{C_{\Phi,0}c_\zeta}{1-\lambda}\varepsilon_{\mathrm{emp}}
  \le\frac18\vartheta_{\mathrm{loc}}.
  \label{eq:B4-choice-smallness}
\end{equation}
All these constants depend only on the uniform parameters in
Theorem~\ref{thm:uniform-block-local-riccati} and are independent of $j,J,N$.

Define the first exit time
\[
  \tau_{\mathrm{exit}}:=\inf\{t\ge0:\mathfrak e_t^N>r_*\},
\]
with the convention $\inf\varnothing=\infty$.  By \eqref{eq:B4-smallness-assumptions} and
\eqref{eq:B4-choice-smallness}, one has $\tau_{\mathrm{exit}}\ge1$.
For every $s<\tau_{\mathrm{exit}}$,
\[
  \widehat P_{s,j}^{N,\mathrm{loc}}\in\mathcal U_{s,j}
  \subset\mathfrak C_j(\underline p_{\mathrm{loc}},\overline p_{\mathrm{loc}})
\]
by Lemma~\ref{lem:uniform-local-stability-neighborhood}.  Since analysis
reduces covariance,
\[
  0\preceq P_{s,j}^{N,\mathrm{loc}}
  =\Psi_j(\widehat P_{s,j}^{N,\mathrm{loc}})\preceq \widehat P_{s,j}^{N,\mathrm{loc}},
\]
and hence
\[
  \max_j\|P_{s,j}^{N,\mathrm{loc}}\|\le \overline p_{\mathrm{loc}}.
\]
Proposition~\ref{prop:localized-exact-covariance-identities} therefore gives,
for every $s<\tau_{\mathrm{exit}}$,
\begin{equation}
  \widehat P_{s+1,j}^{N,\mathrm{loc}}=\Phi_j(\widehat P_{s,j}^{N,\mathrm{loc}})+\zeta_{s,j}^N,
  \qquad
  \max_j\|\zeta_{s,j}^N\|\le c_\zeta\varepsilon.
  \label{eq:B4-stopped-forcing-bound}
\end{equation}
Notice that this estimate has been derived only before the first exit; no
invariance conclusion has been used to prove it.

We next establish a stopped error estimate.  Fix an integer
$1\le t\le\tau_{\mathrm{exit}}$.  The nonlinear telescoping identity gives
\begin{equation}
  \widehat P_{t,j}^{N,\mathrm{loc}}-\widehat P_{t,j}^0
  =\Phi_j^t(\widehat P_{0,j}^{N,\mathrm{loc}})-\Phi_j^t(\widehat P_{0,j}^0)
  +
  \sum_{s=0}^{t-1}
  \left[
    \Phi_j^{t-1-s}(\widehat P_{s+1,j}^{N,\mathrm{loc}})
    -\Phi_j^{t-1-s}(\Phi_j(\widehat P_{s,j}^{N,\mathrm{loc}}))
  \right].
  \label{eq:B4-stopped-telescoping}
\end{equation}
We verify carefully that the forgetting estimate may be applied to every term
with a positive iterate.  If $0\le s\le t-2$, then
$s+1\le t-1<\tau_{\mathrm{exit}}$, so
\[
  \widehat P_{s+1,j}^{N,\mathrm{loc}}\in\mathcal U_{s+1,j}.
\]
Furthermore, using \eqref{eq:B4-stopped-forcing-bound},
\begin{equation*}
  \|\Phi_j(\widehat P_{s,j}^{N,\mathrm{loc}})-\widehat P_{s+1,j}^0\|
  \le
  \|\widehat P_{s+1,j}^{N,\mathrm{loc}}-\widehat P_{s+1,j}^0\|
  +\|\zeta_{s,j}^N\|
  \le r_*+c_\zeta\varepsilon
  \le\vartheta_{\mathrm{loc}}.
\end{equation*}
Thus both $\widehat P_{s+1,j}^{N,\mathrm{loc}}$ and $\Phi_j(\widehat P_{s,j}^{N,\mathrm{loc}})$ belong to
$\mathcal U_{s+1,j}$ and hence to the common covariance class
$\mathfrak C_j(\underline p_{\mathrm{loc}},\overline p_{\mathrm{loc}})$.  The final summand,
$s=t-1$, has zero iterates and equals $\zeta_{t-1,j}^N$ directly; it requires
no covariance-class assumption on $\widehat P_{t,j}^{N,\mathrm{loc}}$.

Applying \eqref{eq:uniform-local-riccati-forgetting} to the initial
term and to all positive-iterate summands in
\eqref{eq:B4-stopped-telescoping}, while using $C_{\Phi,0}\ge1$ for the
last summand,
yields 
\begin{equation}
  \mathfrak e_t^N
  \le
  C_{\Phi,0}\lambda^t \mathfrak e_0^N
  +C_{\Phi,0}\sum_{s=0}^{t-1}\lambda^{t-1-s}
  \max_j\|\zeta_{s,j}^N\|
  \le
  C_{\Phi,0}\lambda^t \mathfrak e_0^N
  +\frac{C_{\Phi,0}c_\zeta}{1-\lambda}\varepsilon.
  \label{eq:B4-stopped-error-estimate}
\end{equation}
By the choices in \eqref{eq:B4-choice-smallness}, the right-hand side
is at most $\vartheta_{\mathrm{loc}}/4$.  If
$\tau_{\mathrm{exit}}<\infty$, taking $t=\tau_{\mathrm{exit}}$
in \eqref{eq:B4-stopped-error-estimate} would give
$\mathfrak e_{\tau_{\mathrm{exit}}}^N
\le\vartheta_{\mathrm{loc}}/4<r_*$, contradicting the definition of
$\tau_{\mathrm{exit}}$.  Hence $\tau_{\mathrm{exit}}=\infty$, proving
\eqref{eq:B4-empirical-covariance-invariance}.

Since there is no exit, \eqref{eq:B4-stopped-error-estimate} holds for
every $t\ge1$; the case $t=0$ is immediate.  The bound \eqref{eq:B4-empirical-covariance-bound} follows with
$\rho_{\mathrm{emp}}:=\lambda$ and a constant
$C_{\mathrm{emp}}$ depending only on $C_{\Phi,0},c_\zeta$, and $1-\lambda$.
Finally, \eqref{eq:localized-exact-analysis-block},
\eqref{eq:uniform-local-analysis-lipschitz}, and
\eqref{eq:uniform-local-gain-lipschitz} give
\eqref{eq:B4-analysis-covariance-bound} and
\eqref{eq:B4-local-gain-bound}, 
after enlarging
$C_{\mathrm{emp}}$. 
\end{proof}

\begin{corollary}[Entry into the empirical stability regime]
\label{cor:B5-entry-empirical-regime}
Under the hypotheses of
Proposition~\ref{prop:B5-local-gaussian-initialization}, 
on $\mathcal G_{N,\delta}^{\mathrm{cov,loc}}$,
\begin{equation}
  \mathfrak e_0^N\le \Delta_{N,\mathrm{loc}}(\delta).
  \label{eq:B5-initial-error-to-reference}
\end{equation}
There is a constant \(C_{\mathrm{init,loc}}<\infty\), independent of
\(J,N,\delta\) and depending only on the uniform model parameters, such that
\begin{equation}
  N-1\ge C_{\mathrm{init,loc}}
  \left[
    r_{\mathrm{eff},\mathrm{loc}}+\log\left(\frac{4J}{\delta}\right)
  \right]
  \label{eq:B5-local-ensemble-size}
\end{equation}
implies
\begin{equation}
  \Delta_{N,\mathrm{loc}}(\delta)\le\vartheta_{\mathrm{emp}}.
  \label{eq:B5-entry-condition}
\end{equation}
Consequently, if \(\varepsilon\le\varepsilon_{\mathrm{emp}}\), all conclusions
of Theorem~\ref{thm:B4-empirical-covariance-stability} hold with probability
at least \(1-\delta/4\). 
\end{corollary}

\begin{proof}
Since \(\widehat P_0=\widehat P_0^0\) is block diagonal,
\eqref{eq:B5-local-covariance-concentration} gives
\eqref{eq:B5-initial-error-to-reference}. Moreover,
Assumption~\ref{ass:uniform-local-riccati} gives
\(\max_j\|\widehat P_{0,j}\|\le\overline p_0\). Thus a sufficiently large
constant in \eqref{eq:B5-local-ensemble-size} makes
\[
  \Delta_{N,\mathrm{loc}}(\delta)
  \le C_0\overline p_0
  \bigl(C_{\mathrm{init,loc}}^{-1/2}+C_{\mathrm{init,loc}}^{-1}\bigr)
  \le\vartheta_{\mathrm{emp}}.
\]
The last assertion follows from
Theorem~\ref{thm:B4-empirical-covariance-stability}.
\end{proof}

\begin{remark}[Exact empirical recursion]
The proof does not assume that the empirical recursion remains in a stable
covariance class.  It derives the coupling defect only up to a first exit,
uses local Riccati forgetting to show that such an exit is impossible, and
thereby closes the bootstrap.  The only smallness requirements are the initial
local covariance error and the weak spatial-interaction size~$\varepsilon$.
\end{remark}

\subsection{Localized mean estimates}
\label{app:localized-mean-auxiliary}
This section develops the estimates needed to control the local mean
errors,  centering on the decomposition~\eqref{eq:mean-err-iterate-from-initial-localized}.
Lemma~\ref{lem:B7-gaussian-block-maximum} controls moments of
Gaussian block vectors in the block-maximum norm, and
Lemma~\ref{lem:B7-local-innovation-moments} gives time-uniform
innovation bounds. On the high-probability initialization event,
Proposition~\ref{prop:B7-refined-gain-decomposition} separates the
localized gain error into a geometrically decaying sampling component
and an \(O(\varepsilon)\) localization component. Finally,
Lemma~\ref{lem:B7-localized-closed-loop-stability} establishes
exponential stability of the associated localized closed-loop
products. These estimates are collected in
Proposition~\ref{prop:localized-mean-package}.
\subsubsection{Gaussian block-maximum moments}

\begin{lemma}[Gaussian block-maximum moment bound]
\label{lem:B7-gaussian-block-maximum}
Let \(Z=(Z_1,\ldots,Z_J)\) be a centered jointly Gaussian block vector. If
\(\max_j\|\operatorname{Cov}(Z_j)\|\le M_Z\) and
\(\max_j\operatorname{Tr}(\operatorname{Cov}(Z_j))\le\tau_Z\), then, for
every \(q\ge1\),
\begin{equation}
  \left(\mathbb E\|Z\|_{\max}^q\right)^{1/q}
  \le C\left[\sqrt{\tau_Z}+\sqrt{M_Z(q+\log(2J))}\right].
  \label{eq:B7-gaussian-block-maximum}
\end{equation}
\end{lemma}

\begin{proof}
Gaussian concentration and
\(\mathbb E\|Z_j\|\le
\sqrt{\operatorname{Tr}(\operatorname{Cov}(Z_j))}\) give
\[
  \mathbb P\left(
  \|Z_j\|>\sqrt{\tau_Z}+\sqrt{2M_Zx}
  \right)\le e^{-x},
  \qquad x\ge0.
\]
A union bound over the blocks, followed by the standard sub-Gaussian
tail-to-moment implication, proves
\eqref{eq:B7-gaussian-block-maximum}.  
We use the term \(\log(2J)\) to also cover \(J=1\).
\end{proof}

\subsubsection{Innovation, gain, and stability estimates}

For the mean estimates, 
recall that \(\mathcal F_0^{X,N}:=\sigma(\widehat X_0^{\mathrm{loc}})\),  and define
\begin{align}
  \overline p_{0,\mathrm{loc}}&:=\max_j\|\widehat P_{0,j}\|,
  &\tau_{0,\mathrm{loc}}&:=\max_j\operatorname{Tr}(\widehat P_{0,j}),
  \notag\\
  \mu_{N,\mathrm{loc}}^{(q)}&:=
  \sqrt{\frac{\tau_{0,\mathrm{loc}}}{N}}
  +\sqrt{\frac{\overline p_{0,\mathrm{loc}}(q+\log(2J))}{N}},
  \label{eq:B7-local-initial-moment-scales}\\
  M_{\iota,\mathrm{loc}}&:=\sup_t\max_j\|(S_t)_{jj}\|,
  &\tau_{\iota,\mathrm{loc}}&:=
  \sup_t\max_j\operatorname{Tr}((S_t)_{jj}),
  \notag\\
  \mathcal I_{q,\mathrm{loc}}&:=
  \sqrt{\tau_{\iota,\mathrm{loc}}}
  +\sqrt{M_{\iota,\mathrm{loc}}(q+\log(2J))}.
  \label{eq:B7-local-innovation-scales}
\end{align}
Here  
\(\iota_t:=y_t-H\widehat m_t\) and
\(S_t:=H\widehat P_tH^\top+R\)   are defined in Sections~\ref{subsec:enkf-proofs} and \ref{subsec:localized-proofs}.

The Gaussian block-maximum estimate controls the initial ensemble
mean deviation.
By Lemma~\ref{lem:gaussian-wishart-reduction},
\(\widehat m_0^{N,\mathrm{loc}}-\widehat m_0\) is independent of
\(\mathcal F_0^{X,N}\) and has distribution
\(\mathcal N(0,\widehat P_0/N)\). Applying
Lemma~\ref{lem:B7-gaussian-block-maximum} conditionally, with
\(M_Z=\overline p_{0,\mathrm{loc}}/N\) and
\(\tau_Z=\tau_{0,\mathrm{loc}}/N\), gives
\begin{equation}
  \|\widehat m_0^{N,\mathrm{loc}}-\widehat m_0\|_{q\mid X}
  \le C\mu_{N,\mathrm{loc}}^{(q)}.
  \label{eq:B7-local-initial-mean-moment}
\end{equation}

We next obtain the corresponding time-uniform bound for the innovations.

\begin{lemma}[Local innovation moments]
\label{lem:B7-local-innovation-moments}
Under the linear--Gaussian model
\eqref{eq:initial-state-law}--\eqref{eq:observation-model},
Assumptions~\ref{ass:localized-structure} and
\ref{ass:uniform-local-riccati}, and the smallness conditions of
Theorem~\ref{thm:B3-localization-bias}, assume in addition that the initial
ensemble is independent of the signal and observation noises. Then the
innovations $(\iota_t)_{t\ge0}$ are independent centered Gaussian vectors,
independent of the initial ensemble, and the constants in
\eqref{eq:B7-local-innovation-scales} are finite and independent of
$J$.  Moreover, for every $q\ge1$,
\begin{equation}
  \sup_{t\ge0}
  \|\iota_t\|_{q\mid X}
  =
  \sup_{t\ge0}
  \left(\mathbb E\|\iota_t\|_{\max}^q\right)^{1/q}
  \le C\mathcal I_{q,\mathrm{loc}}.
  \label{eq:B7-local-innovation-moments}
\end{equation}
\begin{proof}
The Gaussian-innovation part of
Lemma~\ref{lem:weighted-innovation-bound} gives the asserted Gaussianity and
independence; the assumed independence of the localized ensemble from the
signal and observation noises also makes it independent of the innovations.
Thus conditional and unconditional innovation moments coincide.
Combining
Theorem~\ref{thm:B3-localization-bias} with the uniform reference bounds in
Theorem~\ref{thm:uniform-block-local-riccati} gives a constant $C<\infty$,
independent of $J$, such that
\[
  \sup_{t\ge0}\max_j\|\widehat P_{t,j}\|\le C.
\]
Because $H$ and $R$ are block diagonal,
\[
  (S_t)_{jj}=H_j\widehat P_{t,j}H_j^\top+R_j.
\]
Assumption~\ref{ass:uniform-local-riccati} therefore gives
\[
  M_{\iota,\mathrm{loc}}
  \le h_+^2C+r_+,
  \qquad
  \tau_{\iota,\mathrm{loc}}
  \le d_{y,\mathrm{loc}}(h_+^2C+r_+).
\]
The bound \eqref{eq:B7-local-innovation-moments} now follows from
Lemma~\ref{lem:B7-gaussian-block-maximum}.
\end{proof}
\end{lemma}

To quantify the localized gain error between $K_t^{N,\mathrm{loc}}$ and $K_t$, we introduce two auxiliary uncoupled local covariance trajectories
\begin{equation}
  \widehat Q_{t,j}:=\Phi_j^t(\widehat P_{0,j}),
  \qquad
  \widehat Q_{t,j}^N
  :=\Phi_j^t(\widehat P_{0,j}^{N,\mathrm{loc}}),
  \label{eq:B7-auxiliary-local-covariances}
\end{equation}
which are initialized at the exact and the empirical local covariances. Denote their block-diagonal gains by
\[
  K_t^{\mathrm{marg}}
  :=\operatorname{diag}(\mathsf K_j(\widehat Q_{t,j}))_{j=1}^J,
  \qquad
  \widetilde K_t^N
  :=\operatorname{diag}(\mathsf K_j(\widehat Q_{t,j}^N))_{j=1}^J.
\]
We decompose the localized gain error into two parts: the deviation of $\widetilde K_t^N$ from $K_t^{\mathrm{marg}}$ is due to sampling error, while  the discrepancies between $K_t^{N,\mathrm{loc}}$ and $\widetilde K_t^N$, and between $K_t^{\mathrm{marg}}$ and $K_t$, are localization bias. The sampling error is propagated by local Riccati maps and is thus controlled by the local Riccati forgetting estimate. For the localization bias, $K_t^{N,\mathrm{loc}}$ and $\widetilde K_t^N$ differ only in whether the forecast is global or local, while $K_t^{\mathrm{marg}}$ and $K_t$ differ in both the forecast and the analysis steps. We therefore apply Proposition~\ref{prop:forced-block-local-riccati} and Theorem~\ref{thm:B3-localization-bias} separately  to these bias terms.

\begin{proposition}[Refined localized gain decomposition]
\label{prop:B7-refined-gain-decomposition}
Assume \(\widehat P_0=\widehat P_0^0\).  If \(\varepsilon\) is sufficiently small
and 
$C_{\mathrm{init,loc}}$
 in \eqref{eq:B6-local-ensemble-size} is sufficiently
large, 
on the event \(\mathcal G_{N,\delta}^{\mathrm{cov,loc}}\in\mathcal F_0^{X,N}\)
defined in Proposition~\ref{prop:B5-local-gaussian-initialization} 
which occurs with
probability at least \(1-\delta\), 
\begin{equation}
  K_t^{N,\mathrm{loc}}-K_t
  =\Gamma_t^{\mathrm{sam}}+\Gamma_t^{\mathrm{loc}},
  \label{eq:B7-gain-decomposition}
\end{equation}
where
\[
  \Gamma_t^{\mathrm{sam}}:=\widetilde K_t^N-K_t^{\mathrm{marg}},
  \qquad
  \Gamma_t^{\mathrm{loc}}:=(K_t^{N,\mathrm{loc}}-\widetilde K_t^N)
  +(K_t^{\mathrm{marg}}-K_t),
\]
and, for constants \(C_\Gamma<\infty\) and \(\rho_\Gamma\in(0,1)\) independent of
\(J,N,\delta,t\),
\begin{align}
  \|\Gamma_t^{\mathrm{sam}}\|_{\mathrm{loc}}
  &\le C_\Gamma\rho_\Gamma^t\Delta_{N,\mathrm{loc}}(\delta),
  \label{eq:B7-sampling-gain-bound}\\
  \|\Gamma_t^{\mathrm{loc}}\|_{\mathrm{loc}}
  &\le C_\Gamma\varepsilon.
  \label{eq:B7-localization-gain-bound}
\end{align}
\begin{proof}
Proposition~\ref{prop:B5-local-gaussian-initialization} 
gives
\(
\mathbb P(\mathcal G_{N,\delta}^{\mathrm{cov,loc}})
\ge1-\delta/4\ge1-\delta
\), and this event depends only on the centered initial anomaly matrix.
After enlarging the sample-size constant and decreasing the model smallness
constants if necessary, Corollary~\ref{cor:B5-entry-empirical-regime} and
Theorem~\ref{thm:B4-empirical-covariance-stability} place the empirical
initial covariances in
$\mathfrak C_j(\underline p_{\mathrm{loc}},\overline p_{\mathrm{loc}})$.
The exact and reference initial covariances already belong to this class, so
the two auxiliary trajectories remain in the uniform trajectory class from
\eqref{eq:uniform-local-common-trajectory-class}.

The two auxiliary trajectories in
\eqref{eq:B7-auxiliary-local-covariances} evolve under the same local
Riccati maps.  Therefore, \eqref{eq:B5-local-covariance-concentration},
\eqref{eq:uniform-local-riccati-forgetting}, and
\eqref{eq:uniform-local-gain-lipschitz}  give
\[
  \|\widetilde K_t^N-K_t^{\mathrm{marg}}\|_{\mathrm{loc}}
  \le C\rho_\Gamma^t\Delta_{N,\mathrm{loc}}(\delta),
\]
which proves \eqref{eq:B7-sampling-gain-bound}.

Next compare the actual empirical local covariance $\widehat P_{t,j}^{N,\mathrm{loc}}$ with the auxiliary
uncoupled covariance $\widehat Q_{t,j}^N$ which has the same initial value.  On
$\mathcal G_{N,\delta}^{\mathrm{cov,loc}}$,
Theorem~\ref{thm:B4-empirical-covariance-stability} ensures $\widehat P_{t,j}^{N,\mathrm{loc}} \in \mathfrak C_j(\underline p_{\mathrm{loc}}, \overline p_{\mathrm{loc}})$ for all time, while
Theorem~\ref{thm:uniform-block-local-riccati} gives $\Phi_j(\widehat P_{t,j}^{N,\mathrm{loc}}), \widehat Q_{t,j}^N \in \mathfrak C_j(\underline p_{\mathrm{tr}},\overline p_{\mathrm{tr}})$.  
Take \[
  \underline p^\sharp:=\min\{\underline p_{\mathrm{loc}},\underline p_{\mathrm{tr}}\},
  \qquad
  \overline p^\sharp:=\max\{\overline p_{\mathrm{loc}},\overline p_{\mathrm{tr}}\}.
\]
Then all three covariances belong to $\mathfrak C(\underline p^\sharp, \overline p^\sharp)$.  
Besides, by \eqref{eq:localized-perturbed-local-riccati}, the recursions of $\widehat P_{t,j}^{N,\mathrm{loc}}$ and $\widehat Q_{t,j}^N$ align with Proposition~\ref{prop:forced-block-local-riccati}, with the residual in the $\widehat P_{t,j}^{N,\mathrm{loc}}$ recursion bounded by
$C\varepsilon$. Proposition~\ref{prop:forced-block-local-riccati} gives
\[
  \sup_{t\ge0}\max_j
  \|\widehat P_{t,j}^{N,\mathrm{loc}}-\widehat Q_{t,j}^N\|
  \le C\varepsilon.
\]
Uniform local gain Lipschitzness \eqref{eq:uniform-local-gain-lipschitz}
 yields 
\begin{equation}
  \|K_t^{N,\mathrm{loc}}-\widetilde K_t^N\|_{\mathrm{loc}}
  \le C\varepsilon.
  \label{eq:B7-actual-auxiliary-gain-bound}
\end{equation}

Because the initial covariance is block diagonal,
\(\widehat P_{0,j}=\widehat P_{0,j}^0\), so
\(\widehat Q_{t,j}=\widehat P_{t,j}^0\) and
\(K_t^{\mathrm{marg}}=K_t^0\).  The gain conclusion of
Theorem~\ref{thm:B3-localization-bias} therefore gives
\[
  \|K_t^{\mathrm{marg}}-K_t\|_{\mathrm{loc}}\le C\varepsilon.
\]
Combining this with \eqref{eq:B7-actual-auxiliary-gain-bound} proves \eqref{eq:B7-localization-gain-bound}.
\end{proof}
\end{proposition}

Recall that $B_t^{N,\mathrm{loc}}:=A(I-K_t^{N,\mathrm{loc}}H)$  in Proposition~\ref{prop:localized-mean-package}. The following lemma establishes that their products are exponentially stable.

\begin{lemma}[Localized closed-loop stability]
\label{lem:B7-localized-closed-loop-stability}
Assume \(\widehat P_0=\widehat P_0^0\).  If \(\varepsilon\) is sufficiently small
and  $C_{\mathrm{init,loc}}$  in \eqref{eq:B6-local-ensemble-size} is sufficiently
large, then there are \(C_m<\infty\) and \(\rho_m\in(0,1)\), independent of
\(J,N,\delta\), such that, on
\(\mathcal G_{N,\delta}^{\mathrm{cov,loc}}\),
\begin{equation}
  \left\|
  B_{t-1}^{N,\mathrm{loc}}\cdots B_s^{N,\mathrm{loc}}
  \right\|_{\mathrm{loc}}
  \le C_m\rho_m^{t-s},
  \qquad t\ge s\ge0.
  \label{eq:B7-localized-closed-loop-product}
\end{equation}
\end{lemma}

\begin{proof}
Recall that $B_t:=A(I-K_tH)$ in \eqref{eq:B_t-def} is the full Kalman closed-loop matrix.  By
Corollary~\ref{cor:B3-full-closed-loop-stability}, its products are uniformly
exponentially stable.  Moreover,
\[
  B_t^{N,\mathrm{loc}}-B_t
  =A(K_t-K_t^{N,\mathrm{loc}})H.
\]
Proposition~\ref{prop:B7-refined-gain-decomposition} gives, 
 on
\(\mathcal G_{N,\delta}^{\mathrm{cov,loc}}\), 
\[
  \sup_{t\ge0}
  \|B_t^{N,\mathrm{loc}}-B_t\|_{\mathrm{loc}}
  \le C\bigl[\Delta_{N,\mathrm{loc}}(\delta)+\varepsilon\bigr].
\]
The smallness and ensemble-size conditions in the statement make this smaller
than the perturbation threshold in Lemma~\ref{lem:B3-robust-products}, which proves
\eqref{eq:B7-localized-closed-loop-product}.
\end{proof}


\subsection{Steady-state covariance and asymptotic spatial decay}
\label{sec:localized-steady-state-decay} 
We first establish existence, uniform convergence, and closed-loop
stability for the uncoupled local algebraic Riccati solutions. We then
use the block-decay algebra and finite-step contraction argument to
prove the spatial-decay and first-order perturbation results for the
coupled steady-state covariance stated in
Section~\ref{subsec:localized-steady-state-main}.

\begin{proposition}[Uniform local algebraic Riccati solutions]
\label{prop:B8-local-fixed-points}
Under Assumption~\ref{ass:uniform-local-riccati}, every local Riccati map
$\Phi_j$ has a unique fixed point
$\widehat P_{\infty,j}^0$ in the common covariance class.  The convergence is
uniform in the block index:
\begin{equation}
  \|\Phi_j^t(P_j)-\widehat P_{\infty,j}^0\|
  \le C_{\mathrm{fp}}\rho_{\mathrm{fp}}^t,
  \qquad
  P_j\in\mathfrak C_j(\underline p_{\mathrm{loc}},
  \overline p_{\mathrm{loc}}),
  \label{eq:B8-local-fixed-point-convergence}
\end{equation}
where $C_{\mathrm{fp}}<\infty$ and $\rho_{\mathrm{fp}}\in(0,1)$ are independent of $j$ and
$J$.  The associated closed-loop matrices
\[
  B_{\infty,j}^0
  :=A_j\bigl(I-\mathsf K_j(\widehat P_{\infty,j}^0)H_j\bigr)
\]
 are uniformly exponentially  
 stable.  
\end{proposition}

\begin{proof}
Fix $j$ and set $P_{t,j}:=\Phi_j^t(P_j)$. The trajectory class in
\eqref{eq:uniform-local-common-trajectory-class} and 
the uniform forgetting
estimate \eqref{eq:uniform-local-riccati-forgetting},
applied to a fixed class containing both it and the initial class,
give
\[
  \|P_{t+s,j}-P_{t,j}\|
  =\|\Phi_j^t(P_{s,j})-\Phi_j^t(P_j)\|
  \le C_\Phi\rho_\Phi^{2t}\|P_{s,j}-P_j\|.
\]
The last factor is uniformly bounded, so $P_{t,j}$ is uniformly Cauchy. Its
limit is a fixed point by continuity; the same estimate gives uniqueness and
\eqref{eq:B8-local-fixed-point-convergence}. Taking the limit in the uniform
product estimate of Theorem~\ref{thm:uniform-block-local-riccati} gives the
closed-loop stability.
\end{proof}

\begin{proof}[Proof of Theorem~\ref{thm:B8-spatially-decaying-fixed-point}]
It suffices to prove the second estimate in
\eqref{eq:B8-full-fixed-point-convergence-decay}, since the first then follows
from $\|M\|_{\mathrm{loc}}\le C\|M\|_{\mathcal J_\alpha}$, after enlarging
$C_\infty$ if necessary. We prove the second estimate using the normalized
block-decay norm $\|\cdot\|_{\alpha,*}$ defined in
\eqref{eq:normalized-block-decay-norm}. Since
$\|M\|_{\alpha,*}=C_\alpha\|M\|_{\mathcal J_\alpha}$, the resulting bound is
equivalent to the desired $\mathcal J_\alpha$ estimate, up to a change in
$C_\infty$. 
We show that $\Phi^{t}$ with a finite, large enough $t$, is contractive in a small tube centered at $\widehat P_\infty^0$, and apply Banach's fixed-point theorem to obtain the fixed point.

Define a fixed-point tube $\mathcal V_{\infty, \alpha, *} := \{ P \in  \mathbb S_+^{d_u}: \| P - \widehat P_\infty^0\|_{\alpha, *} \leq  \vartheta\}$. The analysis in Lemma~\ref{lem:B3-finite-step-perturbation} still holds when shifting the tube center to the stationary reference trajectory $\widehat P_t^0 \equiv \widehat P_\infty^0$. Choose the finite step $T_*$ from
Lemma~\ref{lem:B3-finite-step-perturbation}, whose proof also yields 
\begin{equation}
\sup_{P \in \mathcal V_{\infty, \alpha, *}}\|\mathrm D \Phi_D^{T_*}(P)\|_{\mathrm{op},\alpha, *} \leq \theta_* < 1.
\label{eq:spatially-decaying-fixed-point-proof-ref-diff}
\end{equation}
We verify 
that the full $T_*$-step map remains contractive. Recall from Proposition~\ref{prop:exact-riccati-identities} that 
\[
\mathrm D \Phi^{T_*}(P)[Z] = \mathcal B_{T_*}(P) Z \mathcal B_{T_*}(P)^\top, \qquad \mathcal B_{T_*}(P) = B(\Phi^{{T_*}-1}(P))\cdots B(P).
\]
For $P \in \mathcal V_{\infty,\alpha,*}$, \eqref{eq:B3-gain-regularity} gives 
\[
  \|B(P)-B_D(P)\|_{\alpha, *}
  \le C\|E\|_{\alpha, *}\le C\varepsilon.
\]
The coupled and reference trajectories starting from the same point differ by
$O(\varepsilon)$ during the fixed number $T_*$ of intermediate steps, by
\eqref{eq:B3-finite-step-consistency}. Gain Lipschitzness \eqref{eq:B3-gain-regularity} therefore
shows that the corresponding derivative factors $B(\Phi^t(P))$ and $B_D(\Phi_D^t(P))$ differ by $O(\varepsilon)$. Telescoping the product of the $T_*$ derivative factors gives
\begin{equation}
  \sup_{P \in \mathcal V_{\infty, \alpha, *}}
  \|\mathrm D\Phi^{T_*}(P)-\mathrm D\Phi_D^{T_*}(P)\|_{\mathrm{op}, \alpha, *}
  \le C_{T_*}\varepsilon
\label{eq:spatially-decaying-fixed-point-proof-diff-deviation}
\end{equation}
when the tube radius $\vartheta$ is sufficiently small. Therefore, after decreasing $\varepsilon_\infty$, \eqref{eq:spatially-decaying-fixed-point-proof-ref-diff} and \eqref{eq:spatially-decaying-fixed-point-proof-diff-deviation} imply that there are $\theta_\infty \in (0,1)$ and a radius $\vartheta > 0$ such that
\begin{equation}
  \|\Phi^{T_*}(P)-\Phi^{T_*}(Q)\|_{\alpha, *}
  \le \theta_\infty\|P-Q\|_{\alpha, *}
  \label{eq:B8-full-finite-step-contraction}
\end{equation}
whenever $P, Q \in \mathcal V_{\infty, \alpha, *}$. Furthermore, \eqref{eq:B3-finite-step-consistency} gives 
\begin{equation}
  \|\Phi^{T_*}(\widehat P_\infty^0)-\widehat P_\infty^0
  \|_{\alpha, *}\le C\varepsilon.
\label{eq:spatially-decaying-fixed-point-proof-bias}
\end{equation} 
Choose $\varepsilon_\infty$ so that this defect is at most
$(1-\theta_\infty)\vartheta/2$.  Hence, for every positive-semidefinite $P$ in the
closed radius-$\vartheta$ ball,
\[
  \|\Phi^{T_*}(P)-\widehat P_\infty^0\|_{\alpha,*}
  \le \theta_\infty\vartheta+\frac{1-\theta_\infty}{2}\vartheta<\vartheta.
\]
The intersection of this closed ball with $\Splus^{d_u}$ is complete and
is invariant under $\Phi^{T_*}$.  Banach's fixed-point theorem gives a
unique fixed point $P_*$ of $\Phi^{T_*}$ in $\mathcal V_{\infty, \alpha, *}$. \eqref{eq:B8-full-finite-step-contraction} and \eqref{eq:spatially-decaying-fixed-point-proof-bias} imply
\[
  \|P_*-\widehat P_\infty^0\|_{\alpha, *}
  \le \frac{C}{1-\theta_\infty}\varepsilon.
\]
The matrix $\Phi(P_*)$ is another fixed point of $\Phi^{T_*}$.  It also
lies in the same tube: since $\Phi_D(\widehat P_\infty^0)=\widehat P_\infty^0$,
Lemma~\ref{lem:B3-tube-regularity} gives
\[
  \|\Phi(P_*)-\widehat P_\infty^0\|_{\alpha,*}
  \le C\varepsilon+L_0
  \|P_*-\widehat P_\infty^0\|_{\alpha,*}
  \le C'\varepsilon<\vartheta
\]
after decreasing $\varepsilon_\infty$.  Uniqueness therefore implies
$\Phi(P_*)=P_*$.  Set $\widehat P_\infty:=P_*$.  It is positive
semidefinite because it is the limit of the positive-semidefinite iterates of
$\Phi^{T_*}$ started at $\widehat P_\infty^0$. 
Robustness of the
closed-loop products gives stability of $B_\infty$.

The gain and closed-loop bounds in \eqref{eq:B8-fixed-point-decay-bound}
follow from 
decay-norm gain regularity \eqref{eq:B3-gain-regularity} and from $A=D+E$.  Since the reference
fixed point is block diagonal, the entrywise conclusions
\eqref{eq:B8-fixed-point-entry-decay} and
\eqref{eq:B8-fixed-point-diagonal-bias} follow immediately.

Finally, Proposition~\ref{prop:B8-local-fixed-points} implies that the
reference trajectory enters 
$\mathcal V_{\infty,\alpha,*}$
after a dimension-independent
finite time.  Theorem~\ref{thm:B3-localization-bias} then places the coupled
trajectory in the same tube for small interaction and initial discrepancy.
The contraction \eqref{eq:B8-full-finite-step-contraction}, together with
uniform estimates between grid times, proves
\eqref{eq:B8-full-fixed-point-convergence-decay}.
\end{proof}

\begin{proof}[Proof of Proposition~\ref{prop:B8-first-order-expansion}]
Let $X_\varepsilon:=\widehat P_\infty^\varepsilon-
\widehat P_\infty^0$.  Theorem~\ref{thm:B8-spatially-decaying-fixed-point}
gives $\|X_\varepsilon\|_{\alpha,*}=O(\varepsilon)$.  Taylor
expansion of the fixed-point equation in the block-decay Banach algebra gives
\[
  X_\varepsilon
  =B_\infty^0X_\varepsilon(B_\infty^0)^\top
  +\varepsilon G_\infty^{(1)}+\mathcal R_\varepsilon,
\]
where 
\[
  \|\mathcal R_\varepsilon\|_{\alpha, *}
  \le C\bigl(
  \|X_\varepsilon\|_{\alpha,*}^2
  +\varepsilon\|X_\varepsilon\|_{\alpha,*}
  +\varepsilon^2\bigr)
  \le C\varepsilon^2.
\]
The Lyapunov operator
$\mathcal L_\infty(X)=B_\infty^0X(B_\infty^0)^\top$ has powers satisfying 
$\|\mathcal L_\infty^{\ell}\|_{\mathrm{op},\alpha,*}\le C\rho^{2\ell}$ by \eqref{eq:B3-reference-product-stability}.
Thus
$I-\mathcal L_\infty$ is invertible on the block-decay algebra, with inverse
given by the convergent series in \eqref{eq:B8-first-order-series}.
Applying this inverse proves \eqref{eq:B8-first-order-expansion}.
\end{proof}

\end{document}